\documentclass[a4paper]{amsart}

\usepackage{amssymb,amsmath}
\usepackage{graphicx}
\usepackage[sf,SF]{subfigure}
\usepackage{caption}
\usepackage{amsthm,euscript}
\usepackage{xcolor}
\usepackage[bookmarks=true,hidelinks]{hyperref}
\usepackage{units}
\usepackage{bbm}
\usepackage{enumitem}
\usepackage{marginnote}

\newcommand{\ip}[2]{\langle #1,\, #2\rangle}	

\newtheorem{theorem}{Theorem}
\newtheorem{proposition}[theorem]{Proposition}

\newtheorem{remark}[theorem]{Remark}
\newtheorem{definition}[theorem]{Definition}
\newtheorem{lemma}[theorem]{Lemma}

\numberwithin{equation}{section}
\numberwithin{theorem}{section}

\newcommand{\Prob}{\mathcal{P}}

\newcommand{\recon}{\mathcal{R}}
\newcommand{\proj}{\mathcal{A}}
\newcommand{\W}{\mathcal{W}}
\newcommand{\R}{\mathbb{R}}
\newcommand{\N}{\mathbb{N}}
\newcommand{\Z}{\mathbb{Z}}
\newcommand{\Dx}{{\Delta x}}
\newcommand{\Dt}{{\Delta t}}
\newcommand{\Dy}{{\Delta y}}
\newcommand{\loc}{{\mathrm{loc}}}
\renewcommand{\leq}{\leqslant}
\renewcommand{\geq}{\geqslant}
\renewcommand{\phi}{\varphi}
\newcommand{\Lip}{\mathrm{Lip}}

\newcommand{\wto}{\rightharpoonup}

\newcommand{\hf}{{\unitfrac{1}{2}}}
\newcommand{\thf}{{\unitfrac{3}{2}}}
\newcommand{\iphf}{{i+\hf}}
\newcommand{\ipthf}{{i+\thf}}
\newcommand{\imhf}{{i-\hf}}
\newcommand{\jphf}{{j+\hf}}
\newcommand{\kphf}{{k+\hf}}

\newcommand{\lphf}{{l+\hf}}

\newcommand{\jmhf}{{j-\hf}}
\newcommand{\cell}{{\mathcal{C}}}
\renewcommand{\epsilon}{\varepsilon}
\newcommand{\bx}{\mathbf{x}}

\DeclareMathOperator{\sgn}{sgn}

\allowdisplaybreaks

\title{A second-order numerical method for the aggregation equations}
\author{Jos\'e A. Carrillo}
\address{Mathematical Institute, University of Oxford, Oxford OX2 6GG, UK}
\email{carrillo@maths.ox.ac.uk} 
\author{Ulrik S. Fjordholm}
\address{Department of Mathematics, University of Oslo, 0851 Oslo, Norway}
\email{ulriksf@math.uio.no}
\author{Susanne Solem}
\address{Department of Mathematical Sciences, Norwegian University of Science and Technology, 7491 Trondheim, Norway}
\email{susanne.solem@ntnu.no}
\date{}

\subjclass[2010]{35R09, 35D30, 35Q92, 65M12, 65M08}
\keywords{Aggregation equations, numerical methods, weak measure solutions, measure reconstruction. }

\begin{document}

\begin{abstract}
\noindent Inspired by so-called TVD limiter-based second-order schemes for hyperbolic conservation laws, we develop a formally second-order accurate numerical method for multi-dimensional aggregation equations. The method allows for simulations to be continued after the first blow-up time of the solution. In the case of symmetric, $\lambda$-convex potentials with a possible Lipschitz singularity at the origin we prove that the method converges in the Monge--Kantorovich distance towards the unique gradient flow solution. Several numerical experiments are presented to validate the second-order convergence rate and to explore the performance of the scheme.
\end{abstract}

\maketitle


\section{Introduction}
\noindent In this paper we derive and analyze a formally second-order accurate numerical method for the aggregation equation
\begin{equation}\label{eq:agg_eq}
\partial_t\rho = \nabla\cdot\bigl(\big(\nabla W * \rho\big)\rho  \bigr), \qquad \rho(0)=\rho^0
\end{equation}
where $\rho=\rho(t) \in \Prob(\R^d)$ is a time-parametrized probability measure on $\R^d$ and $\rho^0\in\Prob(\R^d)$ is given. The interaction potential $W : \R^d\to\R$ is assumed to satisfy some or all of the following conditions:
\begin{enumerate}[label=\textbf{(A\arabic*)}]
\item \label{cond:lip} $W$ is Lipschitz continuous, $W(\bx)=W(-\bx)$ and $W(0) = 0$.
\item \label{cond:c1} $W \in C^1\bigl(\R^d\setminus \{0\}\bigr)$.
\item \label{cond:convex} $W$ is $\lambda$-convex for some $\lambda \geq 0$, i.e $W(x)+\frac{\lambda}{2}|x|^2$ is convex.
\end{enumerate}
Potentials satisfying \ref{cond:lip}--\ref{cond:convex} with a Lipschitz singularity at the origin are the so-called \textit{pointy potentials}. When $W$ is a pointy potential, weak solutions to \eqref{eq:agg_eq} might concentrate into Dirac measures in finite time. The finite time blow-up of solutions has attracted a lot of attention, see \cite{LT,BV,BCL2009,BLR2010,HB2010}, wherein almost sharp conditions were given for finite time blow-up and typical blow-up profiles were studied. This finite time blow-up phenomenon explains the necessity of considering measure valued solutions of \eqref{eq:agg_eq}. By utilizing the gradient flow structure of \eqref{eq:agg_eq}, Carrillo et al.~\cite{CDFLS2011} proved existence and uniqueness of solutions to \eqref{eq:agg_eq} when $W$ satisfies \ref{cond:lip}--\ref{cond:convex}. 

Aggregation equations of the form \eqref{eq:agg_eq}, being the continuum limits of particle systems described by
\begin{align} \label{eq:particlesys}
\dot{\bx}_i = -\sum_{i \neq j} m_j \nabla W (\bx_i-\bx_j), \quad \sum_i m_i = 1, \quad m_i > 0,
\end{align}
where $\bx_i$ are the particle positions and $m_i$ the weights,
are ubiquitous in modelling concentration in applied mathematics. They find applications in physical and biological sciences, to name a few: granular materials \cite{BCP,LT,CV2002,CMV2003}, particle assembly \cite{HP2006}, swarming \cite{okubo,TB2004,MCO2005,TBL2006,review2013}, bacterial chemotaxis \cite{KeSe70,FiLauPe04,JV2013}, and opinion dynamics \cite{MT14}. Furthermore, attraction-repulsion potentials have recently been proposed as very simple models of pattern formation due to the rich structure of the set of stationary solutions, see \cite{PSTV,BCLR2011,BU,BUKB,BCLR2013,KSUB,BKSUV} for instance.

The numerical method proposed in this paper is motivated by the fact that the Burgers-type equation
\begin{equation}\label{eq:burgers}
\partial_tu + \partial_x f(u)=0, \qquad f(u) = \pm (u-u^2)
\end{equation}
and the one-dimensional aggregation equation
\begin{equation}\label{eq:agg_eq_1d}
\partial_t\rho = \partial_x\bigl(\big(W' * \rho\big)\rho \bigr)
\end{equation}
with $W(x) = \pm |x|$ are equivalent, see \cite{BV,BCFP2015}. Indeed, defining the primitive $u(x,t) = \int^x_{-\infty}\rho(dy,t)$, we see that
\begin{equation*}
W' * \rho = \pm \sgn *\, \rho = \pm(2u-1).
\end{equation*}
Integrating \eqref{eq:agg_eq_1d} over $(-\infty,x]$ therefore gives \eqref{eq:burgers}. Thus, formally speaking, differentiating \eqref{eq:burgers} in $x$ yields \eqref{eq:agg_eq_1d}. This intuition was made rigorous in \cite{BCFP2015} in which entropy solutions to \eqref{eq:burgers} for nondecreasing initial data are shown to be equivalent to gradient flow solutions to \eqref{eq:agg_eq_1d} for measure valued initial data.

Our starting point is a formally second-order accurate finite volume method for solutions of Burgers' equation \eqref{eq:burgers}. By ``differentiating the method'' in $x$ we obtain a numerical method for \eqref{eq:agg_eq_1d} with $W(x)=\pm|x|$. This method is then extended to the class of potentials $W$ satisfying \ref{cond:lip}, \ref{cond:c1} and any dimension $d$. The order of accuracy of the method is preserved when measured in the right metric, namely the Monge--Kantorovich distance $d_1$. Indeed, the Monge--Kantorovich distance $d_1$ at the level of \eqref{eq:agg_eq_1d} corresponds to the $L^1$ norm at the level of \eqref{eq:burgers} in one dimension due to the relation
\begin{equation*}
d_1(\mu,\nu)= \sup_{\|\phi\|_\Lip\leq 1} \int_{\R} \phi(x)\ d(\mu-\nu)(x) =  \int_\R\big|(\mu-\nu)((-\infty,x])\big|\,dx = \|u-v\|_{L^1(\R)}.
\end{equation*}
The second-order accuracy of the numerical method for Burgers' equation is obtained by reconstructing the numerical approximation into a piecewise linear function in every timestep (see e.g.\ \cite{GR91,LeVeque}). A reconstruction also takes place in the proposed scheme, but the result of the procedure is a \emph{reconstructed measure}. This measure consists of a combination of constant values in the grid cells and Dirac deltas at the grid points. This mixed reconstruction, Diracs plus piecewise constants, is the main difference between our method compared to other methods (of lower order) to solve the aggregation equation with measure valued initial data \cite{JamesVauch2015,CJLV16,DLV2016,DLV2017}. Other numerical schemes based on finite volumes \cite{CCH15} or optimal transport strategies \cite{GT06,CM09,CRW16} have been proposed.

Above, and throughout the paper, we use the terms `formally second-order' and `second-order' in the sense of having a local truncation error of order $O(\Dt\Dx^2)$. This nomenclature is standard in the literature on numerical methods for hyperbolic conservation laws \cite{GR91,KNR95,LeVeque}. Such truncation error estimates rely on Taylor expansions of the exact solution and hence requires the existence of a smooth solution. There are very few rigorous convergence rate results available for such methods for general, non-smooth solutions (beyond the suboptimal $O(\Dx^{\hf})$ estimate due to Kuznetsov \cite{Kuz76}). We would expect that a rigorous convergence rate estimate for the methods presented here (beyond our local truncation estimate) would require a substantial amount of work, and only apply in a limited number of scenarios. We refer to \cite{DLV2017} for a proof of an $O(\Dx^\hf)$ convergence rate for a numerical method for \eqref{eq:agg_eq}.

We derive the method for \eqref{eq:agg_eq_1d} with $W(x)=\pm |x|$ before generalizing it in one dimension to any potential satisfying \ref{cond:lip}, \ref{cond:c1} in Section 3. We study its properties, and show the convergence of the scheme for measure valued solutions in the distance $d_1$ in the main theorem. The scheme and the main theorem is generalized to any dimension in Section 4. Section 5 is devoted to validating the scheme in known particular cases together with accuracy tests and numerical explorations for both potentials covered by the theory and attractive-repulsive potentials not covered. Section 2 deals with the necessary preliminaries about gradient flow solutions to the aggregation equation \eqref{eq:agg_eq}.


\section{Preliminaries on gradient flow solutions}

We define the space of probability measures with finite $p$-th order moment, $1\leq p <\infty$ as
    $$
\Prob_p(\R^d) = \left\{\mu \text{ nonnegative Borel measure}, \mu(\R^d)=1,\ \int_{\R^d} |x|^p \mu(dx) <\infty\right\}.
    $$
This space is endowed with the optimal transport distance $d_p$ defined by
    \begin{equation*}
d_p(\mu,\nu)= \inf_{\gamma\in \Gamma(\mu,\nu)} \left(\int |y-x|^p\,\gamma(dx,dy)\right)^{1/p}
    \end{equation*}
where $\Gamma(\mu,\nu)$ is the set of measures on $\R^d\times\R^d$ with marginals $\mu$ and $\nu$ (see e.g.~\cite{Villani,AGS2005}).
The particular cases that will be useful in our present work are the Euclidean Wasserstein distance $d_2$ and the Monge--Kantorovich distance $d_1$.  Let 
\begin{equation}\label{eq:int_energ}
\W (\rho) = \frac{1}{2} \int_{\R^d \times \R^d} W(x-y) \rho (dx)\rho (dy)
\end{equation}
be the total potential energy associated to the aggregation equation \eqref{eq:agg_eq}. It is by now classical that the aggregation equation \eqref{eq:agg_eq} can be written as
$$
\frac{\partial \rho}{\partial t} = \nabla\cdot\left(\rho\nabla \frac{\delta \W}{\delta \rho}\right)\,,
$$
with $\frac{\delta \W}{\delta \rho}=W\ast\rho$ the variational derivative of the functional $\W$. This is the formal signature of the $d_2$-gradient flow structure of evolutions equations \cite{AGS2005,Villani,CMV2003,CMV2006}.  

We say that $\mu\in AC^\hf_{\loc}([0,+\infty);\Prob_2(\R^d))$ if $\mu$ is locally H\"older continuous of exponent $\hf$ in time with respect to the distance $d_2$ in $\Prob_2(\R^d)$.
A gradient flow solution associated to \eqref{eq:int_energ} is defined as follows, see \cite{AGS2005, CDFLS2011}.

\begin{definition}[Gradient flow solutions]\label{def:gradientflow}
Let $W$ satisfy the assumptions \ref{cond:lip}--\ref{cond:convex}. We say that a map $\rho \in AC^{\hf}_\loc \bigl([0, +\infty);\Prob_2(\R^d)\bigr)$ is a gradient flow solution of \eqref{eq:agg_eq} associated with the functional \eqref{eq:int_energ}, if there exists a Borel vector field $v$ such that $v(t) \in \textrm{Tan}_{\rho(t)}\Prob_2(\R^d)$ for
a.e. $t > 0$, i.e. $\|v(t)\|_{L^2(\rho)} \in L^2_{\loc}(0, +\infty)$, the continuity equation
\begin{equation}\label{eq:gradientfloweq}
\partial_t \rho + \nabla \cdot(v \rho) = 0,
\end{equation}
holds in the sense of distributions, and $v(t) = -\partial^0\W(\rho(t))$ for a.e. $t > 0$. Here, $\partial^0 \W(\rho)$ denotes the element of minimal norm in $\partial \W (\rho)$, the subdifferential of $\W$ at the point $\rho$.
\end{definition}
In \cite{CDFLS2011} it is shown that when $W$ satisfies \ref{cond:lip}--\ref{cond:convex} we have $\partial^0\W = \partial^0W * \rho$, where $\partial^0W(x) = \nabla W(x)$ for $x\neq0$ and $\partial^0W(0)=0$. Hence,
\begin{equation}\label{eq:min_nrm_velocity}
\partial^0\W(x,t) = \int_{x \neq y} \nabla W (x-y)\rho(dy,t)
\end{equation}
is the unique element of minimal norm when $W$ satisfies \ref{cond:lip}--\ref{cond:convex}.

\begin{theorem}[Well-posedness of gradient flow solutions \cite{CDFLS2011}]\label{thrm:unique}
Let $W$ satisfy assumptions \ref{cond:lip}--\ref{cond:convex}. Given $\rho^0 \in \Prob_2 (\R^d)$ there exists a unique gradient flow solution of \eqref{eq:agg_eq}, i.e.\ a curve $\rho \in AC^{\hf}_\loc ([0, +\infty);\Prob_2(\R^d))$ satisfying \eqref{eq:gradientfloweq}
in $\mathcal{D}'([0,\infty) \times \R^d)$ with 
$v(x,t) = -\partial^0 W \ast \rho$ and $\rho(0) = \rho^0$.
\end{theorem}

Let us connect this notion of solution to more classical concepts of weak solutions for PDEs.

\begin{definition}\label{def:weaksol}
A locally in time absolutely continuous in $d_p$ curve $\rho:[0,+\infty)\to \Prob_p(\R^d)$, $1\leq p<\infty$ is said to be a \emph{$d_p$-weak measure solution}
to \eqref{eq:agg_eq} with initial datum $\rho^0\in\Prob_p(\R^d)$ if and
only if $\partial^0 W\ast \rho \in L^1_{\loc}((0,+\infty);L^2(\rho(t)))$
and
\begin{equation}\label{eq:dpsol}
\begin{split}
\int_0^{+\infty}\!\!\!  \int_{\R^d}\frac{\partial\varphi}{\partial
t}(x,t) &\,\rho (dx,t)\,dt + \int_{\R^d} \varphi(x,0) \,\rho^0 (dx)
\\
&=\int_0^{+\infty}\!\!\!  \int_{\R^d} \int_{\R^d}
\nabla\varphi(x,t)\cdot\partial^0 W(x-y) \, \rho(dy,t)\,
\rho(dx,t)\,dt, 
\end{split}
\end{equation}
for all test functions $\varphi\in C^\infty_c ([0,+\infty)\times
\R^d)$.
\end{definition}

In \cite{CDFLS2011} it is proven that the concept of gradient flow solutions to \eqref{eq:agg_eq} under the assumptions \ref{cond:lip}--\ref{cond:convex} is in fact equivalent to the concept of $d_2$-weak measure solutions. As a consequence, the uniqueness of gradient flow solutions imply the uniqueness of $d_2$-weak measure solutions, see \cite[Section 2.3]{CDFLS2011}.

Observe that $d_1$-weak measure solutions to \eqref{eq:agg_eq} are also $d_2$-weak measure solutions to \eqref{eq:agg_eq}. Indeed, since $\|v(t)\|_{L^2(\rho)} \in L^2_{\loc}((0, +\infty))$, we can apply \cite[Theorem 8.3.1]{AGS2005} which implies the absolute continuity with respect to $d_2$ of the curve of probability measures $\rho(t)$. This fact will be the key to identifying the limit of the numerical schemes below.

The notion of gradient flow solutions has been proven to be equivalent to the notion of duality solutions in one dimension \cite{JV2016}, the Fillipov flow solutions \cite{CJLV16}, and, as mentioned in the introduction, it is equivalent to the notion of entropy solutions of the one-dimensional Burgers' equation in the particular case of $W(x)=\pm|x|$, see \cite{BCFP2015}. 

Let us finally mention that global existence of measure valued solutions in one dimension to \eqref{eq:agg_eq} with $\partial_x W \ast \rho $ replaced by $a(\partial_x W \ast \rho)$, where $a$ is a $C^1$ function, was obtained by James and Vauchelet \cite{JV2016} using the notion of duality solutions, introduced by Bouchut and James \cite{BoJa1998}.


\section{A numerical scheme for the 1D aggregation equation}
\noindent Based on the relation 
\begin{equation}\label{eq:urhorelation}
u(x,t)=\rho((-\infty,x],t)
\end{equation}
between solutions to the one-dimensional aggregation equation \eqref{eq:agg_eq_1d} and Burgers' equation \eqref{eq:burgers} in the case $W(x)=\pm |x|$, we will derive a (formally) second-order accurate method for the aggregation model \eqref{eq:agg_eq_1d}. As $\rho$ is assumed to be a probability measure in $x$, we can assume that $u$ will be a nondecreasing function satisfying $u(-\infty,t) = 0$ and $u(+\infty,t)=1$. We will later generalize the resulting method to general potentials and multiple dimensions. 

\subsection{Second-order schemes for Burgers' equation}
We discretize the space-time domain $\R\times\R_+$ as $x_\imhf = (\imhf)\Dx$ and $t^n = n\Dt$ for $i\in\Z$ and $n\in\N_0$, where $\Dx, \Dt>0$ are the discretization parameters. We define also the computational cell $\cell_i:=[x_\imhf,x_\iphf)$. A finite volume approximation of \eqref{eq:burgers} aims to approximate the cell averages 
$$
u_i^n \approx \frac{1}{\Dx}\int_{\cell_i} u(x,t^n)\,dx\,.
$$
Such schemes are generally first-order accurate, and a popular method of increasing the order of accuracy is by \textit{reconstruction}: Given cell averages $u_i^n$, compute a piecewise linear polynomial
\begin{equation*}
\recon u_\Dx(x,t^n) = u_i^n + \sigma_i^n(x-x_i), \qquad x \in \cell_i
\end{equation*}
(see e.g.\ \cite{GR91,LeVeque}). The slopes $\sigma_i^n\in\R$ are selected using e.g.\ the minmod limiter, which for increasing data $u_i^n\leq u_{i+1}^n$ is given by $\sigma_i^n = \frac{1}{\Dx}\min \big(u_i^n-u_{i-1}^n, u_{i+1}^n-u_i^n\big)$. Defining the edge values $u_\iphf^{n,\pm} = \recon u_\Dx(x_\iphf \pm 0,t^n)$, a (formally) second-order accurate finite volume method for \eqref{eq:burgers} is given by
\begin{equation}\label{eq:num_burgers}
\begin{split}
u_i^{n+1} &= u_i^n - \beta \left(F\big(u_\iphf^{n,-},u_\iphf^{n,+}\big) - F\big(u_\imhf^{n,-},u_\imhf^{n,+}\big)\right), \qquad \beta := \frac{\Dt}{\Dx} \\
u_i^0 &= \frac{1}{\Dx}\int_{\cell_i}u^0(x)\,dx.
\end{split}
\end{equation}
Here, $F$ is any monotone numerical flux function, such as the Lax--Friedrichs-type flux
\begin{equation*}
F(u,v) = \frac{f(u) + f(v)}{2} - \frac{c}{2}(v-u), \qquad c=\max_i\big(|f'(u_i^n)|\big).
\end{equation*}
This numerical flux is chosen here for its simplicity, and is a Lax--Friedrichs-type flux where the usual constant $1/\beta$ is replaced by the maximum velocity $c$.

\subsection{Second-order schemes for the aggregation model}
In this section we transfer the above approach to the one-dimensional aggregation equation \eqref{eq:agg_eq_1d}, first for the Newtonian potential $W(x)=\pm|x|$ in Section \ref{sec:newtonpot} and then to more general potentials in Section \ref{sec:generalpot}.

\subsubsection{Newtonian potential}\label{sec:newtonpot}
Analogous to the relation \eqref{eq:urhorelation}, we define $\rho$ through the relation
\[
\rho_\iphf^n = \frac{u_{i+1}^n-u_i^n}{\Dx} \qquad\Leftrightarrow\qquad u_i^n = \sum_{j \leq i} \Dx \rho_\jmhf^n, \quad \rho_{-\infty} = 0.
\]
As a simplifying assumption, let us assume that the initial data for the conservation law \eqref{eq:burgers} has been sampled through point values, $u_i^0=u^0(x_i)$. For the initial data $\rho_\iphf^0$, this relation and the definition $u^0=\rho^0((-\infty,x])$ yield
\begin{align*}
\rho_{i+\hf}^0 &= \frac{1}{\Dx}\left(u^0(x_{i+1})-u^0(x_i)\right) = \frac{1}{\Dx} \rho^0((x_i,x_{i+1}]).
\end{align*}

Taking the difference in $i$ of \eqref{eq:num_burgers} yields the following numerical method for \eqref{eq:agg_eq_1d} in the case $W(x)=\pm|x|$:
\begin{equation}\label{eq:num_rho_burgers}
\begin{split}
\rho_\iphf^{n+1} = \rho_\iphf^n + \frac{\beta}{2}\Bigg[&\Dx \sum_{j\neq i} \pm\sgn\big(x_i-x_j\big)\Big(\rho_{j+1}^{n,+}\rho_{i+1}^{n,+} + \rho_{j+1}^{n,-}\rho_{i+1}^{n,-} \\
& - \rho_j^{n,+}\rho_i^{n,+} - \rho_j^{n,-}\rho_i^{n,-}\Big) + c^n\big(\rho_{i+1}^{n,+} - \rho_{i+1}^{n,-}\big) - c^n\big(\rho_i^{n,+} - \rho_i^{n,-}\big)\Bigg],
\end{split}
\end{equation}
where 
$$
c^n = \Dx \max_i\Big(\Big| \sum_{j\neq i} \pm \sgn (x_i-x_j)\rho_j^{n,-}\Big|, \Big|\sum_{j\neq i} \pm \sgn (x_i-x_j)\rho_j^{n,+}\Big|\Big)
$$
and
\begin{equation}\label{eq:edge_val}
\rho_i^{n,+} = \rho_\imhf^n - \frac{1}{2}\big(\sigma_{i+1}^n-\sigma_i^n\big), \qquad \rho_{i+1}^{n,-} = \rho_\iphf^n + \frac{1}{2}\big(\sigma_{i+1}^n-\sigma_i^n\big),
\end{equation}
with initial data given by
\begin{equation}\label{eq:num_init}
\rho_\iphf^0 = \frac{1}{\Dx}\rho^0((x_i,x_{i+1}]).
\end{equation}
Observe that
\[
\rho_i^{n,\pm} = \frac{u_\iphf^{n,\pm}-u_\imhf^{n,\pm}}{\Dx} \,.
\]
Now, notice that when $W(x)=\pm |x|$ we have $W'(x)=\pm \sgn(x)$ for all $x\neq0$, so that $\pm \sgn $ can be replaced by $W'$. Thus, the method \eqref{eq:num_rho_burgers} can be written as follows,
\begin{equation*}
\begin{split}
\rho_\iphf^{n+1} = \rho_\iphf^n + \frac{\beta}{2}\Big[&a_{i+1}^{n,+}\rho_{i+1}^{n,+} + a_{i+1}^{n,-}\rho_{i+1}^{n,-} - a_i^{n,+}\rho_i^{n,+} - a_i^{n,-}\rho_i^{n,-} \\
& \qquad \qquad + c^n\big(\rho_{i+1}^{n,+} - \rho_{i+1}^{n,-}\big) - c^n\big(\rho_i^{n,+} - \rho_i^{n,-}\big)\Big],
\end{split}
\end{equation*}
where
\begin{equation*}
\begin{split}
a_i^{n,+} = \Dx \sum_{j\neq i} W'\bigl(x_i-x_j\bigr)\rho_j^{n,+}, \qquad
a_i^{n,-} = \Dx \sum_{j\neq i} W'\bigl(x_i-x_j\bigr)\rho_j^{n,-}.
\end{split}
\end{equation*}

\begin{figure}
\centering
\subfigure[Cell averages of $u_\Dx$.]{
\includegraphics[trim={1.3cm 0 -1.3cm 0},clip, width=0.42\textwidth]{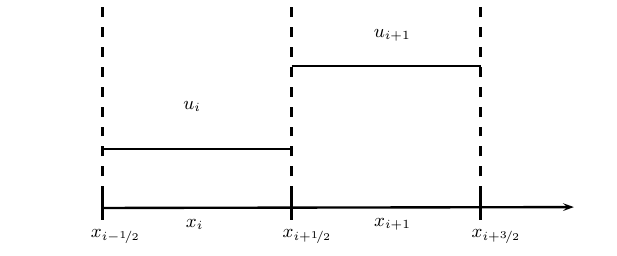}}
\subfigure[The reconstruction $\recon u_\Dx$.]{
\includegraphics[width=0.41\textwidth]{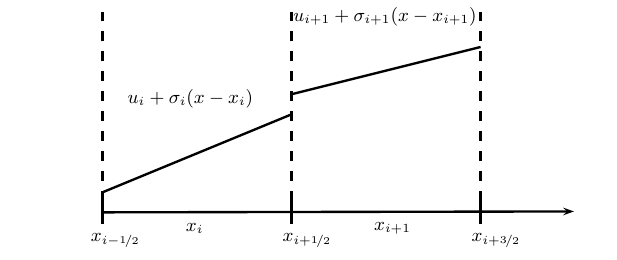}}
\subfigure[Point masses of $\rho_\Dx$.]{
\includegraphics[width=0.47\textwidth]{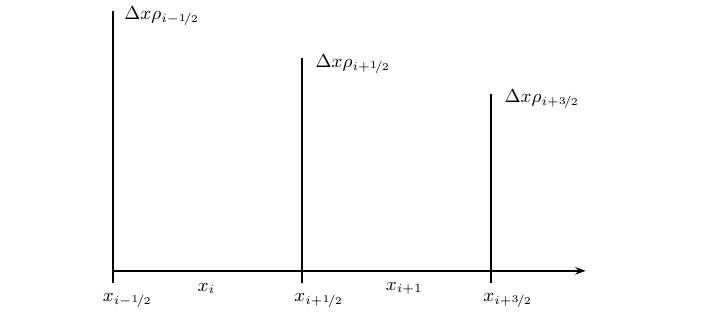}}
\subfigure[The reconstructed measure $r$ of $\rho_\Dx$, where $r([x_i,x_{i+1}))=\Dx\rho_\iphf$.]{\includegraphics[width=0.49\textwidth]{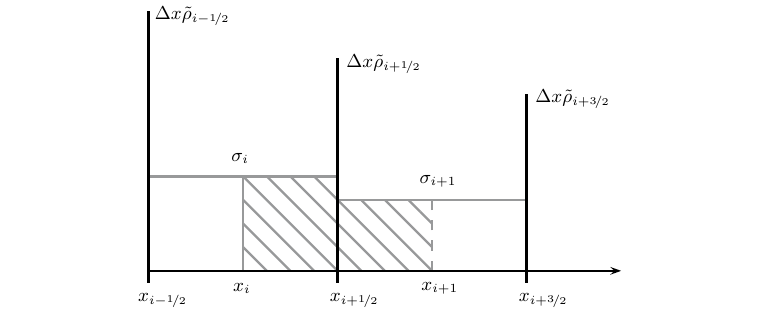}}
\caption{The reconstruction of $u_\Dx$ (top) translated into a reconstruction procedure for $\rho_\Dx$ (bottom). The solid, vertical lines represent Dirac measures centered at the midpoints $x_\iphf$.}
\label{fig:reconstruction}
\end{figure}  

\subsubsection{General potentials}\label{sec:generalpot}
Let $\rho_\iphf^0$ be as in \eqref{eq:num_init}. We realize the numerical approximation $\rho_\iphf^n$ as the measure
\begin{equation*}
\rho_\Dx (x,t^n) = \Dx\sum_i \rho_\iphf^n \delta_{x_\iphf}.
\end{equation*}
This function is {reconstructed} by defining a \emph{reconstructed measure} $r^n$ as
\begin{equation}\label{eq:rho_recon}
r^n = \sum_i \left[\Dx \tilde{\rho}_\iphf^n \delta_{x_\iphf} +  \sigma_i^n\mathcal{L}\big|_{\cell_i}\right], \qquad \tilde{\rho}_\iphf^n := \rho_\iphf^n -\frac{1}{2}(\sigma_i^n+\sigma_{i+1}^n)
\end{equation}
where $\sigma_i^n = \min \big(\rho_\imhf^n, \rho_\iphf^n\big)$ and $\mathcal{L}\big|_{A}$ denotes the Lebesgue measure restricted to the set $A$ (cf. Figure \ref{fig:reconstruction}). It is easy to check that the reconstruction preserves mass, in the sense
\[
r^n\big([x_i,x_{i+1})\big)=\Dx\rho_\iphf^n,
\]
and that $r^n$ is nonnegative. It follows that $r^n \in \Prob_1(\R)$ whenever $\rho_\Dx(t^n) \in \Prob_1(\R)$. Moreover, as the reconstruction procedure redistributes mass over a distance no greater than $\Dx$, we have
\begin{equation}\label{eq:reconstrcontinuous}
d_1\big(\rho_\Dx(t^n), r^n\big) \leq \Dx.
\end{equation}
We can now \textit{define} $\rho_i^{n,\pm}$ as taking information from $r^n$ in the downwind or upwind direction,
\begin{equation}\label{eq:edge_val2}
\begin{split}
\rho_i^{n,+} &= \frac{1}{\Dx}r^n\big((x_\imhf,x_\iphf]\big) = \rho_\iphf^n - \frac{1}{2}(\sigma_{i+1}^n-\sigma_i^n),\\
\rho_i^{n,-} &= \frac{1}{\Dx}r^n\big([x_\imhf,x_\iphf)\big) = \rho_\imhf^n + \frac{1}{2}(\sigma_i^n-\sigma_{i-1}^n)
\end{split}
\end{equation}
(compare with \eqref{eq:edge_val}). Furthermore, if we set
\begin{equation*}
\rho_\Dx^{n,+} = \sum_i \Dx \rho_i^{n,+}\delta_{x_i}, \qquad \rho_\Dx^{n,-} = \sum_i \Dx \rho_i^{n,-}\delta_{x_i}
\end{equation*}
then we find that
\begin{equation*}
\Dx \sum_{j\neq i} \pm \sgn \bigl(x_i-x_j\bigr)\rho_j^{n,\pm} = \partial_x^0 W \ast \rho_\Dx^{n,\pm}(x_i)
\end{equation*}
when $W(x)=\pm|x|$. We use the above expression to define the numerical velocities for general potentials $W$ as follows,
\begin{equation}\label{eq:velocities}
\begin{split}
a_i^{n,+} &= \partial_x^0 W \ast \rho_\Dx^{n,+}(x_i) = \Dx \sum_{j\neq i} W'\bigl(x_i-x_j\bigr)\rho_j^{n,+},\\
a_i^{n,-} &= \partial_x^0 W \ast \rho_\Dx^{n,-}(x_i) = \Dx \sum_{j\neq i} W'\bigl(x_i-x_j\bigr)\rho_j^{n,-}.
\end{split}
\end{equation}
Moreover, we can replace $W'$ in \eqref{eq:velocities} by a continuous and piecewise linear approximation $W_\Dx'$ satisfying $W_\Dx'(k\Dx)=W'(k\Dx)$ for all $k \neq 0$. This will be used in the upcoming convergence proof.

Summing up, with the reconstruction \eqref{eq:rho_recon} and the velocities \eqref{eq:velocities} we define the (formally) second-order accurate numerical scheme
\begin{subequations}\label{eq:numericalmethod_1d}
\begin{equation}\label{eq:scheme_1d}
\rho_\iphf^{n+1} = \rho_\iphf^n + \frac{\Dt}{\Dx}\big(J_{i+1}^n - J_i^n\big),
\end{equation}
where the fluxes are given by the Lax--Friedrichs-type flux formula
\begin{equation}\label{eq:diffusiondef}
J_i^n = \frac{a_i^{n,+}\rho_i^{n,+} + a_i^{n,-}\rho_i^{n,-}}{2} + \frac{c^n}{2}\big(\rho_i^{n,+} - \rho_i^{n,-}\big), \qquad c^n = \max_i |a_i^{n,\pm}|
\end{equation}
and $\rho_\iphf^0$ is defined in \eqref{eq:num_init}. We emphasize that this numerical scheme is well-defined for any potential $W$ satisfying assumptions \ref{cond:lip}, \ref{cond:c1}.

Notice that if we replace $J_i^n$ in \eqref{eq:diffusiondef} with the upwind flux 
\begin{equation}\label{eq:upwinddef}
J_i^n=\max\big(a_i^{n,+},0\big)\rho_i^{n,+} + \min\big(a_i^{n,-},0\big)\rho_i^{n,-}\,,
\end{equation}
\end{subequations}
then the scheme \eqref{eq:scheme_1d}, \eqref{eq:upwinddef} also defines a (formally) second-order accurate numerical scheme, and the upwinding flux formula \eqref{eq:upwinddef} is a perfect valid alternative to the Lax--Friedrichs-type flux \eqref{eq:diffusiondef}. 

From now on, we will refer to the numerical scheme \eqref{eq:numericalmethod_1d} meaning that we discuss either the numerical scheme \eqref{eq:scheme_1d}, \eqref{eq:diffusiondef} or \eqref{eq:scheme_1d}, \eqref{eq:upwinddef} indistinctively. We will only provide the proofs in the case of the Lax--Friedrichs-type flux \eqref{eq:diffusiondef}, but we emphasize that the upwind scheme shares the same stability and convergence properties as the Lax--Friedrichs method.

\begin{remark}
The numerical scheme \eqref{eq:numericalmethod_1d} is only (formally) first-order accurate in time. A higher-order integration in time, such as Heun's method or another Runge--Kutta method, is needed to make the scheme second-order in both time and space. See e.g.~\cite[Section 19.4]{LeVeque} for more details.
\end{remark}

\subsection{Properties of the scheme}\label{sec:schemeprop1d}
The properties of the scheme \eqref{eq:numericalmethod_1d} are similar to those of the first-order accurate schemes for \eqref{eq:agg_eq_1d} developed by James and Vauchelet \cite{JamesVauch2015, JamesVauch2016}. Define the linear time interpolation
\begin{equation}\label{eq:num_approx}
\rho_\Dx (t) := \frac{t^{n+1}-t}{\Dt}\rho_\Dx(t^n) + \frac{t-t^n}{\Dt}\rho_\Dx(t^{n+1}), \qquad t\in[t^n,t^{n+1})
\end{equation}
where $\rho_\Dx(t^n) = \Dx\sum_i \rho_\iphf^n\delta_{x_\iphf}$ and $\rho_\iphf^n$ is computed with the numerical scheme \eqref{eq:numericalmethod_1d}.

\begin{lemma}\label{lem:cfl}
Assume that $\rho^0 \in \Prob_1(\R)$ and that $W$ satsfies \ref{cond:lip}--\ref{cond:c1}. Assume moreover that $\beta := \Dt/\Dx$ satisfies the \emph{CFL condition}
\begin{equation}\label{eq:cfl}
\beta \leq \frac{1}{2\|W\|_{\Lip}}.
\end{equation}
Then for all $t\geq0$ and $n\in\N_0$:
\begin{itemize}
\item[(i)] \emph{Positivity/mass preservation:} $\rho_\Dx(t) \geq 0$ and $\int_\R\rho_\Dx(dx,t) = 1$,
\item[(ii)] \emph{Finite speed of propagation:} $c^n := \max_i \bigl|a_i^{n,\pm}\bigr| \leq \|W\|_\Lip$,
\item[(iii)] \emph{Bounded first order moment:}
\begin{equation}\label{eq:finitemoment}
\int_{\R} |x|\,\rho_\Dx(dx,t) \leq \int_{\R} |x|\,\rho^0(dx) + 2t\|W\|_\Lip, 
\end{equation}
\item[(iv)] \emph{Uniform tightness:} Let $r \geq 1$ and $\epsilon > 0$. Then
\begin{equation*}
\int_{\R \setminus [-r,r]} |x|\rho^0(dx) < \epsilon \quad \implies  \quad \int_{\R \setminus [-R,R]} |x|\rho_\Dx(dx, t) < \epsilon C(t),
\end{equation*}
where $R=r+t/\beta$ and $C(t) = \exp \left(\frac{3}{2} \|W\|_\Lip t \right)$.
\item[(v)] \emph{Preservation of the center of mass:}
$$
\int_{\R} x \,\rho_\Dx(dx, t) =
\int_{\R} x \,\rho_\Dx(dx,0)
$$
\item[(vi)] \emph{Time continuity:} The map $t \mapsto \rho_\Dx(t)$ is uniformly Lipschitz, in the sense that
\begin{equation}\label{eq:dtimecont}
d_1\big(\rho_\Dx(t), \rho_\Dx(s)\big) \leq 2\|W\|_{\Lip}|t-s|
\end{equation}
for all $t,s\geq0$, where $d_1$ denotes the Monge--Kantorovich--Rubinstein metric.
\item[(vii)] \emph{Bounded second order moment:} If in addition $\rho^0 \in \Prob_2(\R)$ then $\rho_\Dx(t)$ has bounded second order moment:
\begin{equation*}
\begin{split}
\int_\R |x|^2\,\rho_\Dx(dx,t) \leq \int_\R |x|^2\,\rho^0(dx) + 6t\|W\|_\Lip \int_\R |x|\,\rho^0(dx) + 12 t^2 \|W\|_\Lip^2, 
\end{split}
\end{equation*}
\end{itemize}
\end{lemma}
\begin{proof}
From the definition \eqref{eq:num_approx} of $\rho_\Dx(t)$ it is clear that we only need to check each property at the discrete times $t=t^n$.

\textit{(i) and (ii):} The property \textit{(i)} clearly holds for $n=0$. Assume that \textit{(i)} holds for some $n\in\N_0$. By the definition \eqref{eq:edge_val2} we have $\rho_j^{n,\pm}\geq0$ for all $j$. It then follows that the velocity $a_i^{n,\pm}$ is bounded:
\begin{align*}
\bigl|a_i^{n,\pm}\bigr| &= \bigl|\Dx \sum_{j\neq i} W'\bigl(x_i-x_j\bigr)\rho_j^{n,\pm} \bigr| \leq \|W\|_{\Lip} \Dx \sum_j \rho_j^{n,\pm} \\
&= \|W\|_{\Lip} \Dx \sum_j \Big( \rho_j^{n} \pm \frac{1}{2}(\sigma_{j+1}^n-\sigma_j^n) \Big)\\
&=\|W\|_{\Lip} \Dx \sum_j \rho_j^{n}=\|W\|_{\Lip}.
\end{align*}
Using the fact that $\rho_\iphf^n = \frac{1}{2}(\rho_i^{n,+}+\rho_{i+1}^{n,-})$, the scheme \eqref{eq:scheme_1d} can be rewritten as  
\begin{align*}
\rho_\iphf^{n+1} = \frac{1-\beta\big(c^n-a_{i+1}^{n,-}\big)}{2}\rho_{i+1}^{n,-} & + \frac{1-\beta\big(c^n+a_i^{n,+}\big)}{2}\rho_i^{n,+}\\
& + \frac{\beta}{2}\big(c^n+a_{i+1}^{n,+}\big)\rho_{i+1}^{n,+} + \frac{\beta}{2}\big(c^n-a_i^{n,-}\big)\rho_i^{n,-}.
\end{align*}
By the induction hypothesis, the definition \eqref{eq:diffusiondef} of $c^n$ and the CFL condition \eqref{eq:cfl}, we infer that $\rho_\iphf^{n+1} \geq 0$. Summing the conservative numerical method \eqref{eq:scheme_1d} over all $i\in\Z$ and using the definition \eqref{eq:num_init} of the initial data yields
\begin{align*}
\sum_i \Dx \rho_\iphf^{n+1}= \sum_i \Dx \rho_\iphf^n= \sum_i \Dx \rho_\iphf^0=1.
\end{align*}

\textit{(iii):} Assume that $\rho_\Dx(t^n)$ satisfies \eqref{eq:finitemoment}. From \eqref{eq:scheme_1d} and summation by parts, the first order moment can be written as
\begin{align}
\Dx \sum_i |x_\iphf|\rho_\iphf^{n+1} = &\ \Dx\sum_i |x_\iphf| \rho_\iphf^n \nonumber\\
 - &\Dx \frac{\beta}{2}\sum_i (a_i^{n,+}+c^n)\rho_i^{n,+}\left(|x_\iphf|-|x_\imhf| \right) \\ 
 - &\Dx \frac{\beta}{2}\sum_i (a_i^{n,-}-c^n)\rho_i^{n,-}\left(|x_\iphf|-|x_\imhf| \right) \nonumber\\
+&\lim_{i\to \infty }\Dx\frac{\beta}{2}|x_\iphf|\left(a_i^{n,+}\rho_i^{n,+} + a_i^{n,-}\rho_i^{n,-} + c^n\left( \rho_i^{n,+}+ \rho_i^{n,-}\right) \right) \nonumber\\
 -&\lim_{i\to -\infty }\Dx\frac{\beta}{2}|x_\iphf|\left(a_i^{n,+}\rho_i^{n,+} + a_i^{n,-}\rho_i^{n,-} + c^n\left( \rho_i^{n,+}+ \rho_i^{n,-}\right) \right). \label{aux}
\end{align}
The last two terms vanish because $a_i^{n,\pm}$ satisfies \textit{(ii)} and $\rho_\Dx^{\pm}(t^n) \leq \frac{3}{2} \rho_\Dx(t^n)$, where $\rho_\Dx(t^n) \in \Prob_1(\R)$ by the induction hypothesis. From the bound $\big||x_\iphf|-|x_\imhf|\big| \leq \Dx$, \eqref{eq:cfl}, \textit{(ii)} and the induction hypothesis, we get
\begin{align*}
\Dx \sum_i |x_\iphf|\rho_\iphf^{n+1} &\leq \Dx \sum_i |x_\iphf| \rho_\iphf^n \\
& \quad + \Dx \frac{\beta\Dx}{2}\sum_i \left(\big|c^n+a_i^{n,+}\big|\rho_i^{n,+} + \big|c^n-a_i^{n,-}\big|\rho_i^{n,-} \right) \\
&\leq \Dx \sum_i |x_\iphf| \rho_\iphf^n + \|W\|_\Lip\Dt\Dx\sum_i \left(\rho_i^{n,+} + \rho_i^{n,-} \right) \\
&= \Dx \sum_i |x_\iphf| \rho_\iphf^n + 2\Dt \|W\|_\Lip \\
&\leq \Dx \sum_i |x_\iphf| \rho_\iphf^0 + 2t^{n+1} \|W\|_\Lip.
\end{align*}

\textit{(iv):} We consider $x > R$. The case $ x < -R$ is similar. Let $k \in \Z$ be such that $R\in\cell_k$. Then, from a summation by parts,
\begin{align*}
\quad \, & \int_{x > R} |x| \rho_\Dx(dx,t^{n+1}) \\
&= \Dx\sum_{i\geq k} x_{i+\hf} \rho_{i+\hf}^{n+1} \\
&= \Dx\sum_{i\geq k} x_{i+\hf} \rho_{i+\hf}^{n} - \Dt \sum_{i\geq k} \left(x_\iphf-x_\imhf \right)J_i^n - \Dt x_{k-\hf}J_k^n \\
&= \Dx\sum_{i\geq k} x_{i+\hf} \rho_{i+\hf}^{n} - \Dt\Dx \sum_{i\geq k} J_i^n - \Dt x_{k-\hf}J_k^n \\
&\leq \Dx\sum_{i\geq k} x_{i+\hf} \rho_{i+\hf}^{n} - \Dt \Dx\sum_{i\geq k} \frac{1}{2}\big(a_{i}^{n,-} - c^n\big)\rho_{i}^{n,-} - \Dt x_{k-\hf}\frac{1}{2}\big(a_{k}^{n,-} - c^n\big)\rho_{k}^{n,-} \\
&\leq \Dx\sum_{i\geq k} x_{i+\hf} \rho_{i+\hf}^{n} + \Dt \Dx\frac{3}{2}\|W \|_\Lip \sum_{i\geq k}\rho_\iphf^n + \Dx x_{k-\hf}\frac{3}{4}\rho_{k-\hf}^{n} \\
&\leq \left(1+\frac{3}{2}\|W\|_\Lip \Dt \right) \Dx\sum_{i\geq k-1} x_{i+\hf} \rho_{i+\hf}^{n} \\
&\leq \dots 
\leq \left(1+\frac{3}{2}\|W\|_\Lip \Dt \right)^n \Dx\sum_{i\geq k-n-1} x_{i+\hf} \rho_{i+\hf}^{0}
\end{align*} 
where we have used $J_i^n \geq \frac{1}{2}(a_i^{n,-} - c^n)\rho_i^{n,-}$ in the first inequality, and $\rho_i^{n,-} \leq 3/2 \rho_{i-\hf}^n$, \eqref{eq:cfl} and \textit{(ii)} in the second. As long as $r \geq 1$, the third inequality follows.

\textit{(v):} The proof is based on the antisymmetry of $W'(x)$. Similar to \eqref{aux}, it is easy to check that \textit{(v)} is equivalent to showing that
$$
\sum_i \left(a_i^{n,+}\rho_i^{n,+} + a_i^{n,-}\rho_i^{n,-} + c^n\left( \rho_i^{n,+}- \rho_i^{n,-}\right) \right)\left(x_\iphf-x_\imhf \right)=0\, \qquad \text{for all } n\in\N.
$$
Since $c^n$ does not depend on $i$, $x_\iphf-x_\imhf=\Delta x$, and taking into account the formulas for $\rho_i^{n,\pm}$ in \eqref{eq:edge_val2}, we deduce that the last statement is equivalent to
$$
\sum_i \left( a_i^{n,+}\rho_i^{n,+} + a_i^{n,-}\rho_i^{n,-} \right)=0\, \qquad \text{for all } n\in\N.
$$
Finally, we have due to the antisymmetry of $W'(x)$ that
\begin{align*}
\sum_i a_i^{n,\pm}\rho_i^{n,\pm} = & \sum_{i\neq j} W'(x_i-x_j) \rho_i^{n,\pm} \rho_j^{n,\pm} \\
= & \sum_{i\neq j} W'(x_j-x_i) \rho_i^{n,\pm} \rho_j^{n,\pm}=- \sum_{i\neq j} W'(x_i-x_j) \rho_i^{n,\pm} \rho_j^{n,\pm},
\end{align*}
leading to
$$
\sum_i a_i^{n,\pm}\rho_i^{n,\pm}=0\, \qquad \text{for all } n\in\N.
$$

\textit{(vi):} The proof is similar to \textit{(iii)}: Multiplying \eqref{eq:scheme_1d} by $\phi_\iphf = \phi(x_\iphf)$ for a Lipschitz continuous function $\phi$ satisfying $\|\phi \|_\Lip \leq 1$ gives
\begin{align*}
\Dx \sum_i & \big(\rho_\iphf^{n+1}-\rho_\iphf^n\big)\phi_\iphf \\
&= \Dx \frac{\beta}{2}\sum_i \left(\big|c+a_i^{n,+}\big|\rho_i^{n,+} + \big|c-a_i^{n,-}\big|\rho_i^{n,-} \right)\big(\phi_\iphf-\phi_\imhf\big) \\
&\leq \frac{\beta}{2}2c^n\Dx^2\|\phi\|_{\Lip} \sum_i \left(\rho_i^{n,+} + \rho_i^{n,-}\right) \\
&\leq 2\|\phi\|_{\Lip}\|W\|_\Lip\Dt,
\end{align*}
and taking the supremum over all $\phi$ with $\|\phi \|_\Lip \leq 1$ yields \eqref{eq:dtimecont} with $t=t^{n+1}$ and $s=t^n$. Iterating over timesteps yields \eqref{eq:dtimecont} for any discrete times $t^n,t^m$, $m,n \in \N$. The inequality \eqref{eq:dtimecont} for any $t,s \in \R_+$ follows from \eqref{eq:num_approx}.

\textit{(vii):} The proof follows similarly to \textit{(iii)}.
\end{proof}

\begin{remark}
Replacing $J_i^n = \frac{1}{2}\left(a_i^{n,+}\rho_i^{n,+} + a_i^{n,-}\rho_i^{n,-} + c^n(\rho_i^{n,+}- \rho_i^{n,-})\right)$ in \eqref{eq:diffusiondef} with the upwind flux \eqref{eq:upwinddef} in the above proof, we can easily deduce the same properties under the same CFL condition. 
\end{remark}

\subsection{Convergence of the method}\label{sec:convergence1d}
Using the properties derived in the previous section we can now prove convergence of the method using a standard compactness technique.

\begin{theorem}\label{thrm:main}
Let $\rho^0 \in \Prob_1(\R)$, assume that $W$ satisfies properties \ref{cond:lip} and \ref{cond:c1}, and that the CFL condition \eqref{eq:cfl} is satisfied. Then for any $T>0$, the numerical approximation \eqref{eq:num_approx} has a uniformly convergent subsequence,
\begin{equation}\label{eq:schemeconverges}
\sup_{t\in[0,T]} d_1\big(\rho_{\Dx'}(t),\rho(t)\big) \to 0 \qquad \text{as } \Dx'\to0,
\end{equation}
and the limit $\rho$ is a $d_1$-weak measure solution of \eqref{eq:agg_eq_1d}, \eqref{eq:min_nrm_velocity} which satisfies
\begin{equation}\label{eq:rholipschitz}
d_1\big(\rho(t),\rho(s)\big) \leq 2\|W\|_{\Lip(\R)}|t-s| \qquad \forall\ t,s\in\R_+.
\end{equation}
If $W$ also fulfills \ref{cond:convex} and $\rho^0 \in \Prob_2(\R)$ then the whole sequence $\rho_\Dx$ converges, and the limit $\rho$ is the unique gradient flow solution of \eqref{eq:agg_eq_1d}.
\end{theorem}
\begin{proof}
Define the set
\[
K := \big\{\rho_\Dx(t)\ :\ \Dx>0,\ t\in[0,T] \big\},
\]
which by Lemma \ref{lem:cfl}~\textit{(i)} and \textit{(iii)} is a subset of $\Prob_1(\R)$ with uniformly bounded first moment. Hence, $K$ is tight, so by Prohorov's theorem $K$ is sequentially precompact in $\Prob(\R)$ with respect to the weak (or ``narrow'') topology (cf.~e.g.~\cite[Theorem 5.1.3]{AGS2005}). We claim that $K$ is also sequentially precompact with respect to $d_1$. By \cite[Theorem 7.12]{Villani}, all we need to check is that the first moments are uniformly integrable with respect to $K$. Fix $\epsilon>0$ and let $r>0$ be such that $\int_{\R\setminus[-r,r]}|x|\rho^0 (dx) < \epsilon$. By Lemma \ref{lem:cfl} \textit{(iv)}, we then have
\begin{align*}
\sup_{\rho\in K} \int_{\R\setminus[-R,R]} |x|\, \rho (dx) < \epsilon C(T)
\end{align*}
for some $R>0$, which proves our claim. Using the $2\|W\|_\Lip$-Lipschitz continuity of $\rho_\Dx$ (Lemma \ref{lem:cfl}~\textit{(vi)}), Ascoli's theorem now implies the existence of a subsequence of $\rho_\Dx$ (which we still denote as $\rho_\Dx$) and some $2\|W\|_\Lip$-Lipschitz continuous $\rho:[0,T]\to \Prob_1(\R)$ such that $d_1\big(\rho_\Dx(t), \rho(t)\big)\to0$ uniformly for $t\in[0,T]$.

We check that the limit $\rho$ satisfies \eqref{eq:agg_eq_1d} in the distributional sense. We multiply $\rho_\Dx^{n+1}$ with a test function $\phi\in C^2_c(\R)$, use \eqref{eq:scheme_1d} and perform a summation by parts,
\begin{equation}\label{eq:distcalc}
\begin{aligned}
\int_\R \phi(x) \rho_\Dx^{n+1}(dx) = &\ \Dx\sum_i \phi(x_\iphf) \rho_\iphf^n \\
&- \frac{\beta\Dx}{2}\sum_i \left(a_i^{n,+}\rho_i^{n,+} + a_i^{n,-}\rho_i^{n,-}\right)\left(\phi(x_\iphf)-\phi(x_\imhf)\right) \\ 
& - \frac{\beta\Dx}{2}\sum_i c^n(\rho_i^{n,+}- \rho_i^{n,-})\left(\phi(x_\iphf)-\phi(x_\imhf)\right).
\end{aligned}
\end{equation}
By Taylor expanding the last term in \eqref{eq:distcalc} around $x_\iphf$, summing by parts, and taking into account \eqref{eq:edge_val2}, Lemma \ref{lem:cfl} {\it (ii)}, and that the mass is conserved, we find that
\begin{align*}
& c^n\frac{\beta\Dx}{2}\sum_i \left(\rho_i^{n,+}- \rho_i^{n,-}\right)\left(\phi(x_\iphf)-\phi(x_\imhf) \right) \\*
=&\ c^n\frac{\beta\Dx}{2} \sum_i \left(\rho_\iphf^n - \rho_\imhf^n -\frac12 \left[ \sigma_i+\sigma_{i+1} -\sigma_{i-1}-\sigma_i\right]\right) \left(\phi(x_\iphf)-\phi(x_\imhf) \right) \\*
=&\ -c^n\frac{\beta\Dx}{2} \sum_i \underbrace{\Big(\rho_\iphf^n - \frac{\sigma_i+\sigma_{i+1}}{2}\Big)}_{\in [0,\rho_{i+1/2}^n]} \underbrace{\big(\phi(x_\ipthf)- 2\phi(x_\iphf)+\phi(x_\imhf)\big)}_{\leq \|\phi''\|_{L^\infty}\Dx^2} \\*
=&\ O(\Dx^2).
\end{align*}
We insert this into the expression \eqref{eq:distcalc}, and by a new Taylor expansion around $x_\iphf$, we know that there exists $y_i \in [x_\imhf, x_\iphf]$ such that 
\begin{align*}
\int_\R \phi(x) & \rho_\Dx^{n+1}(dx) \\
=&\ \Dx \sum_i \phi(x_\iphf) \rho_\iphf^n - \frac{\beta\Dx}{2}\sum_i \left(a_i^{n,+}\rho_i^{n,+} + a_i^{n,-}\rho_i^{n,-}\right)\phi'(x_i)\Dx \\
&- \frac{\beta\Dx}{2}\sum_i\left(a_i^{n,+}\rho_i^{n,+} + a_i^{n,-}\rho_i^{n,-}\right)\phi''(y_i)\frac{\Dx^2}{2} + O(\Dx^2) \\
=&\ \int_\R \phi(x)\,\rho_\Dx(dx,t^n) - \frac{\Dt}{2}\int_\R \phi'(x)a_\Dx^+(x,t^n)\,\rho_\Dx^+(dx,t^n) \\
&\qquad\qquad\qquad \,\,\,\,\, - \frac{\Dt}{2}\int_\R \phi'(x)a_\Dx^-(x,t^n)\,\rho_\Dx^-(dx,t^n) + O(\Dx^2)\,,
\end{align*}
where $a_\Dx^\pm(x,t)=\partial_x^0 W_\Dx \ast \rho_\Dx^\pm(x,t)$ and $\rho_\Dx^\pm$ is defined similar to $\rho_\Dx$, cf.~\eqref{eq:num_approx}. Then for any test function $\phi\in C^2_c(\R\times\R_+)$ we have
\begin{align*}
\int_\R \phi(x,t^n) & \frac{\rho_\Dx(dx,t^n+\Dt)-\rho_\Dx(dx,t^n)}{\Dt} \\
&\ \qquad \qquad \qquad = -\frac{1}{2}\int_\R \partial_x\phi(x,t^n)a_\Dx^{+}(x,t^n) \rho_\Dx^{+}(dx,t^n) \\
& \qquad \qquad \qquad \qquad - \frac{1}{2}\int_\R \partial_x\phi(x,t^n)a_\Dx^{-}(x,t^n)\rho_{\Dx}^{-}(dx,t^n) + O(\Dx).
\end{align*}
The fact that $\rho_{\Dx} \to \rho$, together with the stability property \eqref{eq:reconstrcontinuous} of the reconstruction procedure, implies that also $\rho_{\Dx}^\pm \to \rho$. Recall that $W_\Dx'$ is everywhere continuous. Then the stability result \cite[Lemma 3.1]{CJLV16} implies that $a^\pm_\Dx \rho^\pm_\Dx \wto \ \big(\partial^0 W \ast \rho\big)\rho$, where $\partial^0 W \ast \rho$ is defined in \eqref{eq:min_nrm_velocity}. A standard argument of summation by parts in time implies that $\rho$ is a distributional solution of \eqref{eq:agg_eq_1d} in the sense of \eqref{eq:dpsol}.

We have shown that $\rho$ is a distributional solution of a continuity equation of the form \eqref{eq:gradientfloweq} where the velocity field is given by  $v(t) = -\partial^0 W\ast \rho(t)$ for a.e.~$t > 0$. Furthermore, 
$\|v(t)\|_{L^2(\rho)} \in L^2_{\loc}(0, +\infty)$ since $|v(t,x)|\leq \|W\|_{\Lip}$ for a.e. $t>0$ and $x\in\R$. Finally, from \eqref{eq:rholipschitz} it follows that the continuous curve of probability measures $\rho(t)$ is absolutely continuous in time with respect to $d_1$, and we can thus conclude that $\rho$ is a $d_1$-weak measure solution according to Definition \ref{def:weaksol}.

If $\rho^0 \in \Prob_2(\R)$ then $\rho \in \Prob_2(\R)$ follows from Lemma \ref{lem:cfl}~\textit{(vii)}. Under the additional assumption \ref{cond:convex}, $d_2$-weak measure solutions as defined in \eqref{def:weaksol} are unique, see \cite[Section 2.3]{CDFLS2011}, and they coincide with the unique gradient flow solutions of \eqref{eq:agg_eq_1d} given by Theorem \ref{thrm:unique}. Thus, what remains to show to conclude that $\rho$ is the unique gradient flow solution, is that the $d_1$-weak measure solution $\rho(t)$ is locally in time absolutely continuous in $d_2$. As pointed out in Section 2, since $\|v(t)\|_{L^2(\rho)} \in L^2_{\loc}(0, +\infty)$, we can apply the properties of continuity equations in \cite[Theorem 8.3.1]{AGS2005} which imply the absolute continuity with respect to $d_2$ of $\rho(t)$.
\end{proof}

\begin{remark}
The repulsive potential $W(x)=-|x|$ does not satisfy \ref{cond:convex}. However, due to the equivalence in \cite{BCFP2015} we can apply the proof of Theorem \ref{thrm:main} to obtain the convergence of the numerical scheme also for this potential.
\end{remark}

\begin{remark}
Also from the equivalence in \cite{BCFP2015}, we can deduce from Theorem \ref{thrm:main} the convergence of the minmod scheme \eqref{eq:num_burgers} for Burgers' equation \eqref{eq:burgers} to the unique entropy solution whenever the initial data for Burgers' equation is nondecreasing. See \cite{KNR95} for further results in this direction.
\end{remark}

\begin{remark}
The scheme \eqref{eq:numericalmethod_1d} can be extended to the one-dimensional aggregation equation 
\begin{equation}\label{eq:agg_duality}
\partial_t\rho = \partial_x\bigl(a(W' * \rho)\rho \bigr),
\end{equation}
where $a$ is a nonlinear function. This can be done by carefully defining the velocities $a_i^{n,\pm}$ in \eqref{eq:numericalmethod_1d} using the reconstructed values $\rho_i^-,\rho_i^+$ as it is done for the first-order schemes in \cite{JamesVauch2015, JamesVauch2016} using $\rho_i,\rho_{i+1}$. The resulting scheme will satisfy the properties in Lemma \ref{lem:cfl} for suitable choices of initial data, function $a$ and CFL condition. Following the proof of Theorem \ref{thrm:main}, it will then be straightforward to prove that the resulting second-order numerical approximation converges to the unique duality solution of \eqref{eq:agg_duality} as introduced in \cite{JV2016}.
\end{remark}

\subsection{Truncation error}
Although a proof that our scheme converges at rate $O(\Dx^2)$ is currently out of reach, we can prove an $O(\Dx^2)$ truncation error estimate under the assumption that there exists a smooth solution. For the sake of simplicity we show this result only for the semi-discrete version of \eqref{eq:scheme_1d},
\begin{equation}\label{eq:semidiscrete}
\frac{d}{dt}\rho_{\iphf}(t) = \frac{J_{i+1}(\rho_\Dx(t))-J_i(\rho_\Dx(t))}{\Dx}
\end{equation}
where $J_i(\rho_\Dx(t))$ is given by \eqref{eq:diffusiondef}, with $\rho^n$ replaced by $\rho(t)$. In practice, the semi-discrete scheme \eqref{eq:semidiscrete} must be integrated in time using a second-order time integration method in order to preserve an overall second-order convergence rate. If \emph{strong stability preserving} Runge--Kutta methods are employed then all of the stability and convergence properties proved above are maintained by the fully discrete scheme (see e.g.~\cite{GST01}).
\begin{lemma}\label{lem:convrate}
Assume that the solution of \eqref{eq:agg_eq_1d} lies in $C^2_c(\R\times[0,T])$ for some $T>0$ and let $W \in C^3(\R\setminus \{0\})$ satisfy \ref{cond:lip}. Then the semi-discrete scheme \eqref{eq:semidiscrete} converges at a rate of $O(\Dx^2)$ when measured in $d_1$.
\end{lemma}

\begin{proof}
As is standard in the error analysis of numerical methods for evolution equations, it is enough to show that the local truncation error is $O(\Dt\Dx^2)$ in order to show that the global error is $O(\Dx^2)$.

Let $\mu (x,t)$ be the gradient flow solution of \eqref{eq:agg_eq_1d}, and assume that $\mu$ is sufficiently smooth for $t \in [t^n, t^{n+1}]$. Define the projection $\proj\mu (t) = \Dx\sum_i \mu_\iphf(t)\delta_{x_\iphf}$ where $\mu_\iphf(t)=\mu(t,(x_i,x_{i+1}])/\Dx$. Let $\Dt>0$ be sufficiently small that the system of ODEs \eqref{eq:semidiscrete} with $\proj\mu (t^n)$ as initial data has a unique, bounded solution $\rho_\iphf(t)$ for $t\in[t^n,t^{n+1}]$. As before, denote $\rho_\Dx(t)=\sum_i \rho_\iphf(t)\delta_{x_\iphf}$. We will show that
\begin{equation}\label{key}
d_1\big(\rho_\Dx(t^{n+1}),\proj\mu(t^{n+1})\big) \leq C\Dt\Dx^2
\end{equation}
for some $C>0$ independent of $\Dx,\Dt$. Let $\phi:\R\to\R$ be a Lipschitz continuous function and denote $\phi_\iphf = \phi(x_\iphf)$. Integrating $\phi$ with respect to the error $\mathcal{E}^{n+1}:=\proj\mu (t^{n+1}) - \rho_\Dx(t^{n+1})$ yields
\begin{align*}
\ip{\mathcal{E}^{n+1}}{\phi} & = \int_\R \phi(x) \big( \proj\mu(t^{n+1}) - \rho_\Dx(t^{n+1})\big)(dx)\\
& = \Dx\sum_i \phi_\iphf \bigg[\frac{1}{\Dx}\int_{x_i}^{x_{i+1}}\mu(x,t^{n+1}) dx - \rho_\iphf(t^{n+1}) \bigg] \\
&=\Dx\sum_i \phi_\iphf \Bigg[\frac{1}{\Dx}\int_{x_i}^{x_{i+1}}\left(\mu(x,t^{n}) + \int_{t^n}^{t^{n+1}}\partial_t\mu(x,t)\,dt\right) \,dx \\ 
& \qquad \qquad \qquad \qquad \qquad \qquad- \left(\rho_\iphf(t^{n}) + \int_{t^n}^{t^{n+1}}\frac{d}{dt}\rho_\iphf(t)\,dt\right) \Bigg] \\
&=\Dx\sum_i \phi_\iphf \int_{t^n}^{t^{n+1}}\bigg[\frac{1}{\Dx}\int_{x_i}^{x_{i+1}}\partial_t\mu(x,t)\,dx - \frac{d}{dt}\rho_\iphf(t)\bigg]\,dt \\
&= \sum_i \phi_\iphf \int_{t^n}^{t^{n+1}}\bigg[\big(M_{i+1}(t)-M_i(t)\big)- \big(J_{i+1}(t)-J_i(t)\big) \bigg]\,dt 
\end{align*}
where $M(x,t):=\mu(W'\ast \mu)(x,t)$ and $M_i(t):=M(x_i,t)$. From a summation by parts, and suppressing the dependence on $t$ for the sake of notational simplicity, we find that
\begin{align}
\label{eq:errorexpression}
\ip{\mathcal{E}^{n+1}}{\phi} &= \sum_i  \big(\phi_\iphf-\phi_\imhf\big) \int_{t^n}^{t^{n+1}}\big(J_i(\rho_\Dx)-M_i\big)\,dt \notag\\
& = \sum_i  \big(\phi_\iphf - \phi_\imhf\big)\int_{t^n}^{t^{n+1}}\big[ J_i(\rho_\Dx) - J_i(\proj\mu) \big]\,dt \\
& \quad + \sum_i  \big(\phi_\iphf - \phi_\imhf\big)\int_{t^n}^{t^{n+1}}\bigg[\frac{a_i^{+}\mu_i^{+} + a_i^{-}\mu_i^{-}}{2} - M_i + \frac{c}{2}\big(\mu_i^{+} - \mu_i^{-}\big) \bigg]\,dt, \notag
\end{align}
after adding and subtracting $J_i(\proj\mu)$.
First, consider the last sum in \eqref{eq:errorexpression}.
Observe that $\sigma_i = \min\big(\mu_\imhf, \mu_\iphf\big)=\frac{1}{2}\big(\mu_\imhf + \mu_\iphf-|\mu_\iphf-\mu_\imhf|\big)$. After some tedious but easy computations, one can check that $\mu_i^{+}-\mu_i^{-} = O(\Dx^2)$. Furthermore,
\begin{align*}
M_i - a_i^{+}\mu_i^{+} 
&=\mu(x_i)\big((W'\ast \mu)(x_{i}) - a_i^{+} \big) + a_i^{+}\big(\mu(x_i)-\mu_i^{+}\big) \\ 
&=\mu(x_i)\big((W'\ast \mu)(x_{i}) - a_i^{+} \big) + O(\Dx^2),
\end{align*}
as $a_i^{+}$ is bounded and $\mu_i^{+}-\mu(x_i) = O(\Dx^2)$. We split the first term,
\begin{align*}
\big(W'\ast \mu\big)(x_{i}) - a_i^{+} &= \sum_{i \neq j} \bigg[\int_{\cell_j}W'(x_i-x) \mu(x)\, dx - \Dx W'(x_i-x_j)\mu_j^{+}\bigg] \\
&\quad + \int_{\cell_i}W'(x_i-x) \mu(x)\, dx   \\
&= \sum_{i \neq j} \bigg[\int_{\cell_j}W'(x_i-x) \mu(x)\, dx - \Dx W'(x_i-x_j)\mu(x_j)\bigg] \\*
&\quad + O(\Dx^2) + \int_{\cell_i}W'(x_i-x) \mu(x)\, dx \\ 
&= O(\Dx^2) + \int_{\cell_i}W'(x_i-x) \mu(x)\, dx.
\end{align*}
In the above we could apply the midpoint rule since $W \in C^3(\R \setminus \{0\})$. Furthermore, using the antisymmetry of $W'$,
\begin{align*}
\int_{\cell_i}W'(x_i-x) \mu(x) dx = -\int_0^\frac{\Dx}{2}W'(z)\big(\mu(x_i+z)-\mu(x_i-z)\big)\, dz = O(\Dx^2).
\end{align*}
Finally, using the assumption that $\mu$ has compact support, we get
\begin{align*}
\sum_i \big(\phi_\iphf - \phi_\imhf \big)\Big[\big(\mu (W'\ast \mu)\big)(x_i) - a_i^{+}\mu_i^{+}\Big]& \\ 
\leq \Dx \|\phi\|_\Lip \sum_i \Big|\big(\mu (W'\ast \mu)\big)(x_{i}) - a_i^{+}\mu_i^{+}\Big| &= O(\Dx^2) \|\phi\|_\Lip .
\end{align*}
Applying the same analysis to the term $\big(\mu (W'\ast \mu)\big)(x_i) - a_i^{-}\mu_i^{-}$, we find that the last sum in \eqref{eq:errorexpression} is bounded by  $O(\Dt\Dx^2) \|\phi\|_\Lip$. Now, consider the first sum, 
\begin{align}
& \sum_i  \big(\phi_\iphf - \phi_\imhf\big)\int_{t^n}^{t^{n+1}}\big[ J_i(\rho_\Dx) - J_i(\proj\mu) \big]\,dt \notag \\
 & \qquad =\frac{1}{2}\sum_i\big(\phi_\iphf - \phi_\imhf\big)\int_{t^n}^{t^{n+1}}\big[(b_i^+-a_i^+)\rho_i^+ + (b_i^--a_i^-)\rho_i^- \notag \\
 & \qquad \qquad \qquad \qquad \qquad \qquad \qquad \quad + (a_i^++c)(\rho_i^+ - \mu_i^+) + (a_i^--c)(\rho_i^- -\mu_i^-) \big] dt \label{eq:trunc2}
 \end{align}
where, $b_i^\pm$ is the numerical velocity \eqref{eq:velocities} depending on $\rho_\Dx$, and $a_i^\pm$ \eqref{eq:velocities} depending on $\proj \mu$. Estimating $(W'\ast\mu)(x_{i})-a_i^+ $ as above and assuming $\|\phi\|_\Lip \leq 1$, we get
\begin{align*}
& \sum_i\big(\phi_\iphf - \phi_\imhf\big)\int_{t^n}^{t^{n+1}} (a_i^++c)(\rho_i^+ - \mu_i^+) \,dt \\
& \qquad \leq \sum_i\big(\phi_\iphf - \phi_\imhf\big)\int_{t^n}^{t^{n+1}} ((W'\ast \mu)(x_{i}) +c)(\rho_i^+ - \mu_i^+) \,dt + O(\Dt\Dx^2) \\
& \qquad \leq \int_{t^n}^{t^{n+1}} \|W'\ast \mu\|_\Lip \Dx \sum_i \frac{((W'\ast \mu)(x_{i}) +c)}{\|W'\ast \mu\|_\Lip}(\rho_i^+ - \mu_i^+) \,dt + O(\Dt\Dx^2) \\
& \qquad \leq \sup_{t \in [0,T]}\|W'\ast \mu(t)\|_\Lip \int_{t^n}^{t^{n+1}} d_1\big(\rho_\Dx^+(t),\proj\mu^+(t)\big) \, dt + O(\Dt\Dx^2),
\end{align*}
where $\rho_\Dx^+(t) = \Dx \sum_i \rho_i^+$ and $\proj\mu^+(t)=\Dx \sum_i \mu_i^+$. The first term in \eqref{eq:trunc2} satisfies 
\begin{align*}
&\sum_i\big(\phi_\iphf - \phi_\imhf\big)\int_{t^n}^{t^{n+1}} (b_i^+-a_i^+)\rho_i^+ dt \\
&\qquad \qquad \leq \frac{3}{2}\Dx \|\phi\|_\Lip \|\rho\|_\infty \sum_i \int_{t^n}^{t^{n+1}} |b_i^+-a_i^+| dt 
\end{align*}
as $\rho_\Dx$ is bounded. After splitting the sum into $i < j$ and $j<i$, performing a summation by parts, and remembering that $W$ is in $C^3(\R \setminus \{0\})$, we have that
\begin{align*}
b_i^+-a_i^+ &= \Dx \sum_{i \neq j} W'(x_i-x_j)(\rho_j^+-\mu_j^+) \\
&\leq C \|W''\|_{L^\infty(\R\setminus \{0\})} \sum_j \Dx^2 \Big|\sum_{k\leq j} \rho_k^+-\mu_k^+\Big| + 2\Dx \|W\|_\Lip \Big|\sum_{j\leq i}\mu_j^+-\rho_j^+\Big|,
\end{align*}
where $C$ (here and in the following) is a constant which might depend on $\mu_\Dx, \rho_\Dx$ and $W$, but not on $\Dt$ or $\Dx$. Plugging this into the above, one finds that
\begin{align*}
& \sum_i\big(\phi_\iphf - \phi_\imhf\big)\int_{t^n}^{t^{n+1}} (b_i^+-a_i^+)\rho_i^+ dt \\
& \qquad \qquad \leq C \Dx \|\phi\|_\Lip \sum_i \int_{t^n}^{t^{n+1}} \bigg( \Dx^2 \sum_j \Big|\sum_{k\leq j} \rho_k^+-\mu_k^+\Big| + \Dx \Big|\sum_{j\leq i}\mu_j^+-\rho_j^+\Big| \bigg)\, dt \\
& \qquad \qquad \leq C \|\phi\|_\Lip \int_{t^n}^{t^{n+1}} d_1\big(\rho_\Dx^+(t),\proj\mu^+(t)\big)\, dt.
\end{align*} 
The same analysis can be performed on $\rho^--\mu^-$ and $b^- - a^-$.
Finally, combining all the estimates above, and taking the supremum over all $\phi$ with $\|\phi\|_\Lip \leq 1$ yields
\begin{align*}
d_1\big(\rho_\Dx(t^{n+1}),\proj\mu(t^{n+1})\big) & \leq O(\Dt\Dx^2) + C \int_{t^n}^{t^{n+1}} d_1\big(\rho_\Dx^+(t),\proj\mu^+(t)\big)  \\
& \qquad \qquad \qquad \qquad \qquad \qquad + d_1\big(\rho_\Dx^-(t),\proj\mu^-(t)\big) \, dt \\
& \leq O(\Dt\Dx^2) + C \int_{t^n}^{t^{n+1}} d_1\big(\rho_\Dx(t),\proj\mu(t)\big) \, dt,
\end{align*}
after carefully checking that the second inequality in the above holds. Now, applying Gr\" onwall's inequality, we can conclude that 
\begin{align*}
d_1\big(\rho_\Dx(t^{n+1}),\proj\mu(t^{n+1})\big) = O(\Dt\Dx^2).
\end{align*}
\end{proof}

\subsection{Energy decay}\label{subsec-energy}
As long as the numerical approximation computed with \eqref{eq:numericalmethod_1d} stays bounded, the corresponding interaction energy \eqref{eq:int_energ} decays over time modulo a term of order $\Dx$. For the sake of simplicity we show this result only for the semi-discrete version \eqref{eq:semidiscrete} of \eqref{eq:scheme_1d}, \ \eqref{eq:diffusiondef}.

\begin{proposition}
Assume that $W \in C^2(\R \setminus \{0\})$ satisfies \ref{cond:lip}. Let $$\rho_\Dx(t) = \Dx \sum_i \rho_\iphf(t)\delta_{x_\iphf},$$ where $\rho_\iphf(t)$ is a solution to the semi-discrete scheme \eqref{eq:semidiscrete}, and let $\W (\rho_\Dx(t))$ be the corresponding interaction energy \eqref{eq:int_energ}. If either $\rho_\Dx(t)$ is bounded or $|W(\Dx)|\leq C \Dx^2$, then 
\begin{align}\label{eq:energydec}
\frac{d}{dt} \W (\rho_\Dx(t)) \leq -\frac{\Dx}{4}\sum_i \left(a_i^+(t) +a_{i+1}^-(t)\right)^2\rho_i^+ + K\Dx.
\end{align}
The constant $K$ depends on $\|W\|_\Lip$, $\|W''\|_{L^\infty(\R \setminus \{0\})}$, and either $C$ or $\max_{i \in \Z} \{ \rho_i(t) \}$.
\end{proposition}
\begin{proof}
Denote $W(x_i-x_j)$ as $W_{i-j}$ and $W'(x_i-x_j)$ as $W'_{i-j}$. The time derivative of $\W (\rho_\Dx(t))$ is
\begin{align*}
\frac{d}{dt} \W (\rho_\Dx) &= \frac{1}{2}\frac{d}{dt} \int_{\R^2} W(x-y)\rho_\Dx(dx)\rho_\Dx(dy) \\
& =  \Dx^2 \sum_i \sum_j W_{i-j}\rho_\jphf\partial_t\rho_\iphf.
\end{align*}
From the semi-discrete version of \eqref{eq:scheme_1d}, \eqref{eq:diffusiondef}, a summation by parts and Lemma \ref{lem:cfl} {\it (i)--(ii)}, we get
\begin{align*}
\Dx^2 \sum_i & \sum_j W_{i-j} \rho_\jphf\partial_t\rho_\iphf \\
= & -\frac{1}{2}\Dx \sum_i \sum_j \rho_\jphf \left( W_{i-j}- W_{i-1-j} \right)\left[ a_i^+\rho_i^+ + a_i^-\rho_i^- + c(\rho_i^+ -\rho_i^-) \right] \\
\leq & -\frac{1}{2}\Dx^2 \sum_i \sum_{j\neq i} \rho_\jphf W'_{i-j} \left[ a_i^+\rho_i^+ + a_i^-\rho_i^- + c(\rho_i^+ -\rho_i^-) \right] \\
& +  2\|W\|_\Lip \left[\|W''\|_{L^\infty(\R \setminus \{0\})}\Dx + |W(\Dx)|\Dx \sum_i \rho_\iphf\left(\rho_i^+ + \rho_i^- \right) \right] \\
 =: & I + \mathit{II}.
\end{align*}
If either $\rho$ is bounded, $\max_{i \in \Z} \{ \rho_i \} \leq C$, or $|W(x)|\leq C \Dx^2$, then 
\begin{align*}
\mathit{II} \leq 2\|W\|_\Lip \left[\|W''\|_{L^\infty(\R \setminus \{0\})} + C\right] \Dx.
\end{align*}
To estimate $I$ we use the relation $2\rho_\iphf = \rho_i^+ + \rho_{i+1}^-$,
\begin{align}
I = &  -\frac{\Dx^2}{4} \sum_i \sum_{j\neq i} \left(\rho_i^+ + \rho_{i+1}^-\right) W'_{i-j} \left[ a_i^+\rho_i^+ + a_i^-\rho_i^- + c(\rho_i^+ -\rho_i^-) \right] \nonumber \\
= & -\frac{\Dx}{4} \sum_i \left(a_i^+ +a_{i+1}^- \right) \left[ a_i^+\rho_i^+ + a_i^-\rho_i^- + c(\rho_i^+ -\rho_i^-) \right] \nonumber \\
= & -\frac{\Dx}{4} \sum_i \left(a_i^+ +a_{i+1}^- \right)^2\rho_i^+ - \frac{\Dx}{4}\sum_i \left(a_i^+ +a_{i+1}^- \right)\left( a_i^- - a_{i+1}^- \right)\rho_i^-  \nonumber \\
&-\frac{\Dx}{4}\sum_i \left(a_i^+ +a_{i+1}^- \right)\big(c-a_{i+1}^-\big)\big(\rho_i^+-\rho_i^-\big). \label{eq:inter_energy_est}
\end{align}
By a summation by parts and the antisymmetry of $W$,
\begin{align*}
 a_i^- - a_{i+1}^- &= \Dx \sum_{j\neq i} W'(x_i - x_j)\left(\rho_j^- - \rho_{j+1}^- \right) \\
 &= \Dx \sum_{j \neq i,i+1} \left(W'_{i-j} - W'_{i+1-j} \right)\rho_j^- -\Dx W'(\Dx)\big(\rho_i^-+\rho_{i+1}^- \big),
\end{align*}
which is bounded by $\left[\|W''\|_{L^\infty(\R \setminus \{0\})} + 2C\right]\Dx$ under the given assumptions on $W$ and $\rho_\Dx$. It follows that the second term in \eqref{eq:inter_energy_est} is bounded by the same expression as $\mathit{II}$. From \eqref{eq:edge_val2}, $\rho_i^+-\rho_i^- = \rho_\iphf -\rho_\imhf - \hf (\sigma_{i+1}-\sigma_{i-1})$. Then, after yet another summation by parts, the last term in \eqref{eq:inter_energy_est} can be bounded similarly to the second term,
\begin{align*}
\frac{\Dx}{4}\sum_i \left(a_i^+ +a_{i+1}^- \right)\big(c-a_{i+1}^-\big)\big(\rho_i^+-\rho_i^-\big) \leq 3 \|W\|_\Lip \left[\|W''\|_{L^\infty(\R \setminus \{0\})} + C\right] \Dx.
\end{align*}
This concludes the proof.
\end{proof}

A similar expression to \eqref{eq:energydec} can also be found for the semi-discrete version of the upwind scheme \eqref{eq:scheme_1d}, \eqref{eq:upwinddef}.


\section{Extension to several dimensions}
\noindent We proceed by extending the scheme derived in the previous section to multiple spatial dimensions. For the sake of notational simplicity we consider only the two-dimensional version of the aggregation equation \eqref{eq:agg_eq},
\begin{equation*}
\partial_t \rho = \partial_x\big(\big(\partial_x W * \rho\big)\rho \big) + \partial_y\big(\big(\partial_y W * \rho\big)\rho \big),
\end{equation*}
although the scheme derived here is applicable for any number of space dimensions. Moreover, we will restrict ourselves to Cartesian (rectangular) meshes, and we postpone the design of numerical schemes for more general (triangular or quadrilateral) meshes to a future paper. Thus, we consider a mesh of equispaced gridpoints $\mathbf{x}_{\iphf,\jphf}:=(x_\iphf,y_\jphf)$, where $x_\iphf-x_\imhf=\Dx$ and $y_\jphf-y_\jmhf=\Dy$. The spatial domain is partitioned into cells $\cell_{i,j} = [x_\imhf,x_\iphf)\times[y_\jmhf,y_\jphf)$. 

\begin{figure}
\centering
\subfigure[The mass in the striped domain is $\Dx\Dy \rho_{\iphf,\jphf}$ for both the numerical approximation and the reconstruction.]{
\includegraphics[width=0.48\textwidth]{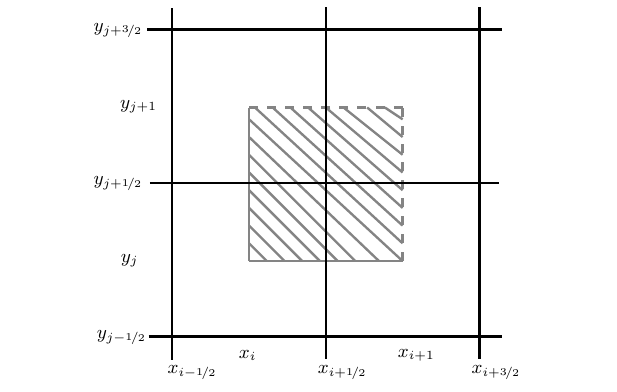}}
\hspace{0.2em}
\subfigure[The subdomains measured by $r$ to obtain the reconstructed values. Red: $\rho_{i,\jphf}^E$, magenta: $\rho_{\iphf,j+1}^S$, blue: $\rho_{i+1,\jphf}^W$, green: $\rho_{\iphf,j}^N$.]{
\includegraphics[width=0.48\textwidth]{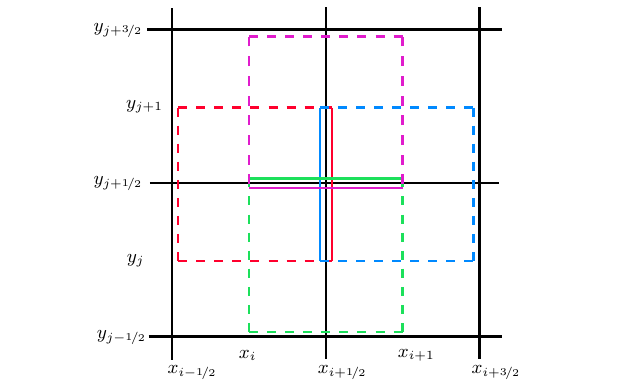}}
\caption{Reconstruction in two dimensions.}
\label{fig:rec_values}
\end{figure}  
A finite volume method for a two-dimensional conservation law would approximate the average over each cell $\cell_{i,j}$. By duality, we let the numerical approximation $\rho_{\iphf,\jphf}$ be centered at the vertices $\mathbf{x}_{\iphf,\jphf}$. Given a numerical approximation 
$\rho_\Delta\in \Prob_1(\R^2)$ of the form
\begin{equation*}
\rho_\Delta = \Dx\Dy\sum_{i,j}\rho_{\iphf,\jphf}\delta_{\mathbf{x}_{\iphf,\jphf}}
\end{equation*}
(here and below we suppress the dependence on $n$ for the sake of notational convenience), we perform a reconstruction by defining
\begin{align*}
r &= \Dx\Dy\sum_{i,j} \tilde{\rho}_{\iphf,\jphf}\delta_{\mathbf{x}_{\iphf,\jphf}} + \sum_{i,j}\sigma_{i,j}\mathcal{L}\big|_{\cell_{ij}}, \\*
\tilde{\rho}_{\iphf,\jphf} &= \rho_{\iphf,\jphf} -\frac{1}{4}\big(\sigma_{ij}+\sigma_{i+1,j+1}+\sigma_{i,j+1}+\sigma_{i+1,j}\big),\\*
\sigma_{i,j} &= \min\big\{\rho_{k+\hf,l+\hf}\,:\,\text{$k$ and $l$ are such that }(x_{k+\hf},y_{l+\hf})\in\overline{\cell_{i,j}}\big\}.
\end{align*}
Next, define the four reconstructed values
\begin{align*}
\rho_{i,\jphf}^W &= r\big([x_\imhf,x_\iphf)\times(y_j,y_{j+1})\big), &\rho_{i,\jphf}^E = r\big((x_\imhf,x_\iphf]\times(y_j,y_{j+1})\big), \\
\rho_{\iphf,j}^N &= r\big((x_i,x_{i+1})\times(y_\jmhf,y_\jphf]\big), &\rho_{\iphf,j}^S = r\big((x_i,x_{i+1})\times[y_\jmhf,y_\jphf)\big)
\end{align*}
(cf.~Figure \ref{fig:rec_values}). Using the definition of $r$ it is easy to show that
\begin{equation}\label{eq:localconservation2D}
\frac{1}{4}\big(\rho_{i,\jphf}^E + \rho_{i+1,\jphf}^W + \rho_{\iphf,j}^N + \rho_{\iphf,j+1}^S\big) = \rho_{\iphf,\jphf}.
\end{equation}
Moreover, $r$ is a nonnegative measure in $\Prob_1(\R^2)$ and
\begin{equation*}
d_1(r,\rho_\Delta) \leq \Dx+\Dy.
\end{equation*}
Let
\begin{equation*}
\rho_{\iphf,\jphf}^0=\frac{1}{\Dx\Dy}\rho^0((x_i,x_{i+1}] \times (y_i,y_{i+1}]).
\end{equation*}
Dropping the superindex $^n$ for notational convenience, we propose the following Lax--Friedrichs type scheme:
\begin{equation}\label{eq:laxfr2D}
\rho_{\iphf,\jphf}^{n+1} = \rho_{\iphf,\jphf} + \frac{\Dt}{\Dx}\big(J_{i+1,\jphf}-J_{i,\jphf}\big) + \frac{\Dt}{\Dy}\big(J_{\iphf,j+1}-J_{\iphf,j}\big)
\end{equation}
where the numerical flux function at time $t^n$ is defined as
\begin{equation*}
\begin{split}
J_{i,\jphf} &= \frac{(a^W\rho^W)_{i,\jphf}+(a^E\rho^E)_{i,\jphf}}{2} + \frac{c}{2}\big(\rho_{i,\jphf}^E-\rho_{i,\jphf}^W\big), \\
J_{\iphf,j} &= \frac{(a^N\rho^N)_{\iphf,j}+(a^S\rho^S)_{\iphf,j}}{2} + \frac{c}{2}\big(\rho_{\iphf,j}^N-\rho_{\iphf,j}^S\big),
\end{split}
\end{equation*}
and
\begin{align*}
a^W_{i,\jphf} &= \big(\partial_x^0 W * \rho^W\big)_{i,\jphf} = \Dx\Dy\sum_{(k,l) \neq (i,j)} \partial_x W\big(x_i-x_k,y_\jphf-y_\lphf\big)\rho^W_{k,\lphf}, \\*
a^N_{\iphf,j} &= \big(\partial_y^0 W * \rho^N\big)_{\iphf,j} = \Dx\Dy\sum_{(k,l) \neq (i,j)} \partial_y W\big(x_\iphf-x_\kphf,y_j-y_l\big)\rho^N_{\kphf,l}, \\*
a^E_{i,\jphf} &= \big(\partial_x^0 W * \rho^E\big)_{i,\jphf} = \Dx\Dy\sum_{(k,l) \neq (i,j)} \partial_x W\big(x_i-x_k,y_\jphf-y_\lphf\big)\rho^E_{k,\lphf}, \\*
a^S_{\iphf,j} &= \big(\partial_y^0 W * \rho^S\big)_{\iphf,j} = \Dx\Dy\sum_{(k,l) \neq (i,j)} \partial_y W\big(x_\iphf-x_\kphf,y_j-y_l\big)\rho^S_{\kphf,l}\,.
\end{align*}
Analogously, one can define an upwind-type scheme by mimicking the definition \eqref{eq:upwinddef} by
\begin{equation*}
\begin{split}
J_{i,\jphf} &= \max\big(a^W_{i,\jphf},0\big) \rho^W_{i,\jphf} + \min\big(a^E_{i,\jphf},0\big) \rho^E_{i,\jphf} \,, \\
J_{\iphf,j} &= \max\big(a^N_{\iphf,j},0\big) \rho^N_{\iphf,j} + \min\big(a^S_{\iphf,j},0\big) \rho^S_{\iphf,j}\,.
\end{split}
\end{equation*}
Using \eqref{eq:localconservation2D} it is straightforward to rewrite \eqref{eq:laxfr2D} as
\begin{align*}
&\rho_{\iphf,\jphf}^{n+1} \\
& = \rho_{i,\jphf}^E\left(\frac{1}{4} - \frac{\Dt}{2\Dx}\big(c+a^E_{i,\jphf}\big)\right) + \rho_{i+1,\jphf}^W\left(\frac{1}{4} - \frac{\Dt}{2\Dx}\big(c-a^W_{i+1,\jphf}\big)\right) \\
&+ \rho_{\iphf,j}^N\left(\frac{1}{4} - \frac{\Dt}{2\Dy}\big(c+a^N_{\iphf,j}\big)\right) + \rho_{\iphf,j+1}^S\left(\frac{1}{4} - \frac{\Dt}{2\Dy}\big(c-a^S_{\iphf,j+1}\big)\right) \\
&+ \rho_{i+1,\jphf}^E\frac{\Dt}{2\Dx}\big(c+a^E_{i+1,\jphf}\big) + \rho_{i,\jphf}^W\frac{\Dt}{2\Dx}\big(c-a^W_{i,\jphf}\big) \\
&+ \rho_{\iphf,j+1}^N\frac{\Dt}{2\Dy}\big(c+a^N_{\iphf,j+1}\big) + \rho_{\iphf,j}^S\frac{\Dt}{2\Dy}\big(c-a^S_{\iphf,j}\big).
\end{align*}
The coefficients of the reconstructed values of $\rho$ are nonnegative if we choose e.g.
\[
c \geq |a^E|,|a^W|,|a^N|,|a^S|, \qquad c\Dt \leq \frac{\min(\Dx,\Dy)}{4}.
\]
Since $|a^E|,|a^W|,|a^N|,|a^S| \leq \|W\|_\Lip$, a sufficient condition for nonnegativity of $\rho^{n+1}$ is
\begin{equation}\label{eq:cfl2d}
\Dt \leq \frac{\min(\Dx,\Dy)}{4\|W\|_\Lip}.
\end{equation}
We state this and the remaining stability properties in the following lemma. As in Section \ref{sec:schemeprop1d}, we define the linear interpolation
\begin{equation}\label{eq:num_approx2d}
\rho_\Delta (t) := \frac{t^{n+1}-t}{\Dt}\rho_\Delta(t^n) + \frac{t-t^n}{\Dt}\rho_\Delta(t^{n+1}), \qquad t\in[t^n,t^{n+1})
\end{equation}
where $\rho_\Delta(t^n) = \Dx\Dy\sum_{i,j} \rho_{\iphf,\jphf}^n\delta_{x_\iphf,y_\jphf}$ and $\rho_{\iphf,\jphf}^n$ is computed with the numerical scheme \eqref{eq:laxfr2D}.

\begin{lemma}\label{lem:cfl2d}
Assume that $\rho^0 \in \Prob_1(\R^2)$ and that $W$ satisfies \ref{cond:lip}--\ref{cond:c1}. Consider the scheme \eqref{eq:laxfr2D} with
\begin{equation*}
c^n = \max_{i,j}\big\{|a^E_{i,\jphf}|,|a^W_{i,\jphf}|,|a^N_{\iphf,j}|,|a^S_{\iphf,j}|\big\}
\end{equation*}
and assume that $\Dt$ satisfies the \emph{CFL condition} \eqref{eq:cfl2d}. Then for all $t\geq0$ and $n\in\N_0$:
\begin{itemize}
\item[(i)] $\rho_\Delta(t) \geq 0$ and $\int_{\R^2}\rho_\Delta(d\bx,t) = 1$,
\item[(ii)] $c^n \leq \|W\|_\Lip$,
\item[(iii)] $\rho_\Delta(t)$ has bounded first order moment: 
\begin{equation*}
\int_{\R^2}|\bx|\,\rho_\Delta(d\bx,t) \leq \int_{\R^2}|\bx|\,\rho^0(d\bx) + 4t\|W\|_\Lip,
\end{equation*}
\item[(iv)] Let $r \geq 1$ and $\epsilon > 0$. Then
\begin{equation*}
\int_{\R^2 \setminus [-r,r]^2} |\bx|\,\rho^0(d\bx) < \epsilon \quad \implies  \quad \int_{\R^2 \setminus [-R,R]^2} |\bx|\,\rho_\Delta(d\bx, t) < \epsilon C(t),
\end{equation*}
where $R=r+t (\max \{ \Dx,\Dy\}/\Dt)$ and $C(t) = \exp \left(\frac{7}{2} \|W\|_\Lip t \right)$.
\item[(v)] The center of mass is preserved in time, i.e., 
\[
\int_{\R^2} \bx \,\rho_\Delta(d\bx, t) = \int_{\R^2} \bx \,\rho^0(d\bx) \qquad \text{for all } n\in\N\,.
\]
\item[(vi)] The map $t \mapsto \rho_\Delta(t)$ is uniformly Lipschitz, in the sense that
\begin{equation*}
d_1\big(\rho_\Delta(t), \rho_\Delta(s)\big) \leq 4\|W\|_{\Lip}|t-s|
\end{equation*}
for all $t,s\geq0$, where $d_1$ denotes the Monge--Kantorovich--Rubinstein metric.
\item[(vii)] If in addition $\rho^0 \in \Prob_2(\R)$ then $\rho_\Delta(t)$ has bounded second order moment:
\begin{equation*}
\begin{split}
\int_{\R^2}|\bx|^2\,\rho_\Delta(d\bx,t) &\leq \int_{\R^2}|\bx|^2\,\rho^0(d\bx) + 48 t^2\|W\|_\Lip^2 + 12 t\|W\|_\Lip\int_{\R^2}|\bx|\,\rho^0(d\bx), 
\end{split}
\end{equation*}
\end{itemize}
\end{lemma}
\begin{proof}
The proof is a simple extension of the proof of Lemma \ref{lem:cfl} to two dimensions and is therefore omitted.
\end{proof}

By exactly the same approach as in Section \ref{sec:convergence1d}, we can prove convergence of the two-dimensional scheme.
\begin{theorem}\label{thrm:conv2d}
Let $\rho^0 \in \Prob_1(\R^2)$, assume that $W$ satisfies properties \ref{cond:lip} and \ref{cond:c1}, and that the CFL condition \eqref{eq:cfl2d} is satisfied. Then for any $T>0$, the numerical approximation $\rho_\Delta$ generated by the scheme \eqref{eq:laxfr2D} has a uniformly convergent subsequence,
\begin{equation*}
\sup_{t\in[0,T]} d_1\big(\rho_{\Delta'}(t),\rho(t)\big) \to 0 \qquad \text{as } \Delta'=(\Delta x', \Delta y')\to0,
\end{equation*}
and the limit $\rho$ is a $d_1$-weak measure solution of \eqref{eq:agg_eq} satisfying
\begin{equation*}
d_1(\rho(t),\rho(s)) \leq 2\|W\|_{\Lip(\R)}|t-s| \qquad \forall\ t,s\in\R_+.
\end{equation*}
If $W$ also satisfies \ref{cond:convex} and $\rho^0 \in \Prob_2(\R^2)$ then the whole sequence $\rho_\Delta$ converges, and the limit $\rho$ is the unique gradient flow solution of \eqref{eq:agg_eq}.
\end{theorem}

\begin{lemma}\label{lem:2Dtrunc}
Assume that $\rho(t) \in C^2_c(\R^2)$ is sufficiently smooth in time and let $W \in C^3(\R^2\setminus \{0\})$ satisfy \ref{cond:lip}. Then the numerical scheme \eqref{eq:numericalmethod_1d} converges at a rate of $O(\Dx^2+\Dy^2)$ when measured in $d_1$.
\end{lemma}
\begin{proof}
The proof is a straightforward, but tedious, adaptation of Lemma \ref{lem:convrate}. The step in that proof that uses the antisymmetry of $W'$ carries over to this case by splitting the rectangle $\cell_i$ into four parts, diagonally opposing pairs of which cancel (up to $O(\Dx^2+\Dy^2)$) due to the antisymmetry of $W$.
\end{proof}

Let us finally remark that the generalization to an arbitrary number of dimensions of the previous scheme is a straightforward extension of the scheme presented here.


\begin{figure}[ht!]
\centering
\subfigure[Initial data.]{
\includegraphics[width=0.304\textwidth]{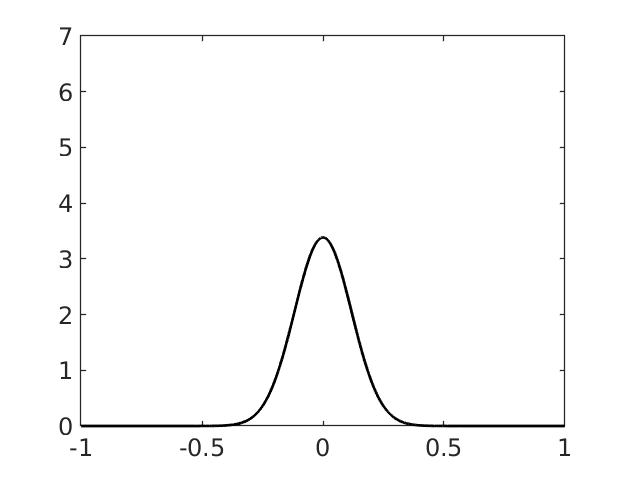}}\subfigure[$t=0.075$]
{\includegraphics[width=0.30\textwidth]{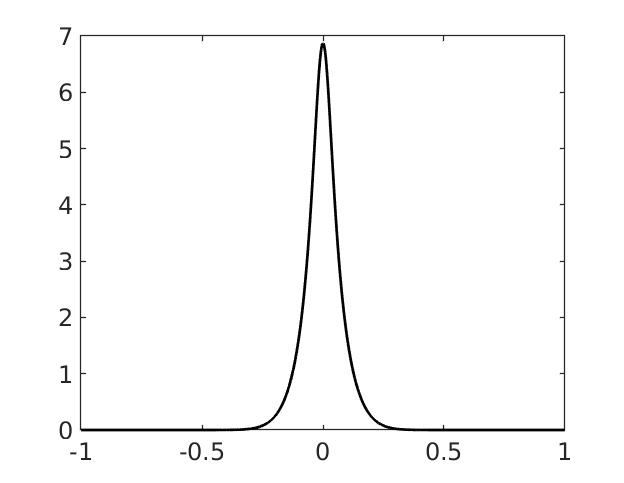}}
\caption{Initial data \eqref{eq:smooth_init} and corresponding solution of \eqref{eq:agg_eq_1d} with $W(x)=|x|$ at $t=0.075$.}\label{fig:simulation1D}
\end{figure}

\section{Numerical simulations}
\noindent We provide several numerical examples to examine the performance of the numerical scheme developed in this paper. We compare it to two numerical schemes: the first-order Lax--Friedrichs type scheme in \cite{CJLV16} and the first-order upwind type scheme in \cite{DLV2017}, which will be denoted as \emph{1st LxF} and \emph{1st upw} respectively. Section \ref{sub:1D} is devoted to numerical simulations of the one-dimensional scheme \eqref{eq:numericalmethod_1d} with a main focus on the convergence rates. In Section \ref{sub:2D} we study the two-dimensional scheme \eqref{eq:laxfr2D} and qualitatively compare it to the first-order schemes.

\subsection{Experiments in 1D}\label{sub:1D}
In this section the convergence rate of \eqref{eq:numericalmethod_1d} is addressed through different examples. We also provide a few examples to study the qualitative behavior. In all the numerical experiments the CFL number is set to $0.4$ and $c^n = \|W\|_\Lip$ for all $n$. A third-order SSP Runge--Kutta method is used to integrate in time, see \cite{GST01}.

\subsubsection{Smooth initial data}
We give an example to numerically verify the second-order convergence rate of \eqref{eq:numericalmethod_1d} for smooth enough data by considering approximations of \eqref{eq:agg_eq_1d} using \eqref{eq:numericalmethod_1d} with initial data
\begin{equation}\label{eq:smooth_init} 
\rho^0 =  \frac{1}{\sqrt{\pi}} \exp(-36x^2),
\end{equation} 
see Figure \ref{fig:simulation1D}(a).
 
The convergence rates can be found in Tables \ref{tab:smooth3} and \ref{tab:smooth6}. In Table \ref{tab:smooth3} we consider the numerical approximation with the attractive potential $W(x)=|x|$ at a time before blow-up of the solution, see Figure \ref{fig:simulation1D}(b). The numerical approximation is compared to a reference solution found by approximating the solution of Burgers' equation \eqref{eq:burgers} using a second-order method on a very fine grid and then differentiating the solution at the level of \eqref{eq:burgers}. We can see that the second-order method \eqref{eq:numericalmethod_1d} converges at rate close to 2 using either of the fluxes \eqref{eq:diffusiondef} (2nd LxF) and \eqref{eq:upwinddef} (2nd upw). This is clearly an improvement over the rates of the first-order methods. 

\begin{table}[ht!]
\centering
\small
\begin{tabular}{|c|c|c|c|c|c|c|c|c|}
\cline{2-9}
\multicolumn{1}{c}{} & \multicolumn{2}{|c|}{1st LxF} & \multicolumn{2}{|c|}{1st upw} & \multicolumn{2}{|c|}{2nd LxF} & \multicolumn{2}{|c|}{2nd upw} \\
\hline 
$ n $ & $d_1$ & OOC & $d_1$ & OOC & $d_1$ & OOC &  $d_1$ & OOC \\ 
\hline 
$ 32 $   & $ 1.66e-02 $ &           & $ 7.44e-03 $ &           & $ 5.38e-03 $ &           & $ 3.08e-03 $ &  \\ 
$ 64 $   & $ 9.77e-03 $ & $ 0.77 $ & $ 4.63e-03 $ & $ 0.68 $ & $ 1.49e-03 $ & $ 1.85 $ & $ 1.22e-03 $ & $ 1.34 $ \\ 
$ 128 $  & $ 5.23e-03 $ & $ 0.90 $ & $ 2.95e-03 $ & $ 0.65 $ & $ 4.94e-04 $ & $ 1.60 $ & $ 4.07e-04 $ & $ 1.58 $ \\ 
$ 256 $  & $ 2.70e-03 $ & $ 0.96 $ & $ 1.75e-03 $ & $ 0.75 $ & $ 1.50e-04 $ & $ 1.72 $ & $ 1.13e-04 $ & $ 1.85 $\\ 
$ 512 $  & $ 1.37e-03 $ & $ 0.98 $ & $ 9.54e-04 $ & $ 0.88 $ & $ 4.15e-05 $ & $ 1.85 $ & $ 2.92e-05 $ & $ 1.95 $ \\ 
$ 1024 $ & $ 6.88e-04 $ & $ 0.99 $ & $ 4.97e-04 $ & $ 0.94 $ & $ 1.08e-05 $ & $ 1.94 $ & $ 7.32e-06 $ & $ 2.00 $ \\ 
\hline 
\end{tabular}
\caption{Convergence rates for $W(x)=|x|$ with the smooth initial data \eqref{eq:smooth_init} at $t=0.075$.}\label{tab:smooth3}
\end{table}


\begin{table}[ht!]
\centering
\small
\begin{tabular}{|c|c|c|c|c|c|c|c|c|}
\cline{2-9}
\multicolumn{1}{c}{} & \multicolumn{2}{|c|}{1st LxF} & \multicolumn{2}{|c|}{1st upw} & \multicolumn{2}{|c|}{2nd LxF} & \multicolumn{2}{|c|}{2nd upw} \\
\hline 
$ n $ & $d_1$ & OOC & $d_1$ & OOC & $d_1$ & OOC &  $d_1$ & OOC \\ 
 \hline 
$ 32 $  & $ 1.71e-02 $  &           & $ 6.59e-03 $  &           & $ 4.93e-03 $  &           & $ 2.54e-03 $ & \\ 
$ 64 $  & $ 9.72e-03 $  & $ 0.82 $ & $ 4.19e-03 $  & $ 0.65 $ & $ 1.34e-03 $  & $ 1.88 $ & $ 1.01e-03 $ & $ 1.32 $ \\ 
$ 128 $ & $ 5.01e-03 $ & $ 0.93 $ & $ 2.64e-03 $  & $ 0.67 $ & $ 4.59e-04 $  & $ 1.54 $ & $ 3.27e-04 $ & $ 1.63 $ \\ 
$ 256 $ & $ 2.57e-03 $  & $ 0.99 $ & $ 1.50e-03 $  & $ 0.81 $ & $ 1.37e-04 $  & $ 1.74 $ & $ 8.98e-05 $ & $ 1.86 $\\ 
$ 512 $ & $ 1.26e-03 $  & $ 1.03 $ & $ 7.84e-04 $  & $ 0.94 $ & $ 3.75e-05 $  & $ 1.87 $ & $ 2.30e-05 $ & $ 1.97 $ \\ 
$ 1024 $& $ 5.90e-04 $  & $ 1.09 $ & $ 3.80e-04 $  & $ 1.05 $ & $ 9.69e-06 $  & $ 1.95 $ & $ 5.73e-06 $ & $ 2.00 $  \\ 
\hline 
\end{tabular}
\caption{Convergence rates for $W(x)=1-\exp(-|x|)$ with the smooth initial data \eqref{eq:smooth_init} at $t=0.075$.}\label{tab:smooth6}
\end{table}

We observe similar convergence rates for interaction potentials where we do not have the equivalence between solutions of \eqref{eq:agg_eq_1d} with $W(x)=\pm|x|$ and \eqref{eq:burgers}, see Table \ref{tab:smooth6}. Here the reference solutions are computed with the respective numerical schemes on a grid consisting of $2^{13}$ cells.

\subsubsection{Measure valued initial data}
\begin{table}
\centering
\begin{tabular}{|c|c|c|c|c|}
\hline 
$ n $ & $d_1$ & OOC & $d_1$ & OOC\\ 
\hline  
$ 32 $ & $ 2.17e-02 $ & & $ 2.29e-02 $ &  \\ 
$ 64 $ & $ 1.23e-02 $ & $ 0.83 $ & $ 1.17e-02 $ & $ 0.97 $ \\ 
$ 128 $ & $ 5.96e-03 $ & $ 1.04 $ & $ 5.11e-03 $ & $ 1.20 $ \\ 
$ 256 $ & $ 3.10e-03 $ & $ 0.95 $ & $ 3.01e-03 $ & $ 0.76 $ \\ 
$ 512 $ & $ 1.49e-03 $ & $ 1.06 $ & $ 1.53e-03 $ & $ 0.98 $ \\ 
$ 1024 $ & $ 7.88e-04 $ & $ 0.92 $ & $ 7.39e-04 $ & $ 1.05 $ \\ 
\hline 
\end{tabular}
\caption{Convergence rates of \eqref{eq:scheme_1d}, \eqref{eq:diffusiondef} with the Dirac initial data \eqref{eq:dirac_init} at $t=0.1$. Left: $W(x)=|x|$; right: $W(x)=1-\exp(-|x|)$.}
\label{tab:dirac1}
\end{table}

We check the convergence rate of the scheme \eqref{eq:scheme_1d}, \eqref{eq:diffusiondef} with potentials $W(x)=|x|$ and $W(x)=1-\exp(-|x|)$ in the case of measure valued initial data represented by the sum of two Dirac measures,
\begin{equation}\label{eq:dirac_init}
\rho^0 = \frac{1}{2}\big(\delta_{-0.5} + \delta_{0.5}\big).
\end{equation}
The numerical approximation is compared to the exact solution of \eqref{eq:agg_eq_1d} for $W(x)=|x|$ (which can be found by solving the corresponding particle system \eqref{eq:particlesys}). In the case $W(x)=1-\exp (-|x|)$ the reference solution is found by approximating the position of the Diracs using a very small timestep. Both reference solutions are projected onto the same grid as the numerical approximation. The numerical approximations converge at a rate of $\Dx$ in $d_1$, see Table \ref{tab:dirac1}. This is exactly what we expect in the case $W(x)=|x|$ as it corresponds to a rate of $\Dx$ in $L^1$ for two initial shocks at the level of Burgers' equation \eqref{eq:burgers}. See \cite{TZ97} and \cite{FS2016} for further results on convergence rates for conservation laws.


\subsubsection{A possible optimal convergence rate}
As observed in the previous section, a second-order convergence rate (in $d_1$) is not always achievable. Indeed, even in the case of (formally) first-order schemes, one does not always obtain a convergence rate of 1. Delarue, Lagouti\`ere and Vauchelet prove that their first-order upwind type scheme converges at a rate of $\hf$ in the 2-Wasserstein distance $d_2$ in \cite{DLV2017}. Furthermore, an example showing that this rate is optimal (in both $d_1$ and $d_2$) is provided: $W(x)=2x^2$ for $|x| \leq 1$, $W(x)=4|x|-2$ for $|x|>1$ and $\rho^0 = 0.5 \delta_{-0.25} +0.5 \delta_{0.25}$.
The exact solution of \eqref{eq:agg_eq_1d}, found by solving the corresponding particle system \eqref{eq:particlesys}, is
\begin{align*}
\rho(t) = \frac{1}{2}\big(\delta_{-x(t)} + \delta_{x(t)}\big), \qquad x(t) =0.25e^{-4t},
\end{align*}
in this case. Applying the second-order scheme \eqref{eq:numericalmethod_1d} to this example, the convergence rate improves to $\unitfrac{2}{3}$ and $\unitfrac{3}{4}$, see Figure \ref{fig:rate}. Even though the rate is far from 2, this suggests that the optimal rate of \eqref{eq:numericalmethod_1d} is somewhat higher than the one for similar first-order schemes.
\begin{figure}[ht!]
\centering
\includegraphics[width=0.45\textwidth]{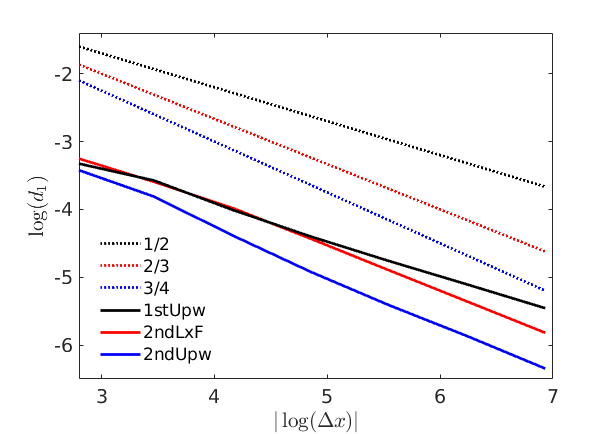}
\caption{Possible optimal $d_1$ convergence rates. Black: 1st upw from \cite{DLV2017}. Red: \eqref{eq:scheme_1d}, \eqref{eq:diffusiondef}. Blue: \eqref{eq:scheme_1d}, \eqref{eq:upwinddef}. Time $t=0.5$.}
\label{fig:rate}
\end{figure}

\subsubsection{Attractive-repulsive potentials}\label{sec:attrep1d}
In the case of the attractive-repulsive potential $$W(x)=\hf |x|^2-|x|,$$ it is known that the (unique) steady state solution of \eqref{eq:agg_eq} is $\hf\chi_{[-1,1]}$, where $\chi_{[-1,1]}$ is the characteristic of the interval $[-1,1]$, see for instance \cite{FR1,FR2,FHK2011}. Note that this potential satisfies conditions \ref{cond:lip} and \ref{cond:c1}, but not \ref{cond:convex}. Hence, from Theorem \ref{thrm:main} we only know that there is a subsequence of \eqref{eq:num_approx} converging to a $d_1$-weak measure solution $\rho$. We apply the schemes \eqref{eq:numericalmethod_1d} to $W(x)=\hf |x|^2-|x|$ and initial data
\begin{equation}\label{eq:cosinit}
\rho^0=\begin{cases}
\frac{\pi}{1.2}\cos\left(\frac{\pi}{0.6}x\right) & \text{if } -0.3 \leq x \leq 0.3, \\
0 & \textrm{otherwise},
\end{cases}
\end{equation}
to see if they converge (in $d_1$) to the right steady state solution. We also apply the schemes 1st LxF and 1st upw to the same test case for comparison.

\begin{table}[ht!]
\centering
\small 
\begin{tabular}{|c|c|c|c|c|c|c|c|c|}
\cline{2-9}
\multicolumn{1}{c}{} & \multicolumn{2}{|c|}{1st LxF} & \multicolumn{2}{|c|}{1st upw} & \multicolumn{2}{|c|}{2nd LxF} & \multicolumn{2}{|c|}{2nd upw} \\
\hline 
$ n $ & $d_1$ & OOC & $d_1$ & OOC & $d_1$ & OOC &  $d_1$ & OOC \\ 
\hline 
$ 32 $ & $ 5.71e-02 $ &  & $ 1.77e-03 $ &  & $ 1.69e-02 $ & & $ 1.66e-03 $ & \\ 
$ 64 $ & $ 3.01e-02 $ & $ 0.92 $ & $ 4.70e-04 $ & $ 1.92 $ & $ 7.12e-03 $ & $ 1.25 $ & $ 1.66e-04 $ & $ 3.32 $ \\ 
$ 128 $ & $ 1.54e-02 $ & $ 0.97 $ & $ 1.19e-04 $ & $ 1.98 $ & $ 2.89e-03 $ & $ 1.30 $ & $ 1.84e-05 $ & $ 3.18 $ \\ 
$ 256 $ & $ 7.75e-03 $ & $ 0.99 $ & $ 3.00e-05 $ & $ 1.99 $ & $ 1.14e-03 $ & $ 1.34 $ & $ 3.15e-06 $ & $ 2.54 $ \\ 
$ 512 $ & $ 3.89e-03 $ & $ 0.99 $ & $ 7.51e-06 $ & $ 2.00 $ & $ 4.56e-04 $ & $ 1.32 $ & $ 7.21e-07 $ & $ 2.13 $ \\ 
$ 1024 $ & $ 1.95e-03 $ & $ 1.00 $ & $ 1.88e-06 $ & $ 2.00 $ & $ 1.81e-04 $ & $ 1.34 $ & $ 1.44e-07 $ & $ 2.32 $ \\ 
\hline 
\end{tabular}
\caption{Convergence rates for $W(x)=\hf |x|^2-|x|$ with initial data \eqref{eq:smooth_init} at $t=20$.}\label{tab:attrep1d}
\end{table}

The two Lax--Friedrichs type schemes exhibit different convergence rates and steady states than the two upwind type schemes, see Table \ref{tab:attrep1d} and Figure \ref{fig:attrep1d}. Considering the convergence rates in Table \ref{tab:attrep1d}, the upwind schemes are superior to the LxF schemes. The 1st upw scheme converges towards the steady state at a rate close to 2 and the 2nd upw scheme at a rate between 2 and 3, whereas 1st LxF converges at a rate of 1 and 2nd LxF at a rate of $1.33$. But, oscillations can be observed in both upwind schemes (see Figure \ref{fig:attrep1d}(b)), more so in the first-order scheme than in the second-order one. Oscillations are not observed for the LxF schemes, see Figure \ref{fig:attrep1d}(a) (although the 2nd LxF solution contains overshoots). The oscillations in the upwind schemes perturb very little mass compared to the LxF schemes, which explains why the upwind approximations are better approximations to $0.5\chi_{[-1,-1]}$ in the $d_1$ sense.

\begin{figure}
\centering
\subfigure[Lax--Friedrichs type schemes]{
\includegraphics[width=0.42\textwidth]{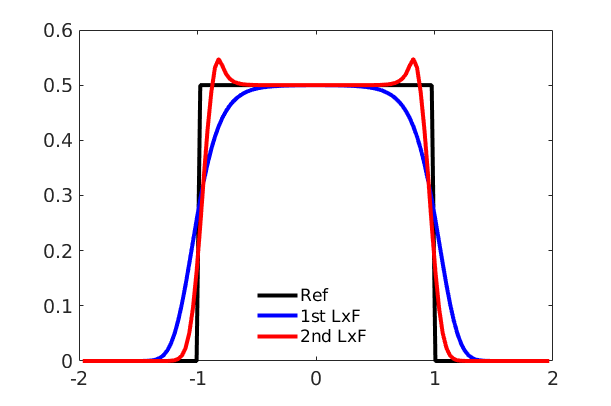}}
\subfigure[Upwind type schemes]{
\includegraphics[width=0.39\textwidth]{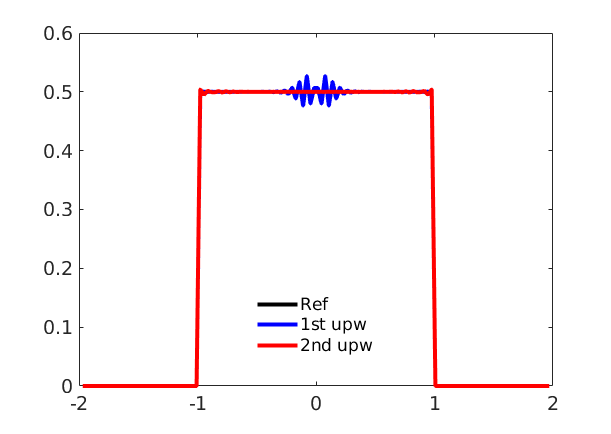}}
\caption{The four numerical schemes approximating $0.5\chi_{[-1,-1]}$ with $128$ cells for $W(x)=\hf |x|^2-|x|$ at $t=20$.}\label{fig:attrep1d}
\end{figure}

Next, we consider a potential that is fully covered by Theorem \ref{thrm:main}, $W(x)=\unitfrac{1}{3}|x|^3-\unitfrac{1}{2}|x|^2$. This potential is related to the scaled granular media equation studied in \cite{BCP,CV2002,CMV2003,CMV2006} for which the convergence as time goes to $\infty$ towards the homogeneous cooling state, whose profile is given by two Diracs located symmetrically about the center of mass separated by length 1, is known. We divide the interval $[-1,1]$ into 256 cells and consider the initial data \eqref{eq:cosinit}, see Figure \ref{fig:simulation1D}(a). As Figure \ref{fig:x3-x2} depicts, both the 2nd LxF and the 2nd upw scheme converge to the expected stationary solution.
\begin{figure}[ht]
\centering
\subfigure[2nd LxF]{\includegraphics[width=0.38\textwidth]{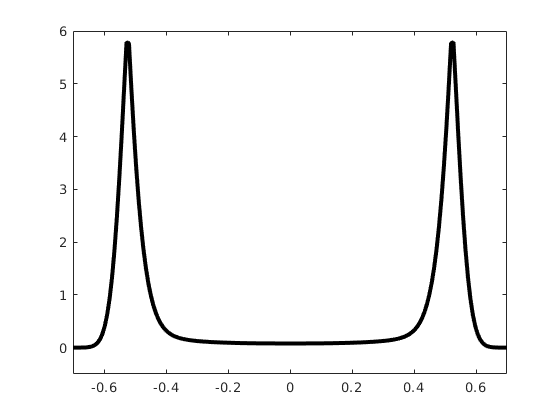}}
\subfigure[2nd upw]{\includegraphics[width=0.38\textwidth]{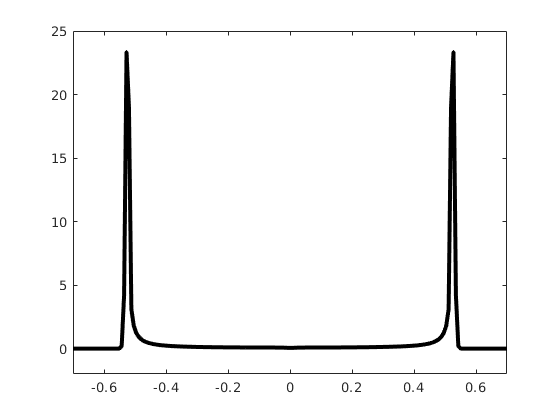}}
\caption{Numerical simulations of \eqref{eq:numericalmethod_1d} with $W(x)=\unitfrac{1}{3}|x|^3-\unitfrac{1}{2}|x|^2$ at $t=50$.}
\label{fig:x3-x2}
\end{figure}

Lastly, we study the 1D numerical method \eqref{eq:numericalmethod_1d} with a potential that is more singular than the ones satisfying \ref{cond:lip}, \ref{cond:c1}, and is therefore not covered by the theory in this paper, $W(x)=\unitfrac{1}{2}|x|^2-\log |x|$. Even though $W$ is more singular in this example than in the previous one, the solution is expected to converge to a steady state that is more regular, the half-ellipse $\sqrt{2-x^2}/\pi$, as can be seen from the results in \cite{CFP}. The initial data and the numerical solution using the 2nd LxF scheme \eqref{eq:scheme_1d}, \eqref{eq:diffusiondef} at $t=20$ are depicted in Figure \ref{fig:xx-log}, where the numerical solution clearly resembles the halfcircle. 
The 2nd upw scheme \eqref{eq:scheme_1d}, \eqref{eq:upwinddef} does not perform well in this case, with severe oscillations. This and the results above suggest that to get a qualitatively good numerical approximation, one has to choose a flux depending on the type of solution that one expects.

\begin{figure}[ht]
\centering
\subfigure[Initial data]{\includegraphics[width=0.38\textwidth]{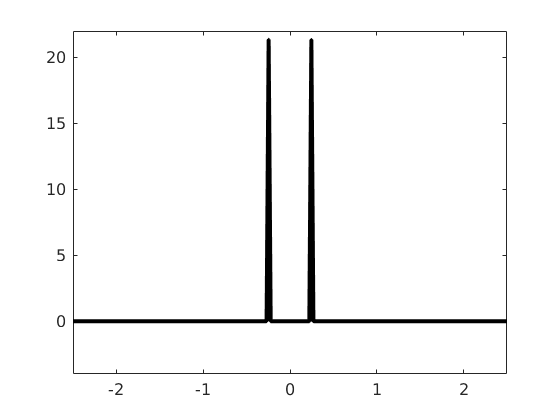}}
\subfigure[Solution at $t=20$]{\includegraphics[width=0.38\textwidth]{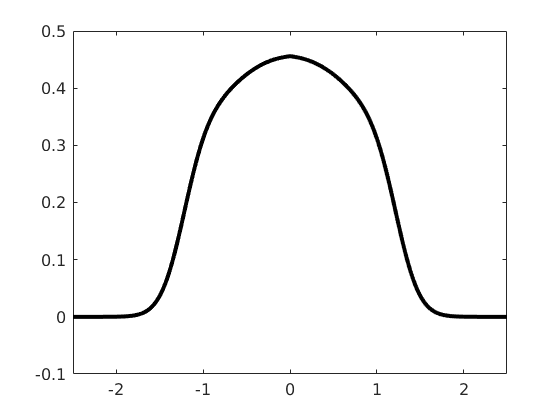}}
\caption{Convergence to a half-ellipse for 2nd LxF with $W(x)=\unitfrac{1}{2}|x|^2-\log (|x|)$.}
\label{fig:xx-log}
\end{figure}

It can be observed in the last three figures that the numerical method preserves the center of mass, as expected due to Lemma \ref{lem:cfl} {\it (v)}.

\subsection{Experiments in 2D}\label{sub:2D} 
We test and compare \eqref{eq:laxfr2D} to 1st LxF and 1st upw. In all numerical experiments in this section the CFL number is set to $0.2$, $c^n = \|W\|_\Lip$ and $\Dx=\Dy$, and the grid is split into $256\times256$ cells unless otherwise stated. Heun's method is used to integrate in time. Let 
\begin{align}\label{eq:oneblob}
b(x,y,x_0,y_0, C) := \exp \bigl(-C(x-x_0)^2-C(y-y_0)^2\bigr).
\end{align}
We will consider two initial data: one ``blob'' $\rho^0(x,y)=\frac{1}{M} b(x,y,x_0,y_0, 10)$ centered at $(x_0,y_0)$, where $b$ is defined in \eqref{eq:oneblob} (see Figure \ref{fig:initial2D}(a)), and three ``blobs'' 
\begin{equation}\label{eq:threeblobs}
\begin{split}
\rho^0=\frac{1}{M}\Big(b\big(x,y,\unitfrac{1}{4}, \unitfrac{1}{3}, 100\big) + b\big(x,y,0.8,0.7,100\big) + 0.9b\big(x,y,0.4,0.6,100\big)\Big),
\end{split}
\end{equation} 
see Figure \ref{fig:initial2D}(b). The constant $M$ normalizes the mass of $\rho^0$ to 1 in each case.
\begin{figure}[ht!]
\centering
\subfigure[One blob centered at $(1,1)$.]{
\includegraphics[trim={0 4em 0 4em}, clip, width=0.35\textwidth]{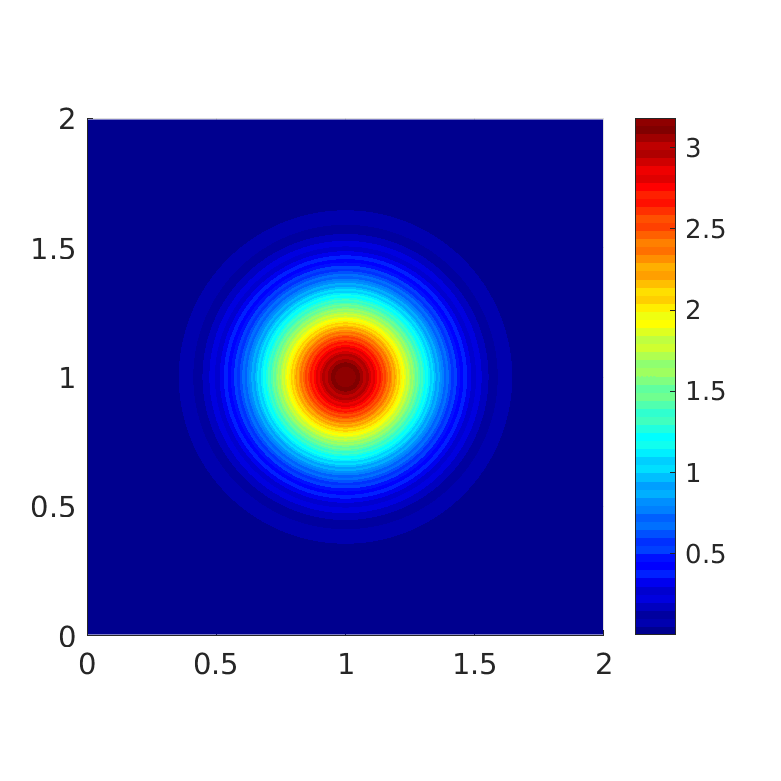}}
\subfigure[Three blobs.]{
\includegraphics[trim={0 4em 0 4em}, clip, width=0.35\textwidth]{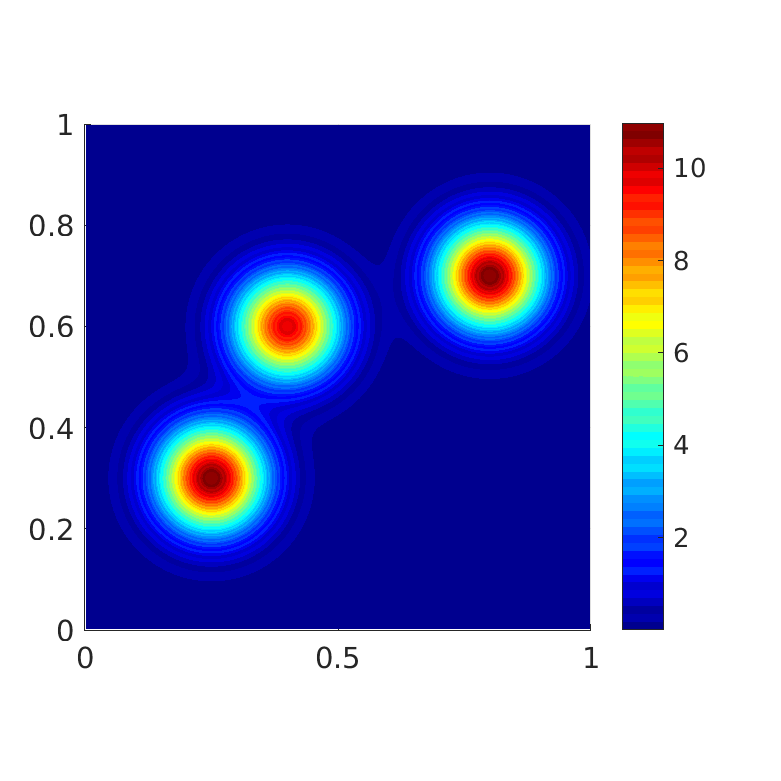}}
\caption{Initial data.}\label{fig:initial2D}
\end{figure}

\begin{figure}[ht!]
\centering
\subfigure[$t=1M$]{\hspace*{-2em}
\includegraphics[trim={3em 0 0 0}, clip, width=0.29\textwidth]{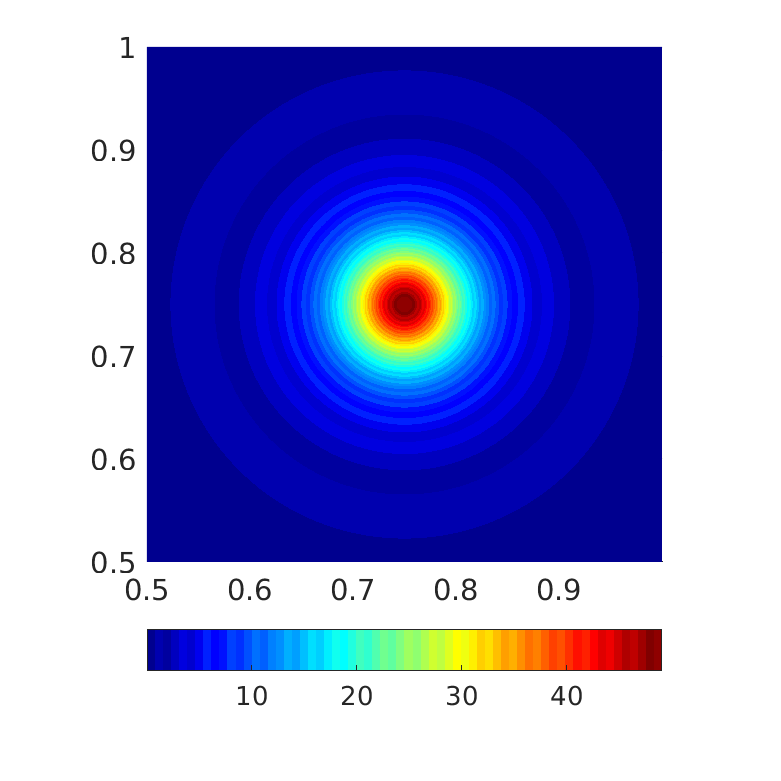}\hspace*{-2.1em}
\includegraphics[trim={3em 0 0 0}, clip, width=0.29\textwidth]{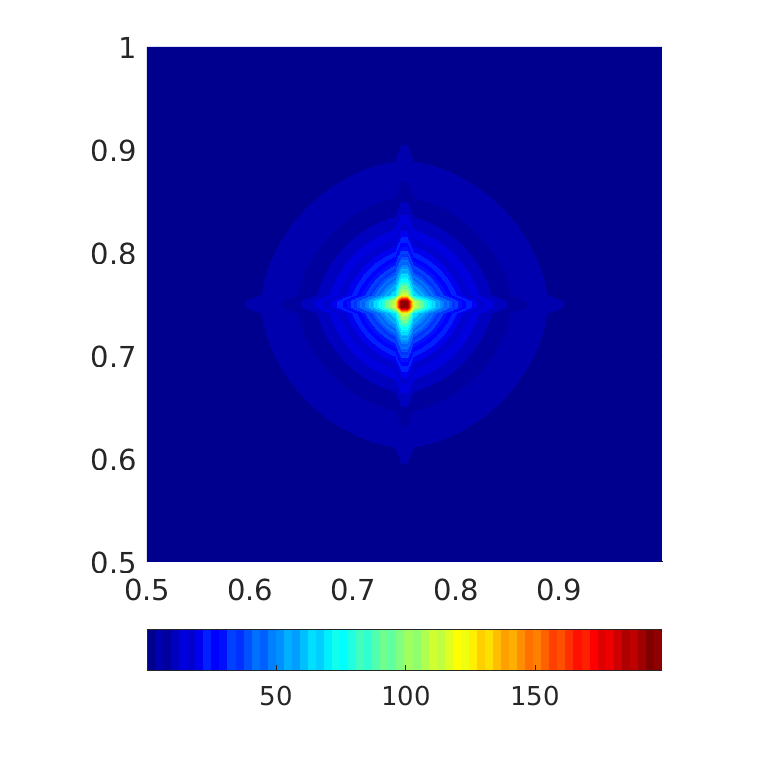}\hspace*{-2.1em}
\includegraphics[trim={3em 0 0 0}, clip, width=0.29\textwidth]{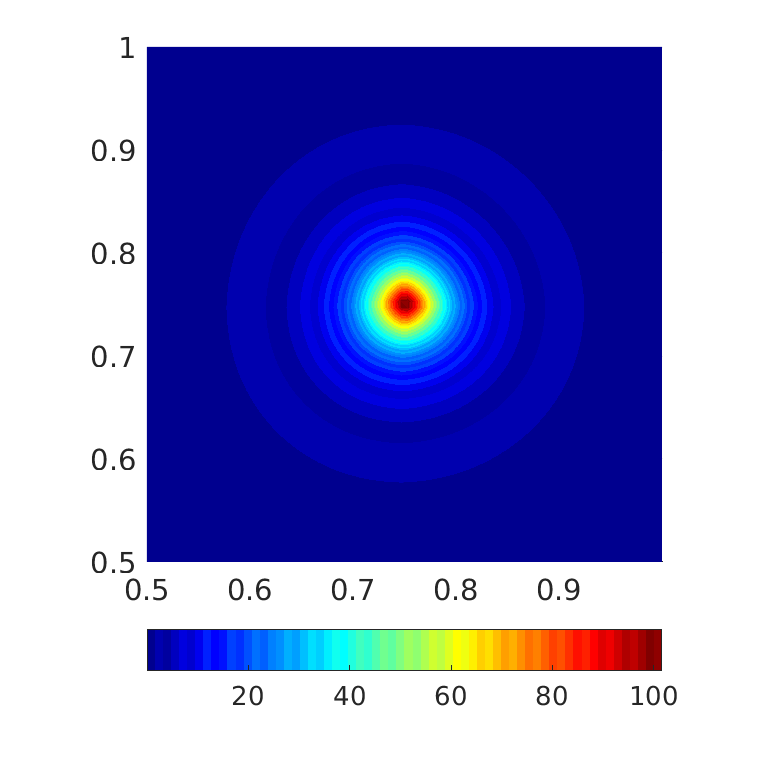}\hspace*{-2.1em}
\includegraphics[trim={3em 0 0 0}, clip, width=0.29\textwidth]{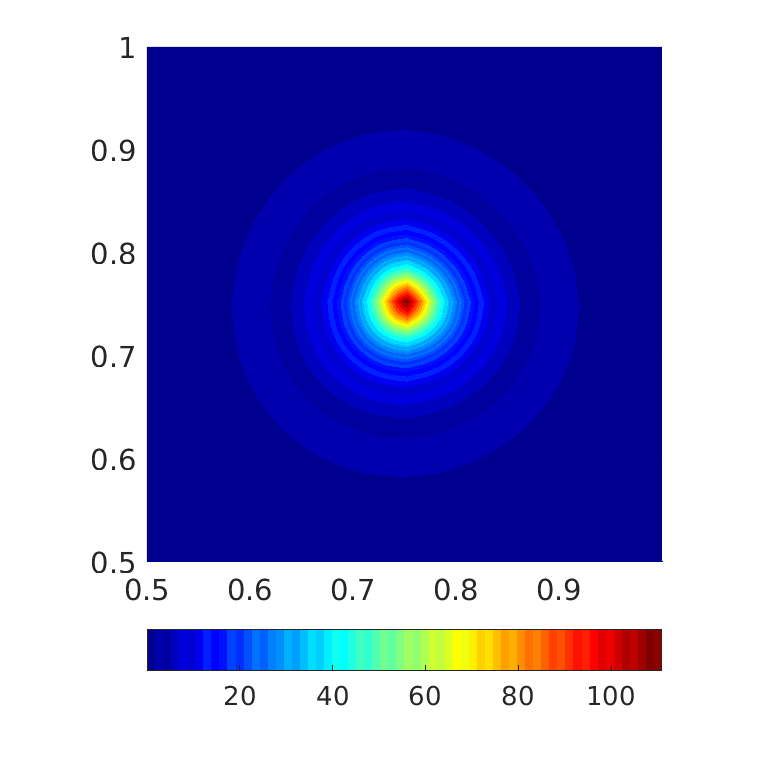}\hspace*{-2.1em}}
\subfigure[$t=2.5M$]{\hspace*{-2em}
\includegraphics[trim={3em 0 0 0}, clip, width=0.29\textwidth]{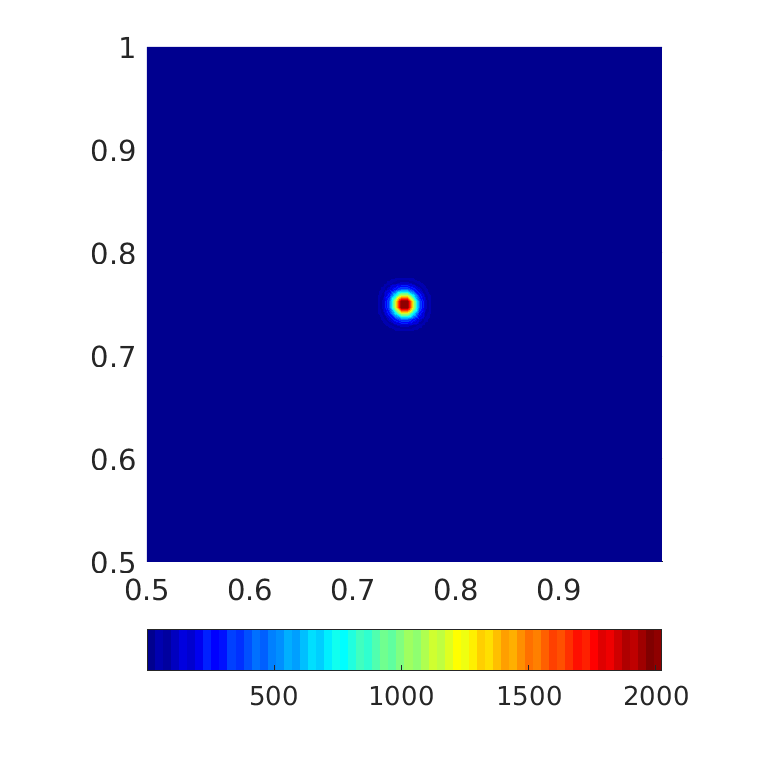}\hspace*{-2.1em}
\includegraphics[trim={3em 0 0 0}, clip, width=0.29\textwidth]{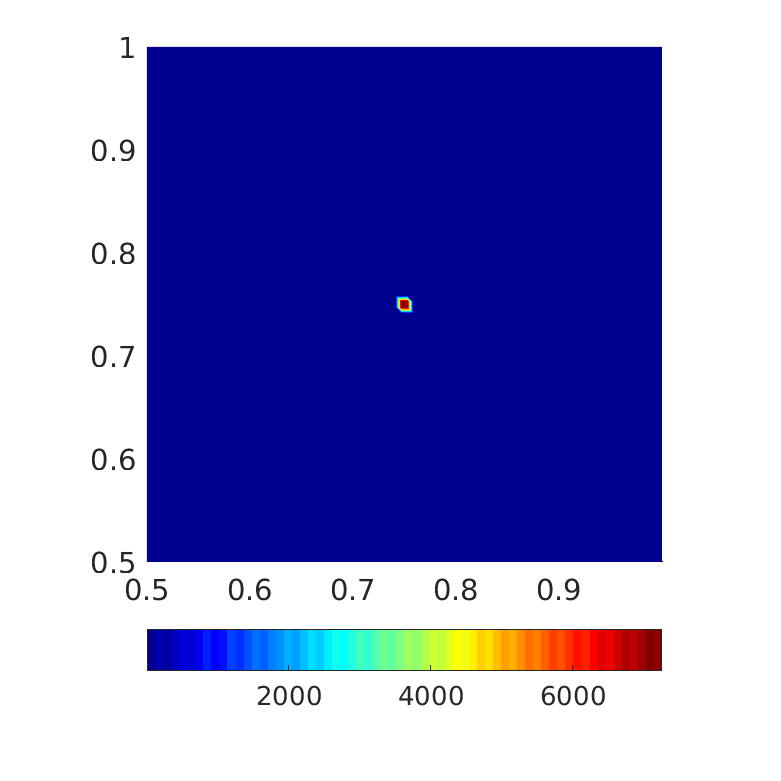}\hspace*{-2.1em}
\includegraphics[trim={3em 0 0 0}, clip, width=0.29\textwidth]{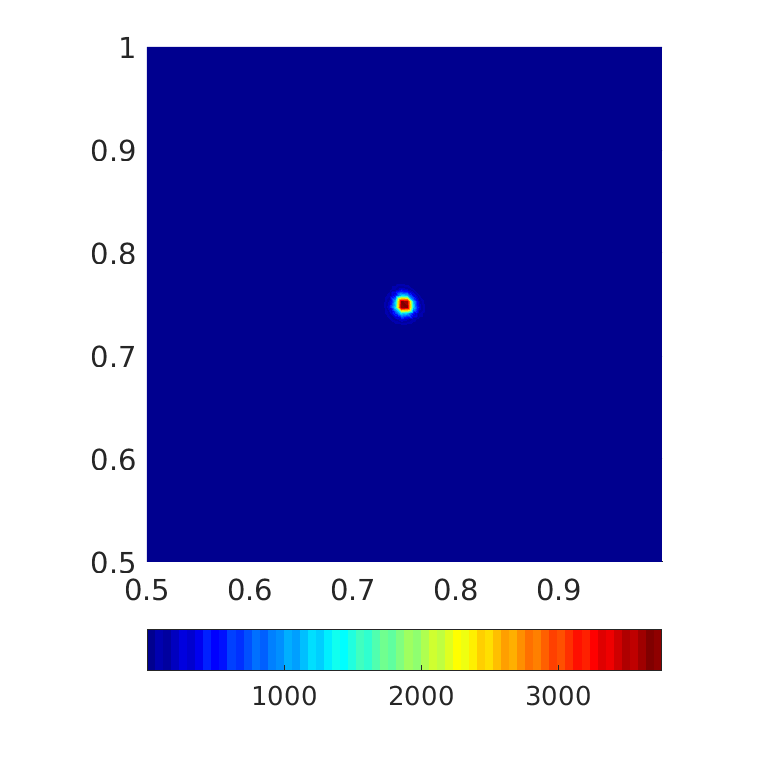}\hspace*{-2.1em}
\includegraphics[trim={3em 0 0 0}, clip, width=0.29\textwidth]{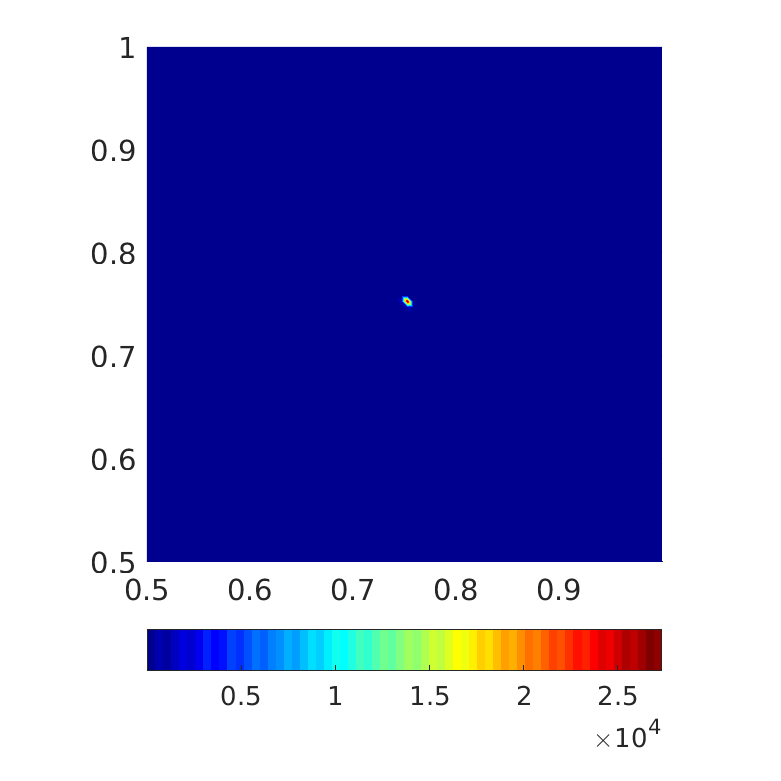}\hspace*{-2.1em}}
\caption{Comparison of all four schemes with $W(\bx)=|\bx|$ and one blob centered at $\bx_0=\big(\unitfrac{3}{4},\unitfrac{3}{4}\big)$ as initial data. From left to right: 1st order LxF, 1st order upwind, 2nd order LxF, 2nd order upwind. The normalization factor is $M=0.3137$.}\label{fig:oneblobsim}
\end{figure}

\subsubsection{Attractive potential}
We start with the simple case of one blob as initial data, $$\rho^0(x,y)=\frac{1}{M}b(x,y,0.75,0.75,10),$$ and study the dynamics for the potential $W(\bx)=|\bx|$, see Figure \ref{fig:oneblobsim}. As expected, in all four cases the mass of the blob has aggregated into a (very) small area in Figure \ref{fig:oneblobsim}(b), and the LxF schemes are more diffusive than the upwind schemes. Also, the second order schemes and the 1st LxF scheme exhibit radially symmetric solutions, see Figure \ref{fig:oneblobsim}(a), but the 1st upw scheme is only axially symmetric. 

\begin{figure}[ht!]
\centering
\subfigure[1st LxF]{
\includegraphics[trim={3em 0 0 0}, clip, width=0.29\textwidth]{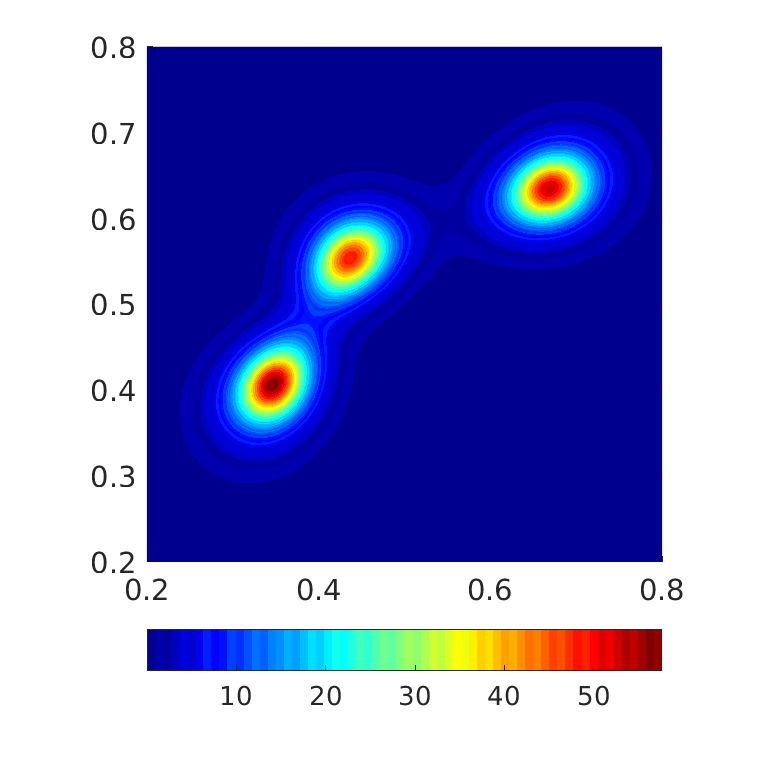}\hspace*{-1.5em}
\includegraphics[trim={3em 0 0 0}, clip, width=0.29\textwidth]{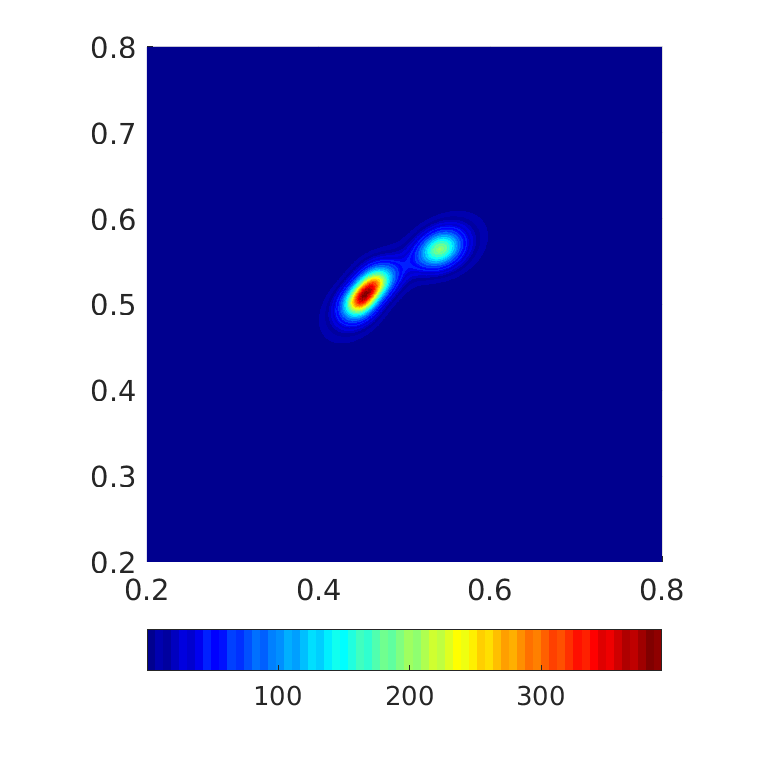}\hspace*{-1.5em}
\includegraphics[trim={3em 0 0 0}, clip, width=0.29\textwidth]{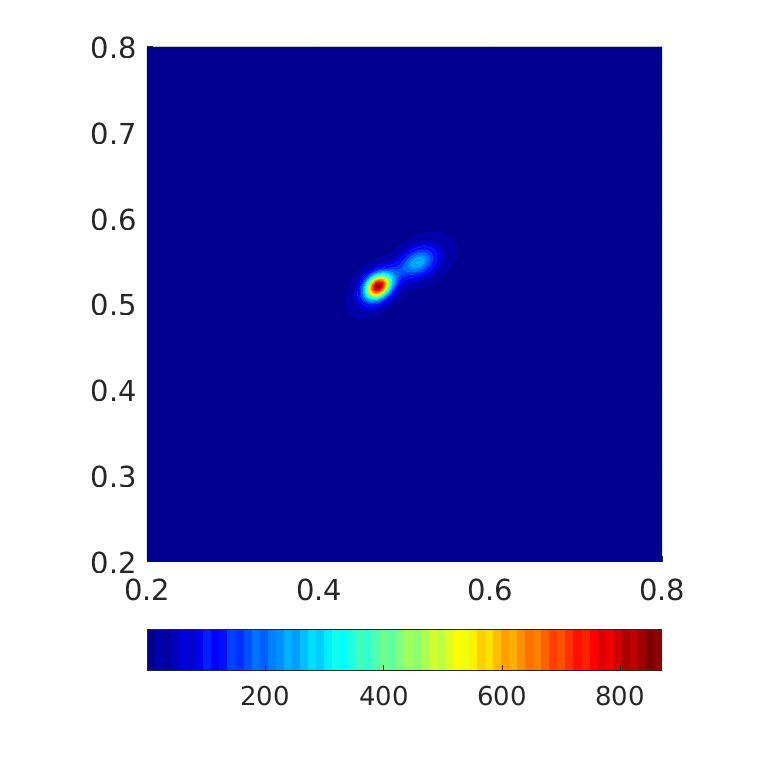}\hspace*{-1.5em}
\includegraphics[trim={3em 0 0 0}, clip, width=0.29\textwidth]{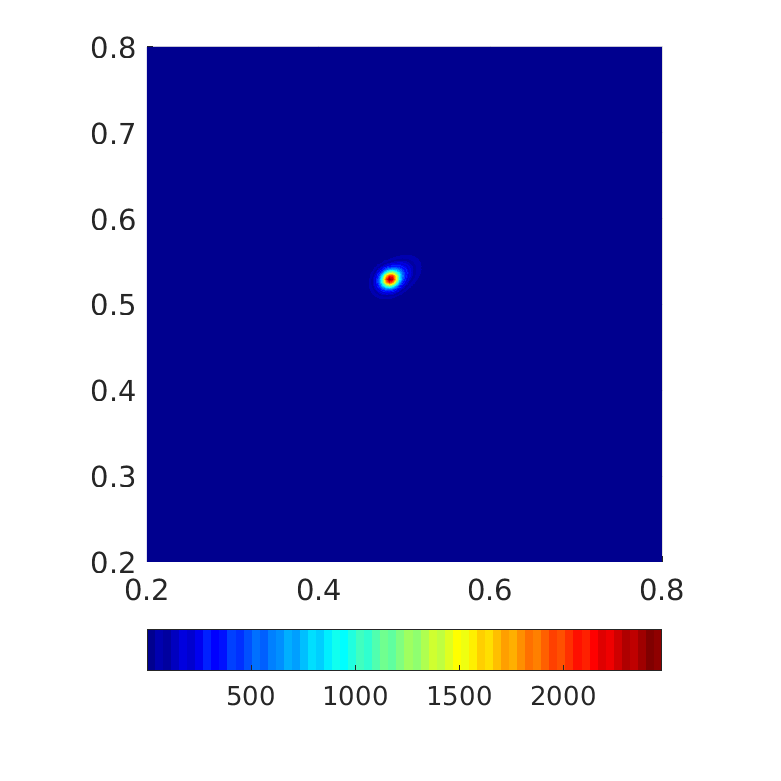}\hspace*{-1.5em}}
\subfigure[1st upw]{
\includegraphics[trim={3em 0 0 0}, clip, width=0.29\textwidth]{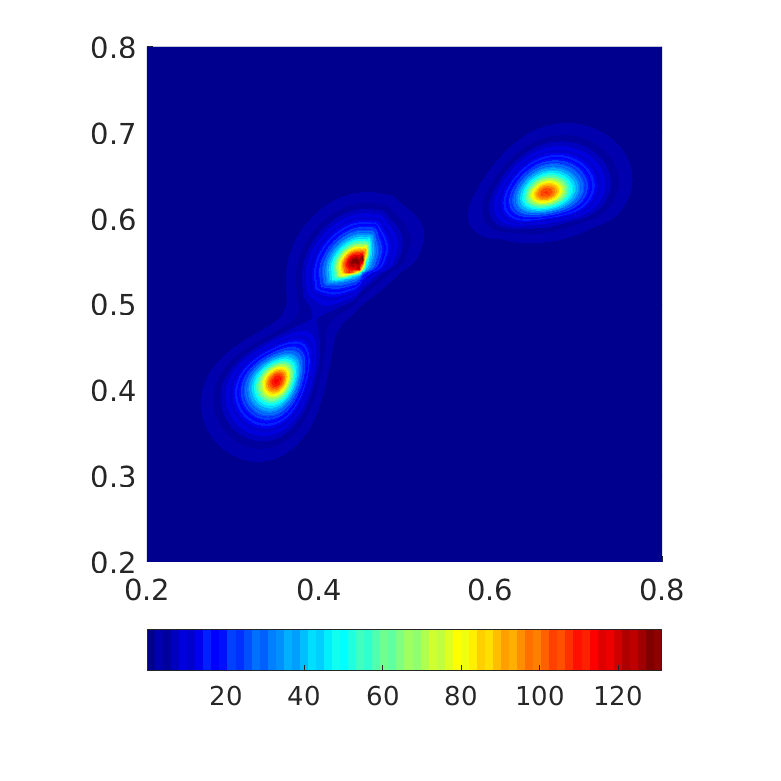}\hspace*{-1.5em}
\includegraphics[trim={3em 0 0 0}, clip, width=0.29\textwidth]{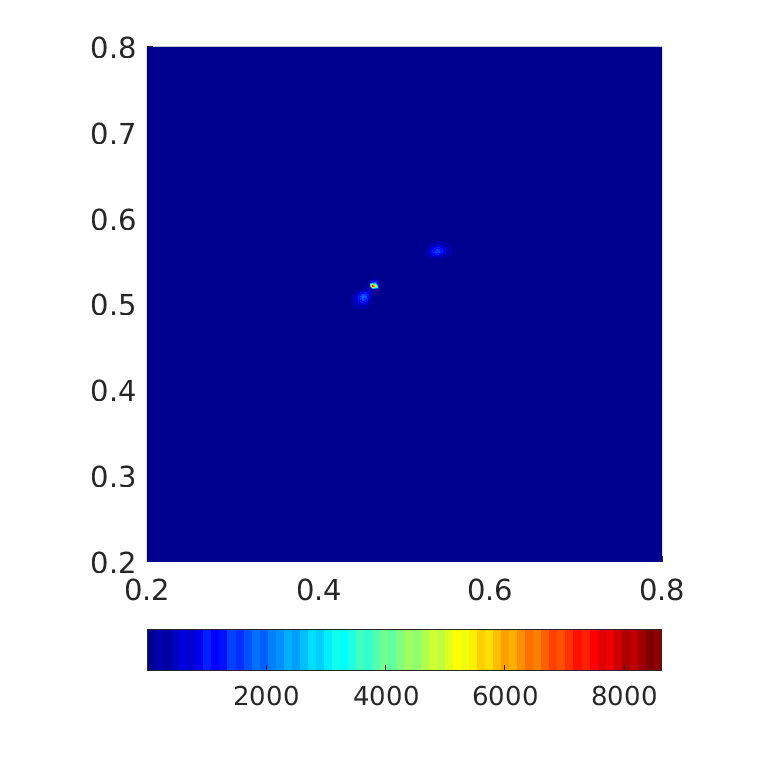}\hspace*{-1.5em}
\includegraphics[trim={3em 0 0 0}, clip, width=0.29\textwidth]{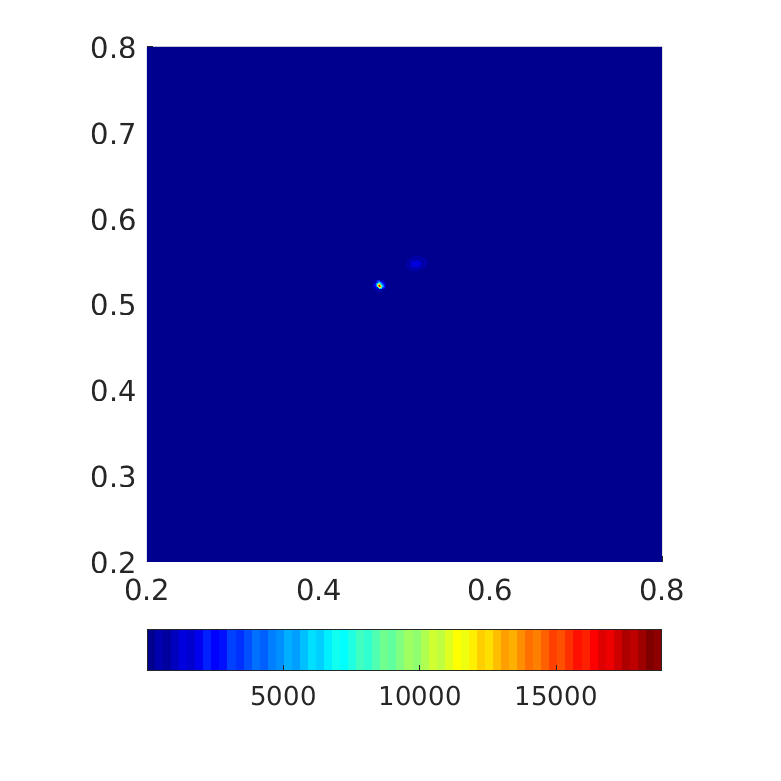}\hspace*{-1.5em}
\includegraphics[trim={3em 0 0 0}, clip, width=0.29\textwidth]{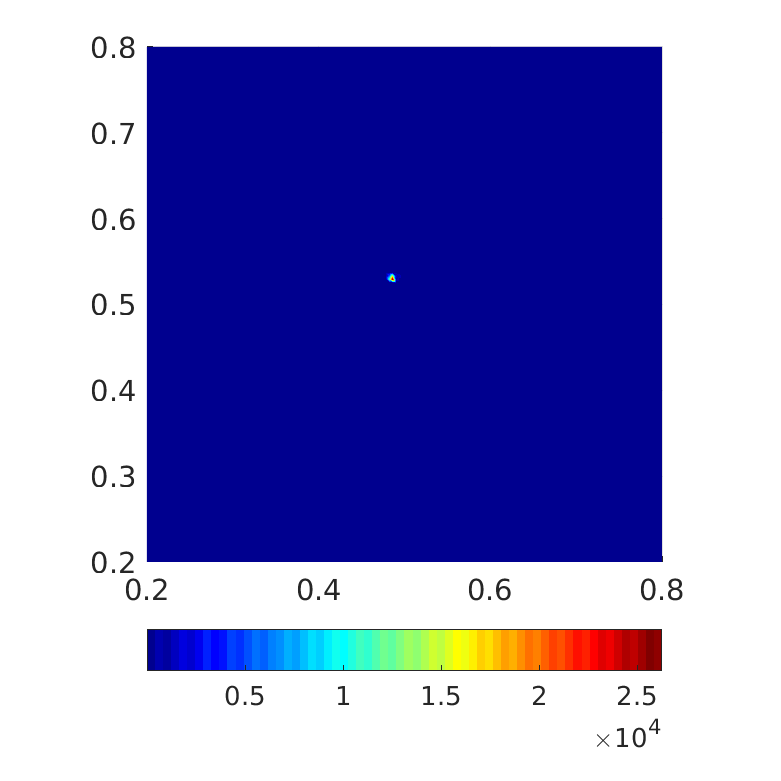}\hspace*{-1.5em}}
\subfigure[2nd LxF]{
\includegraphics[trim={3em 0 0 0}, clip, width=0.29\textwidth]{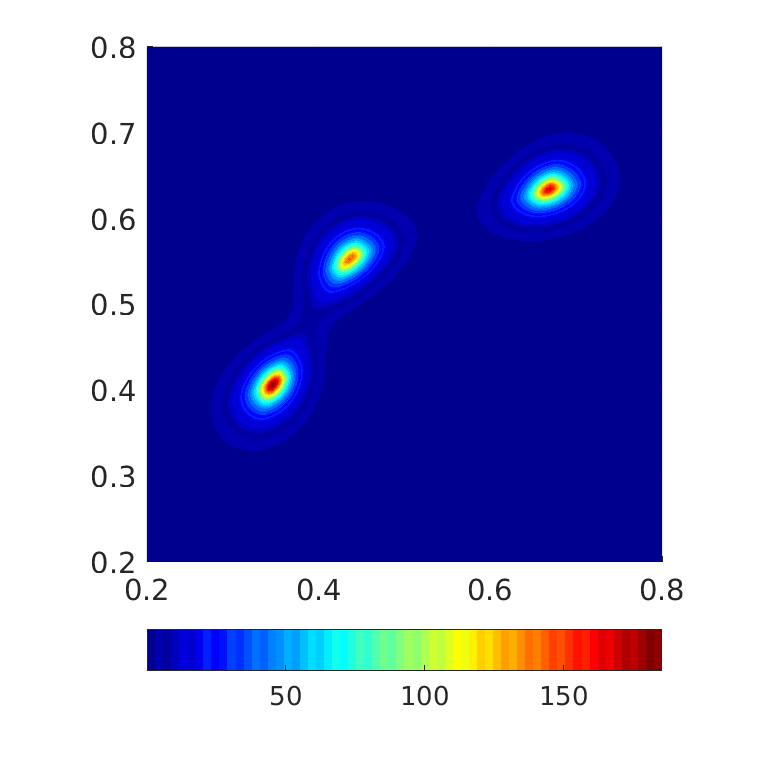}\hspace*{-1.5em}
\includegraphics[trim={3em 0 0 0}, clip, width=0.29\textwidth]{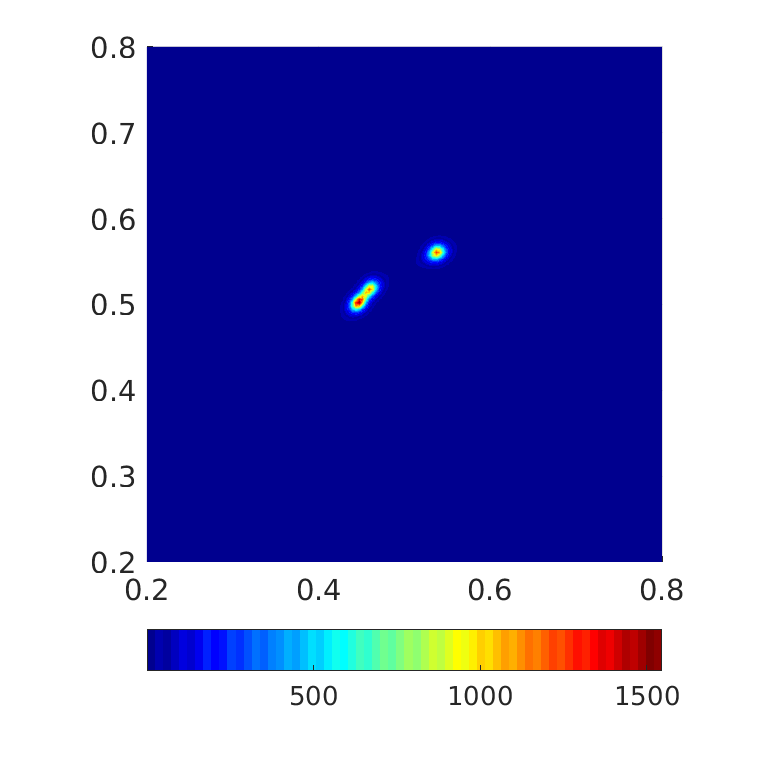}\hspace*{-1.5em}
\includegraphics[trim={3em 0 0 0}, clip, width=0.29\textwidth]{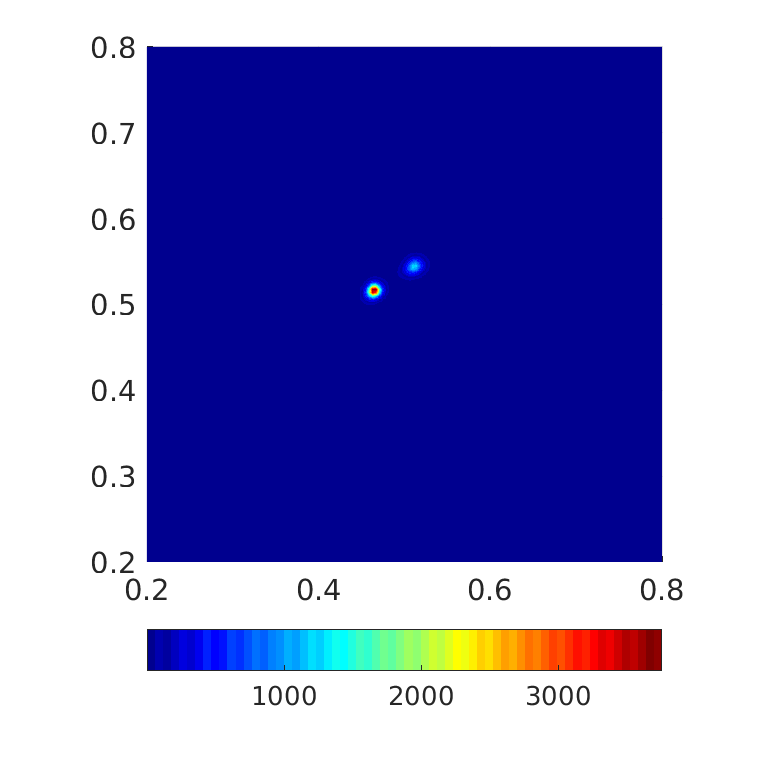}\hspace*{-1.5em}
\includegraphics[trim={3em 0 0 0}, clip, width=0.29\textwidth]{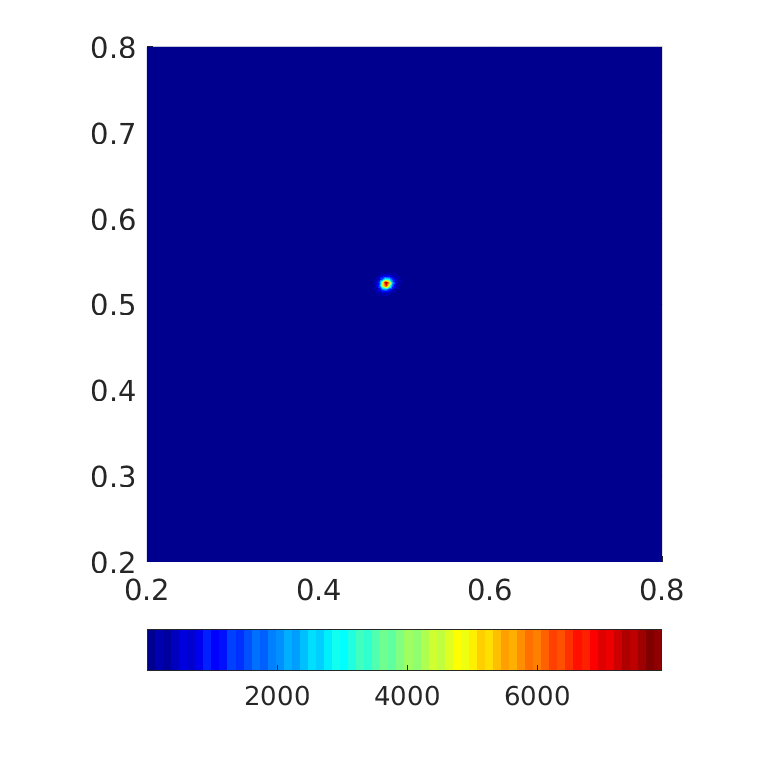}\hspace*{-1.5em}}
\subfigure[2nd upw]{
\includegraphics[trim={3em 0 0 0}, clip, width=0.29\textwidth]{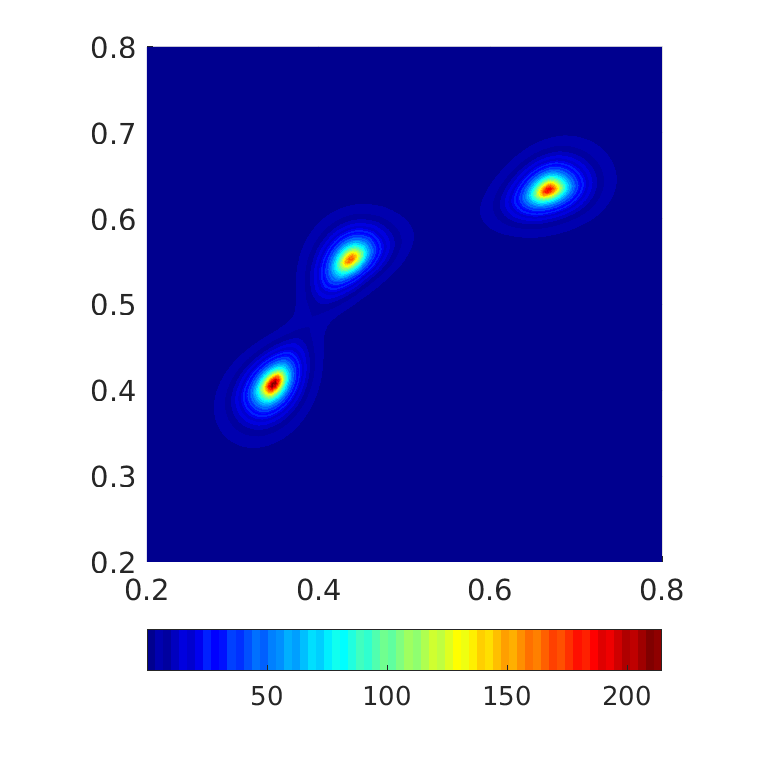}\hspace*{-1.5em}
\includegraphics[trim={3em 0 0 0}, clip, width=0.29\textwidth]{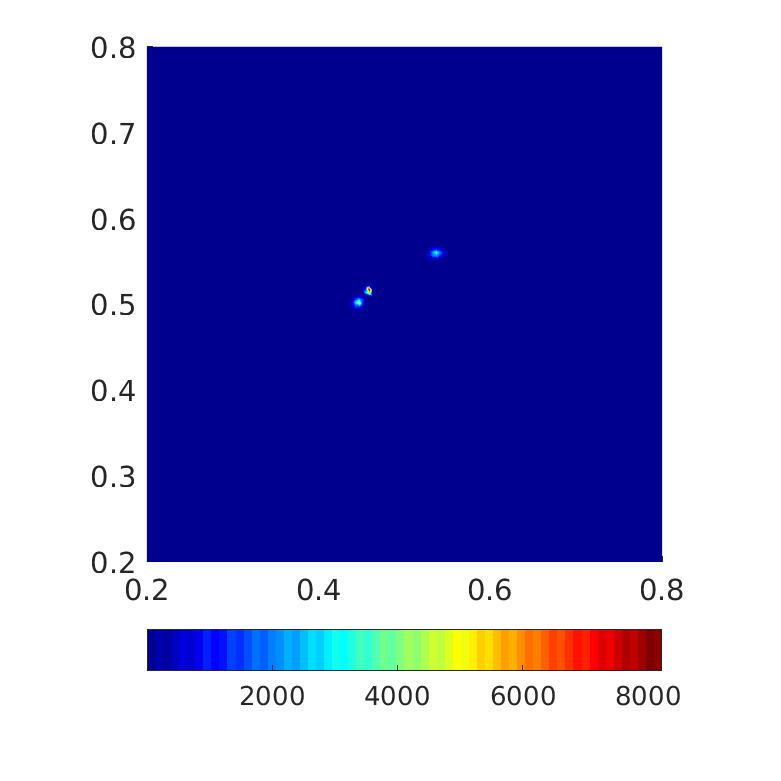}\hspace*{-1.5em}
\includegraphics[trim={3em 0 0 0}, clip, width=0.29\textwidth]{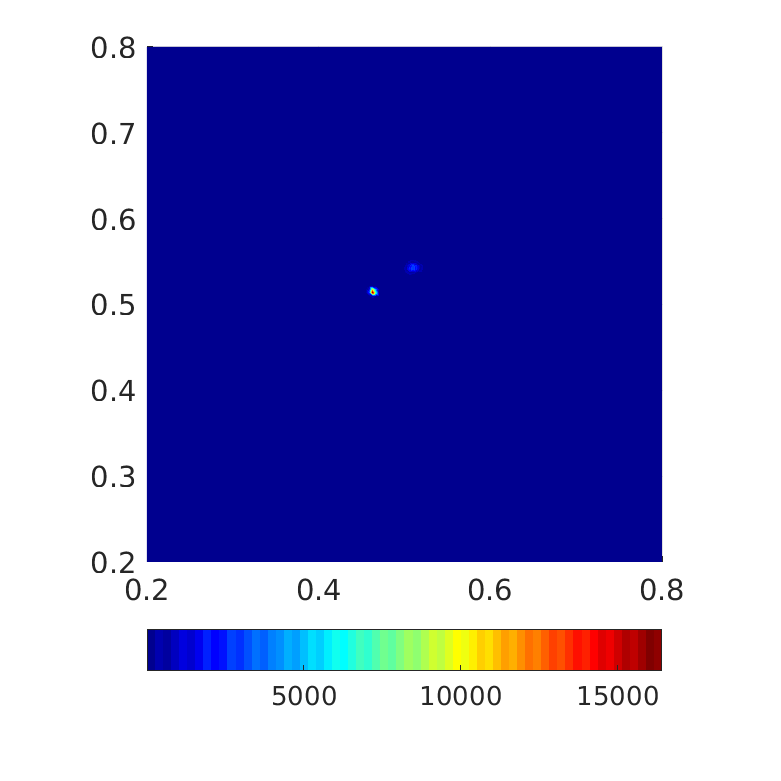}\hspace*{-1.5em}
\includegraphics[trim={3em 0 0 0}, clip, width=0.29\textwidth]{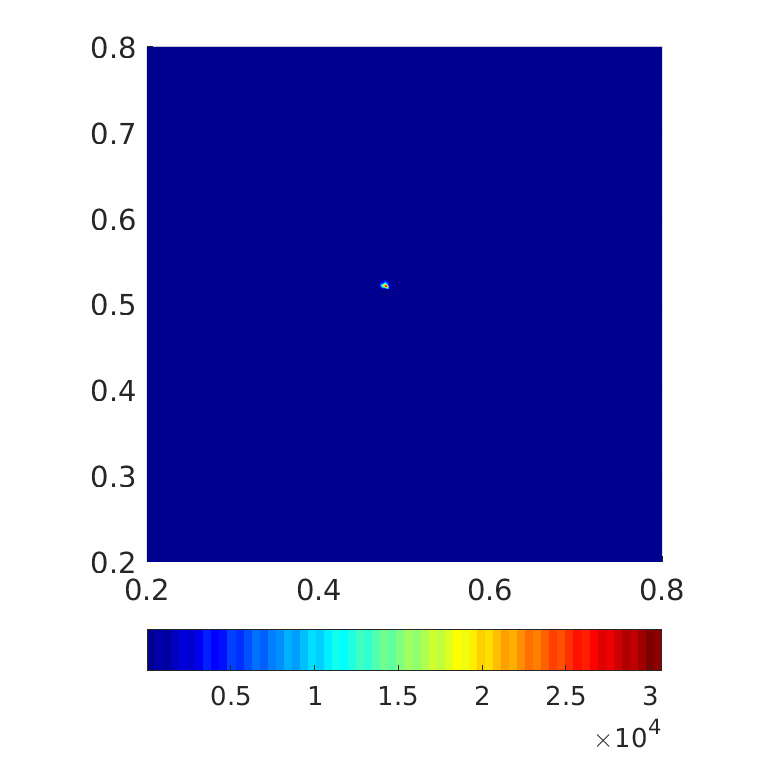}\hspace*{-1.5em}}
\caption{Comparison of all four schemes with $W(\bx)=|\bx|$ and initial data \eqref{eq:threeblobs}. From left to right: $t=2.5M,5M,5.5M,6M$. The normalization factor is $M=0.0910$.}\label{fig:2Dmodulus}
\end{figure}

To numerically verify that the 2D scheme indeed satisfies Lemma \ref{lem:2Dtrunc}, we calculate the convergence rates of the four schemes in the case of the interaction potential $W(x,y) = \sqrt{x^2+xy+y^2}$. This particular potential is chosen in order to highlight the fact that $\nabla W(-\bx) = - \nabla W(\bx)$ is sufficient for Lemma \ref{lem:2Dtrunc} to hold. The 2D Monge--Kantorovich distance is calculated by using the optimal transport algorithm in \cite{multiOT, EMD}. As initial data we choose \eqref{eq:oneblob} with $C=36$ on $[0,2]^2$. The reference solutions are computed with the respective numerical schemes on a $2^{11} \times 2^{11}$ grid. As can be seen in Table \ref{tab:smooth2D}, the convergence rates are close to two for both 2nd LxF and 2nd upw.

\begin{table}[ht!]
\centering
\small 
\begin{tabular}{|c|l|c|l|c|l|c|l|c|}
\cline{2-9}
\multicolumn{1}{c}{} & \multicolumn{2}{|c|}{1st LxF} & \multicolumn{2}{|c|}{1st upw} & \multicolumn{2}{|c|}{2nd LxF} & \multicolumn{2}{|c|}{2nd upw} \\
\hline 
$ n $   & $d_1$          & OOC       & $d_1$          & OOC       & $d_1$          & OOC       &  $d_1$ & OOC \\ 
 \hline 
$ 16 $  & $ 7.84e-03 $&          & $ 4.03e-03 $  &          & $ 3.90e-03 $ &          & $ 5.57e-03 $ & \\ 
$ 32 $  & $ 2.62e-03 $& $ 1.58 $ & $ 1.17e-03 $  & $ 1.78 $ & $ 8.39e-04 $ & $ 2.21 $ & $ 1.23e-03 $ & $ 2.18 $ \\ 
$ 64 $  & $ 8.10e-04 $& $ 1.69 $ & $ 1.12e-03 $  & $ 0.07 $ & $ 1.02e-03 $ & $ -0.27$ & $ 1.28e-03 $ & $ -0.06$ \\ 
$ 128 $ & $ 7.31e-04 $& $ 0.15 $ & $ 2.85e-04 $  & $ 1.97 $ & $ 8.14e-05 $ & $ 3.64 $ & $ 4.58e-05 $ & $ 4.80 $ \\ 
$ 256 $ & $ 3.43e-04 $& $ 1.09 $ & $ 1.37e-04 $  & $ 1.06 $ & $ 2.01e-05 $ & $ 2.02 $ & $ 1.10e-05 $ & $ 2.06 $ \\ 
$ 512 $ & $ 1.46e-04 $& $ 1.23 $ & $ 0.58e-04 $  & $ 1.23 $ & $ 0.36e-05 $ & $ 2.47 $ & $ 0.17e-05 $ & $ 2.66 $ \\ 
\hline 
\end{tabular}
\caption{Convergence rates for $W(x,y) = \sqrt{x^2+xy+y^2}$ with blob initial data ($C=36$) centered at $(1,1)$ on $[0,2]^2$ at time $t=0.075M$.}\label{tab:smooth2D}
\end{table}

We follow up with the aggregation dynamics in the case of initial data \eqref{eq:threeblobs} and the potential $W({\bf x})=|{\bf x}|$. This test case was considered in both \cite{CJLV16} and \cite{DLV2017}. The simulations are presented in Figure \ref{fig:2Dmodulus}. The second-order schemes clearly resolve the solution more sharply than the first-order schemes. But, as the second-order methods require the calculation of twice as many convolutions as the first-order methods in each timestep, the runtimes of the second-order methods are (more than) twice as high as those for the first-order methods. Reducing the resolution of the grid to $202 \times 202$, the runtimes of the second-order methods are lower than those of the first-order methods on a $256\times 256$ grid. Still the second-order methods are sharper than the first-order methods, see Figure \ref{fig:runtime}. We conclude that the computational efficiency of the second-order schemes is higher than that of the first-order schemes.

\begin{figure}
\centering
\subfigure[1st order LxF $256\times256$]{\includegraphics[trim={0 4em 0 4em}, clip, width=0.32\textwidth]{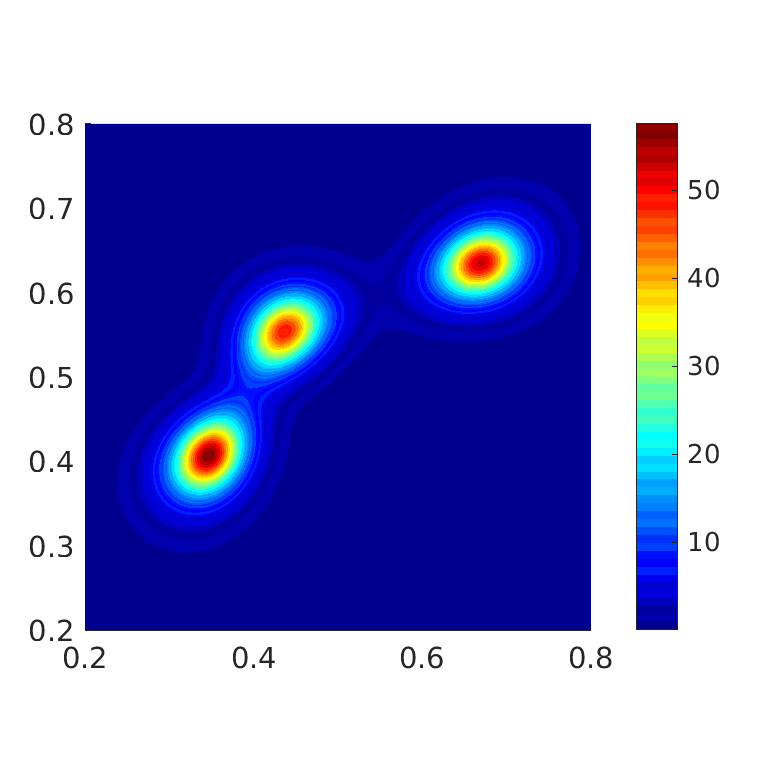}}
\subfigure[2nd order LxF $202\times 202$]{\includegraphics[trim={0 4em 0 4em}, clip, width=0.32\textwidth]{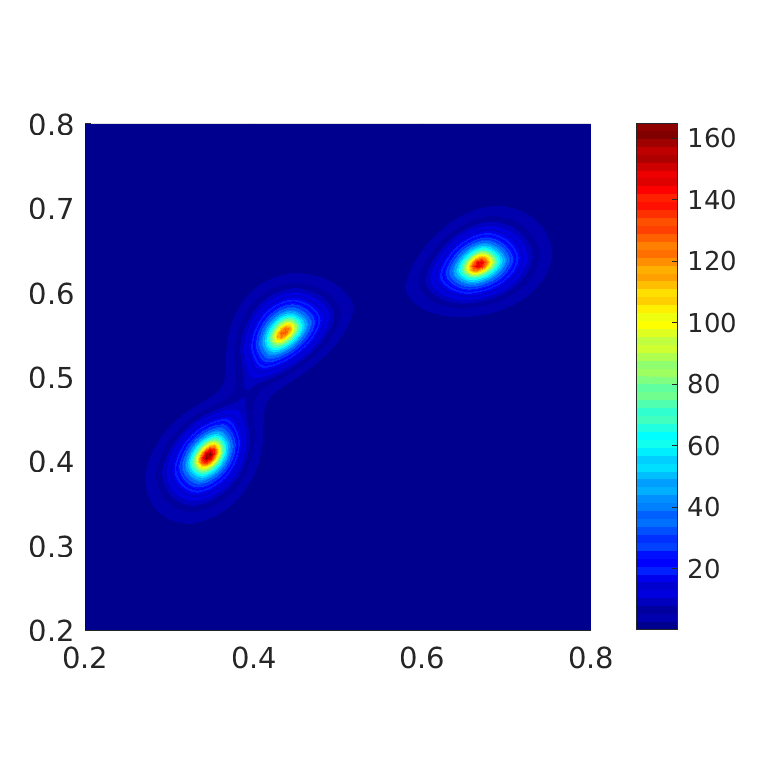}}\\
\subfigure[1st order upwind $256\times256$]{\includegraphics[trim={0 4em 0 4em}, clip, width=0.32\textwidth]{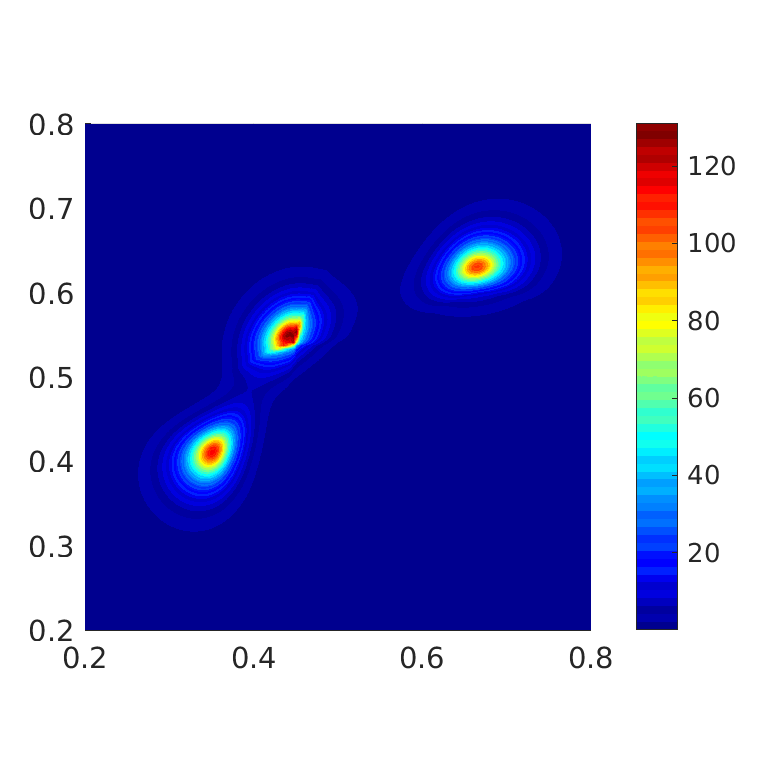}}
\subfigure[2nd order upwind $202 \times 202$]{\includegraphics[trim={0 4em 0 4em}, clip, width=0.32\textwidth]{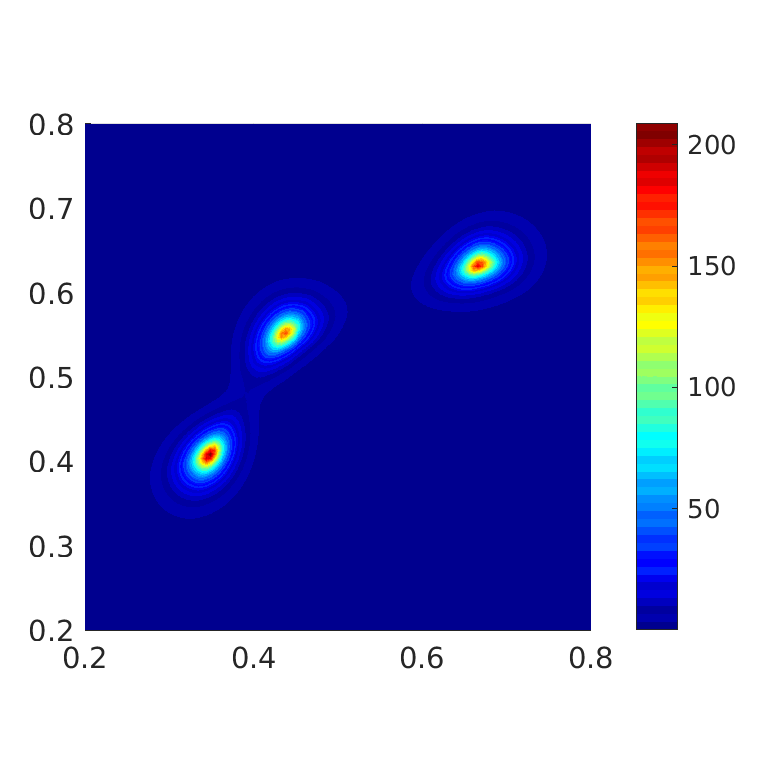}}
\caption{Numerical experiment with $W(\bx)=|\bx|$ at $t=2.5M$. The normalization factor is $M=0.0910$.} \label{fig:runtime}
\end{figure}

\subsubsection{Dissipation of the interaction energy}
After several numerical experiments we observe that the energy \eqref{eq:int_energ} of the second-order numerical scheme developed in this paper seems to be monotonically decreasing over time when $W$ satisfies \ref{cond:lip}--\ref{cond:convex}, both in 1D and 2D. Figure \ref{fig:energy} depicts the decreasing energy for the potentials $W(\bx)=|\bx|$ and $W(\bx)=1-\exp(-5|\bx|)$ with initial data \eqref{eq:threeblobs}. As proven in Section \ref{subsec-energy}, we know that this is almost true for our scheme in the sense that any energy production can be made arbitrarily small. The decay in energy is in accordance with the existing theory for gradient flow solutions of \eqref{eq:agg_eq}. 

\begin{figure}
\centering
\subfigure{
\includegraphics[width=0.3\textwidth]{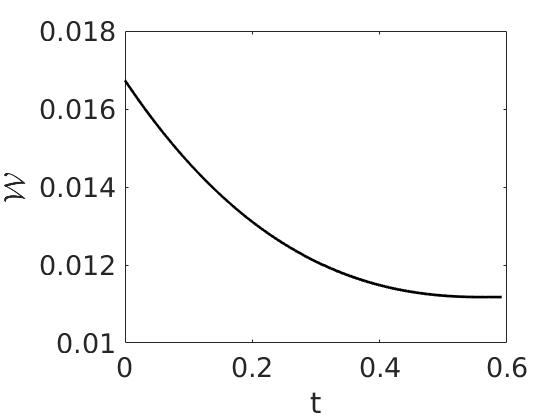}}
\subfigure{
\includegraphics[width=0.3\textwidth]{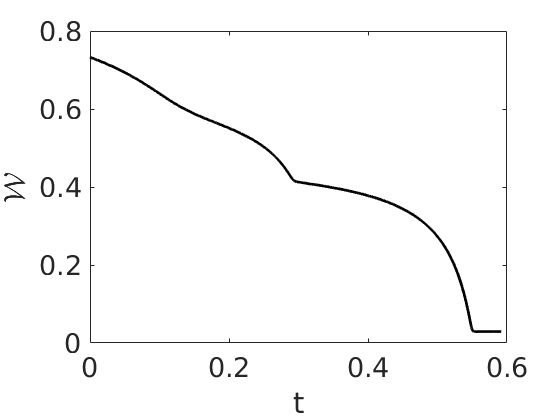}}
\caption{Decay of the energy \eqref{eq:int_energ}. Computed with 2nd LxF on a $256\times 256$ grid. Left: $W(\bx)=|\bx|$. Right: $W(\bx)=1-\exp(-5|\bx|)$.}
\label{fig:energy}
\end{figure}

\subsubsection{Attractive-repulsive potentials}
In this section we only consider initial data consisting of one blob, $\rho^0(x,y)=\frac{1}{M}b(x,y,1,1,10)$. First we study the numerical schemes with the potential $W(\bx)=\unitfrac{1}{4}|\bx|^4-\hf|\bx|^2$, which is fully covered by Theorem \ref{thrm:conv2d}. With this potential, the solution of \eqref{eq:agg_eq} should converge to the uniform distribution on a circle, which we call a $\delta$--ring, of radius $\unitfrac{\sqrt{3}}{3}$ as $t \to \infty$, see \cite{BCLR2011,BCLR2013}. All the numerical schemes form an approximation to a $\delta$--ring, see Figure \ref{fig:attrep2D1}(b), but the 1st upw scheme looks quite different from the others before that point, see Figure \ref{fig:attrep2D1}(a). The schemes form spikes along the circles in Figure \ref{fig:attrep2D1}(b) (with the exception of the 1st LxF scheme), which are plausibly caused by the attempt to approximate a circle on a rectangular grid.

\begin{figure}
\centering
\subfigure[$t=5M$]{
\includegraphics[trim={3em 0 0 0}, clip, width=0.26\textwidth]{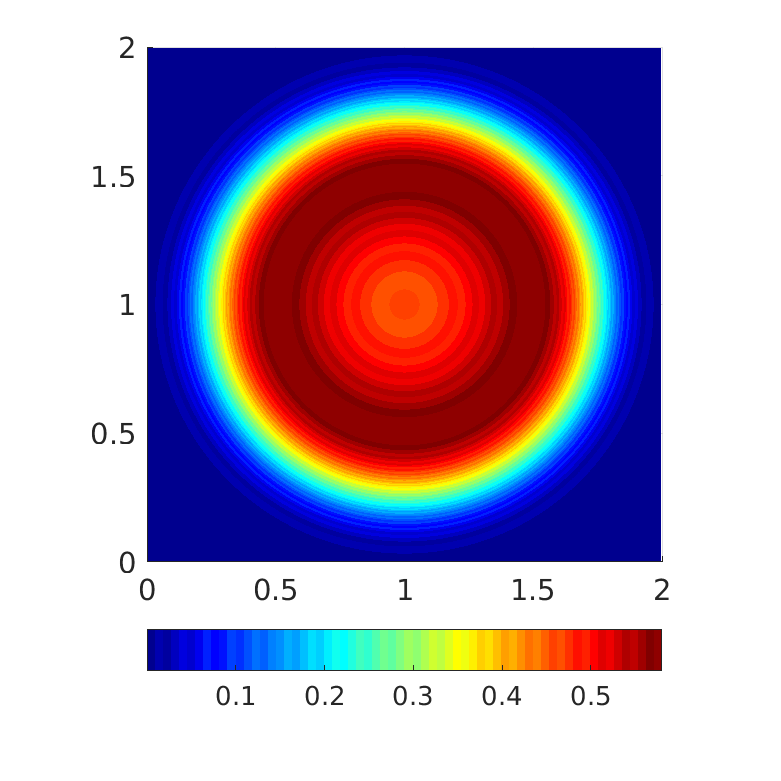}\hspace*{-2em}
\includegraphics[trim={3em 0 0 0}, clip, width=0.26\textwidth]{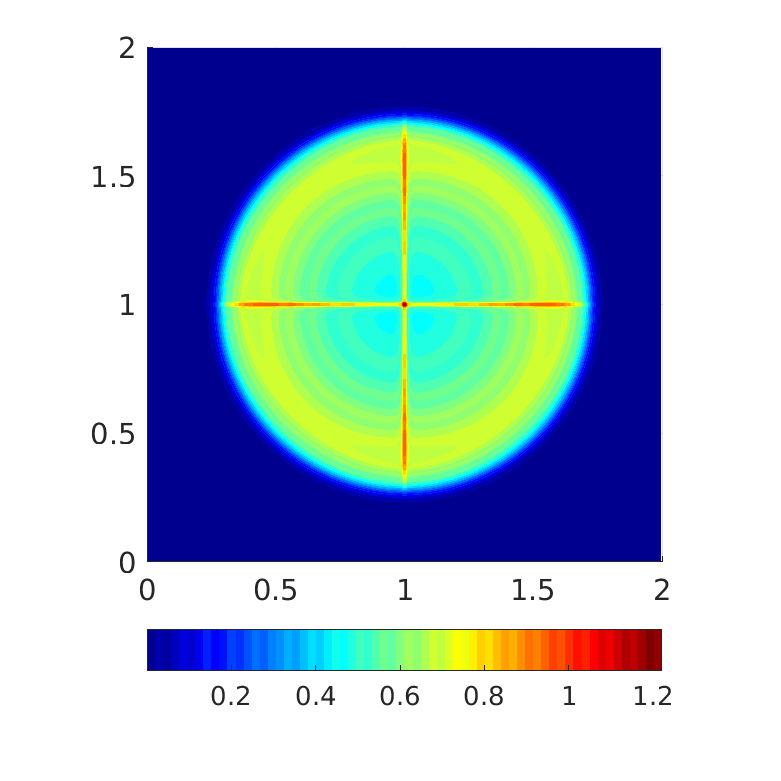}\hspace*{-2em}
\includegraphics[trim={3em 0 0 0}, clip, width=0.26\textwidth]{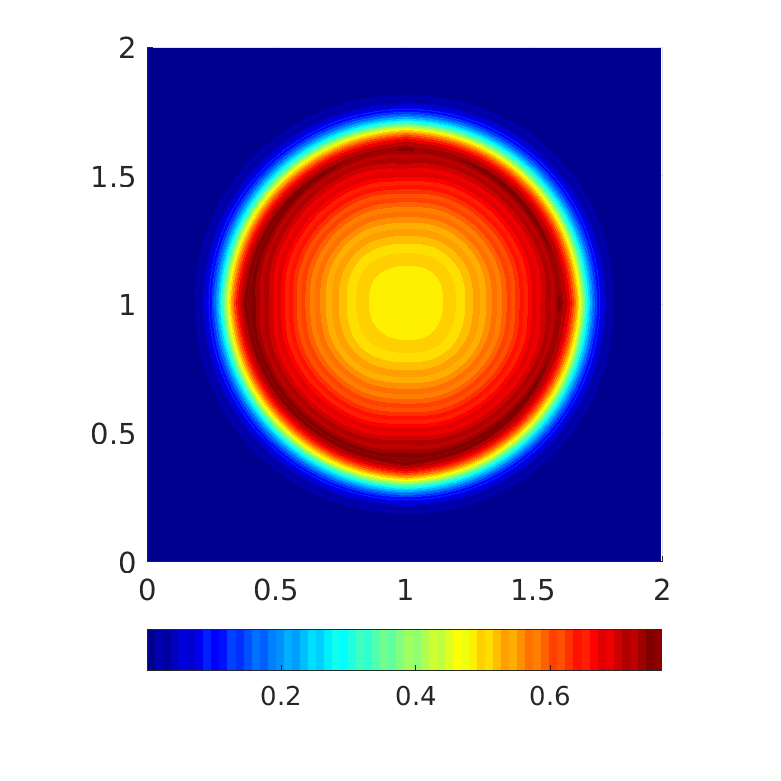}\hspace*{-2em}
\includegraphics[trim={3em 0 0 0}, clip, width=0.26\textwidth]{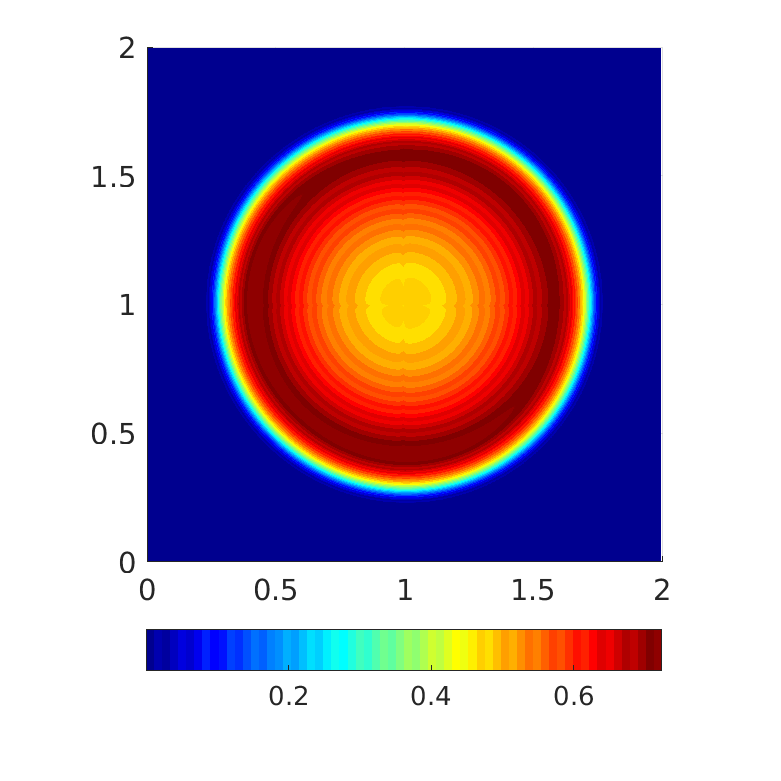}\hspace*{-2em}}
\subfigure[$t=20M$]{
\includegraphics[trim={3em 0 0 0}, clip, width=0.26\textwidth]{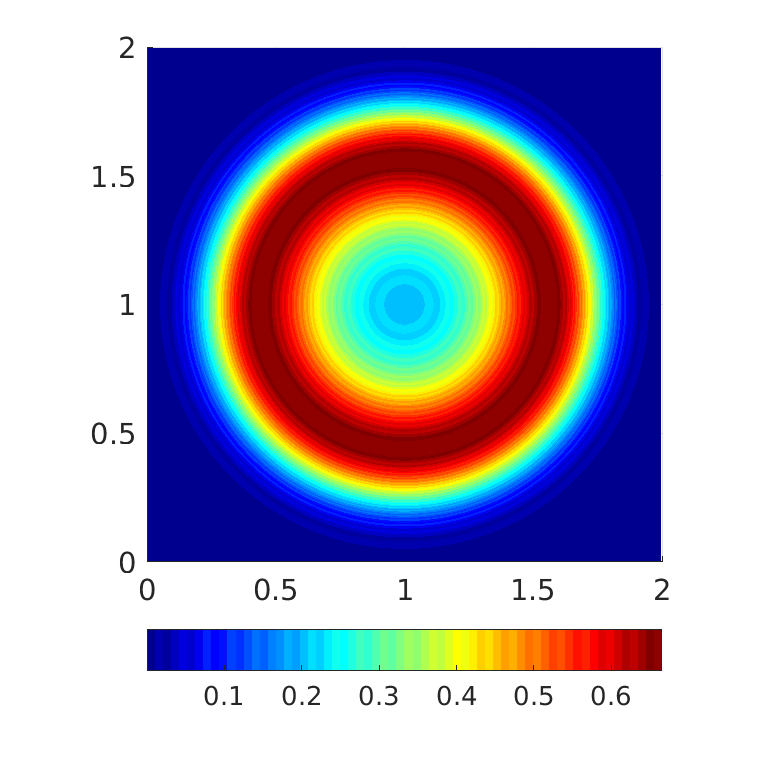}\hspace*{-2em}
\includegraphics[trim={3em 0 0 0}, clip, width=0.26\textwidth]{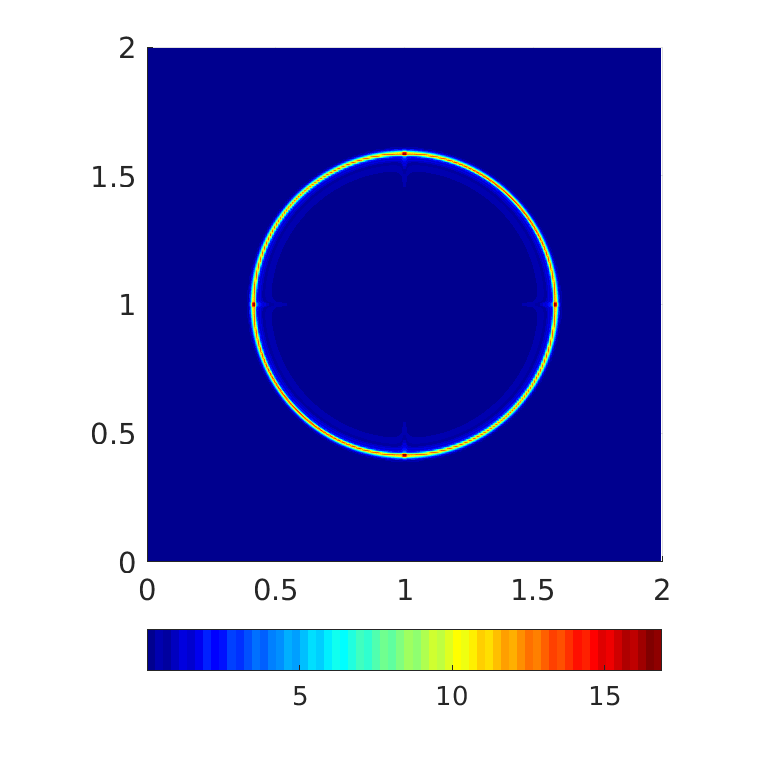}\hspace*{-2em}
\includegraphics[trim={3em 0 0 0}, clip, width=0.26\textwidth]{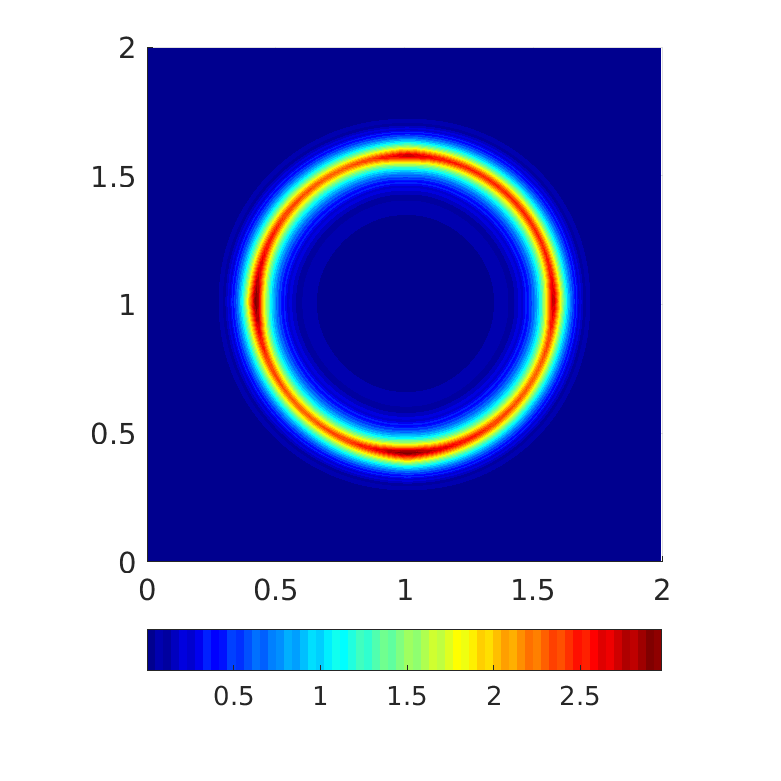}\hspace*{-2em}
\includegraphics[trim={3em 0 0 0}, clip, width=0.26\textwidth]{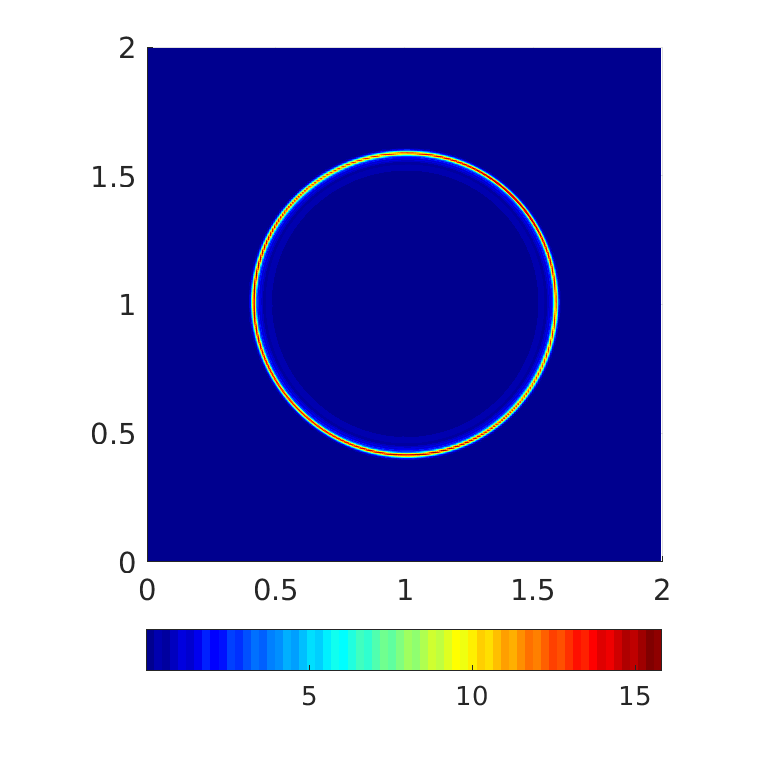}\hspace*{-2em}}
\caption{Comparison of all four schemes with $W(\bx)=0.25|\bx|^4-0.5|\bx|^2$ at two different times.  One blob centered at $(1,1)$ as initial data. From left to right: 1st LxF, 1st upw, 2nd LxF, 2nd upw. The normalization factor is $M=0.3137$.}\label{fig:attrep2D1}
\end{figure}

We now turn to a more singular potential, the 2D version of the first potential in Section \ref{sec:attrep1d}, $W(\bx)=\hf |\bx|^2-\log |\bx|/\sqrt{2\pi}$. In this case we do not have any proof of the convergence properties of the scheme \eqref{eq:laxfr2D}, but we know from \cite{FHK2011,BLL,CDM} that the steady state exact solution of \eqref{eq:agg_eq} is the characteristic of the unit disk with height $\sqrt{2/\pi}$. The simulations can be found in Figures \ref{fig:attrep2D2}(a)--(b). All four schemes act similarly to their 1D counterpart, see Figure \ref{fig:attrep1d}.

\begin{figure}
\centering
\subfigure[$t=4M$]{
\includegraphics[trim={3em 0 0 0}, clip, width=0.26\textwidth]{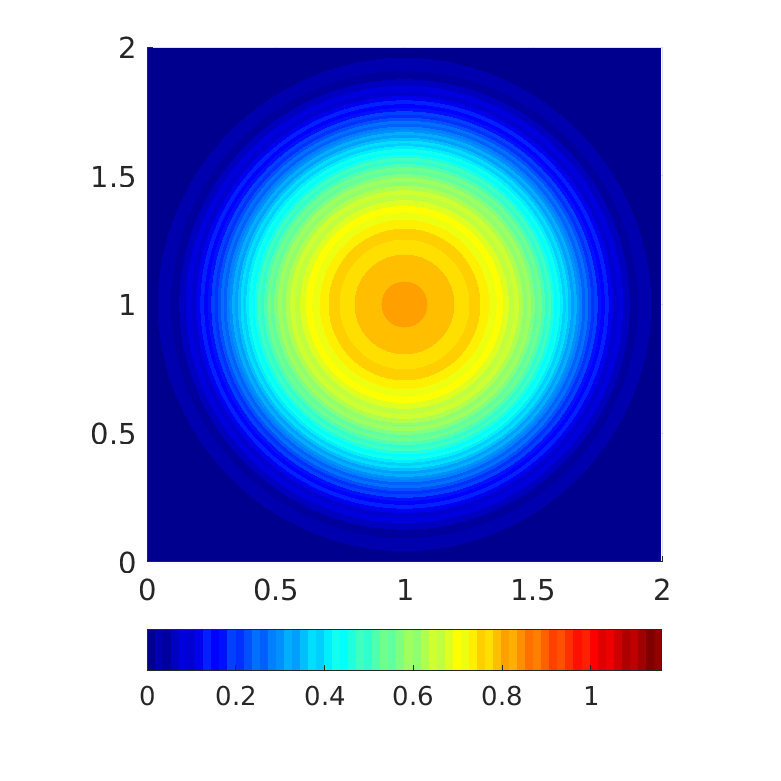}\hspace*{-2em}
\includegraphics[trim={3em 0 0 0}, clip, width=0.26\textwidth]{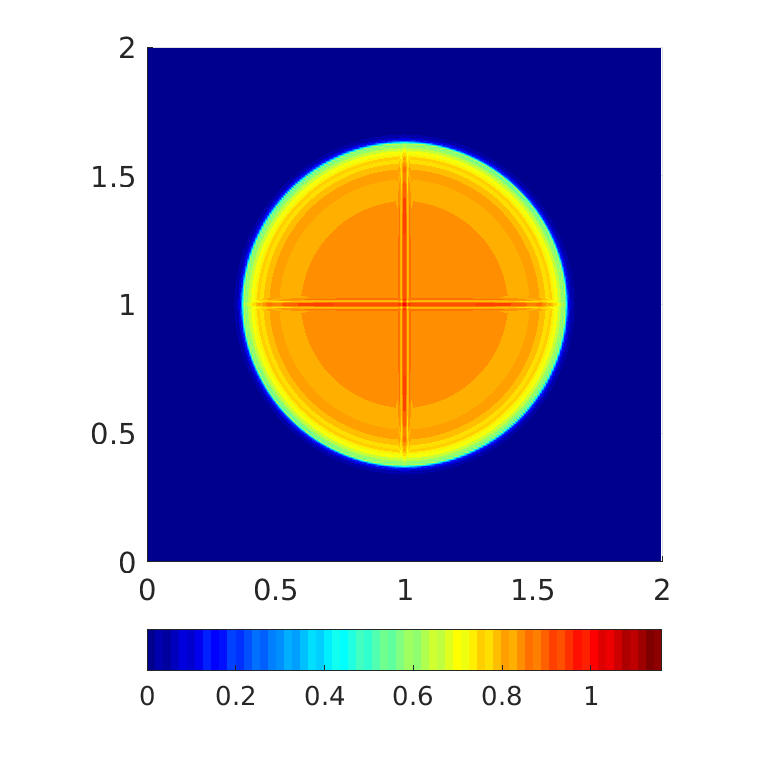}\hspace*{-2em}
\includegraphics[trim={3em 0 0 0}, clip, width=0.26\textwidth]{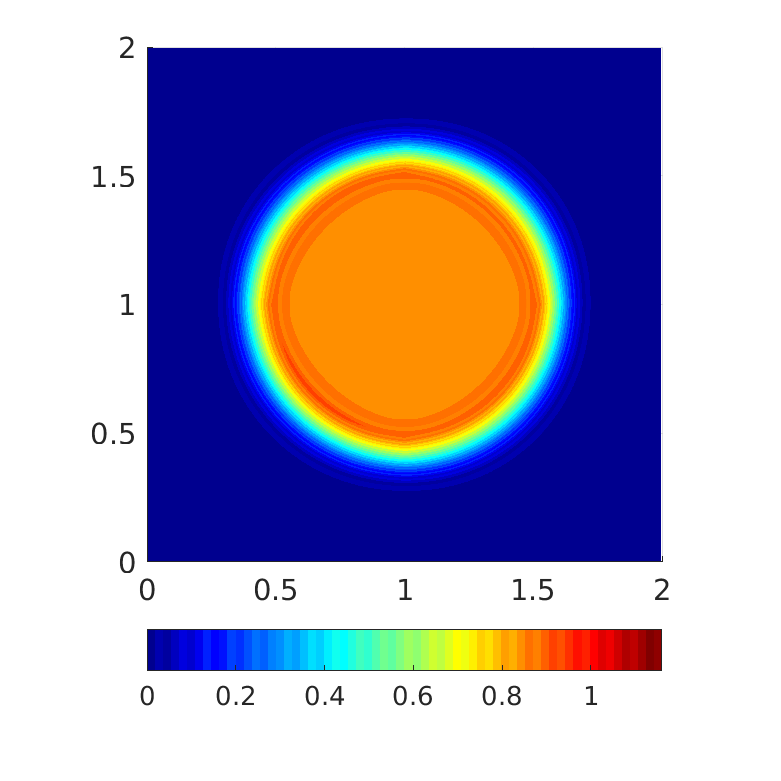}\hspace*{-2em}
\includegraphics[trim={3em 0 0 0}, clip, width=0.26\textwidth]{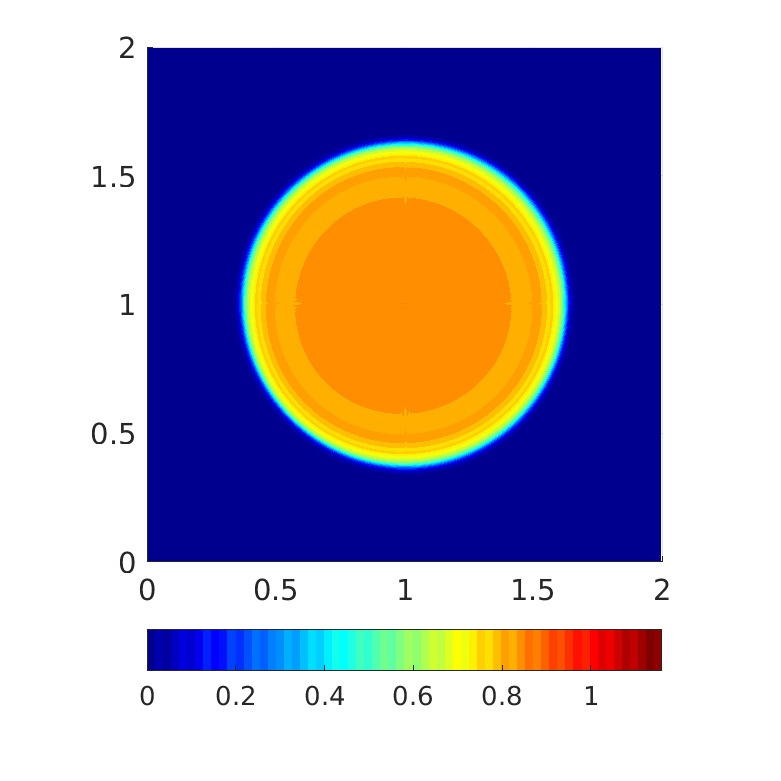}\hspace*{-2em}}
\vspace*{-0.5em}\subfigure[$t=10M$]{
\includegraphics[trim={3em 0 0 0}, clip, width=0.26\textwidth]{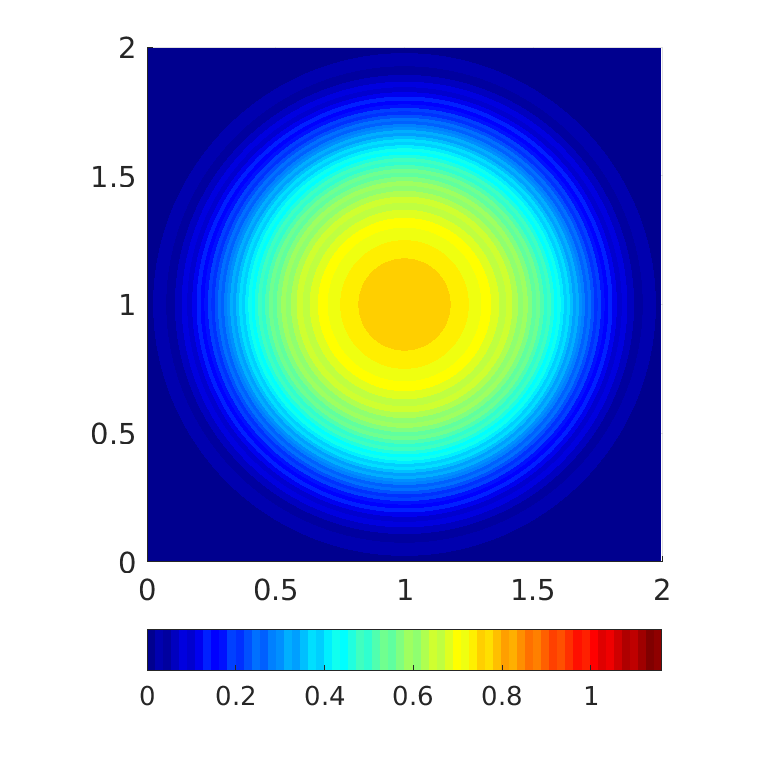}\hspace*{-2em}
\includegraphics[trim={3em 0 0 0}, clip, width=0.26\textwidth]{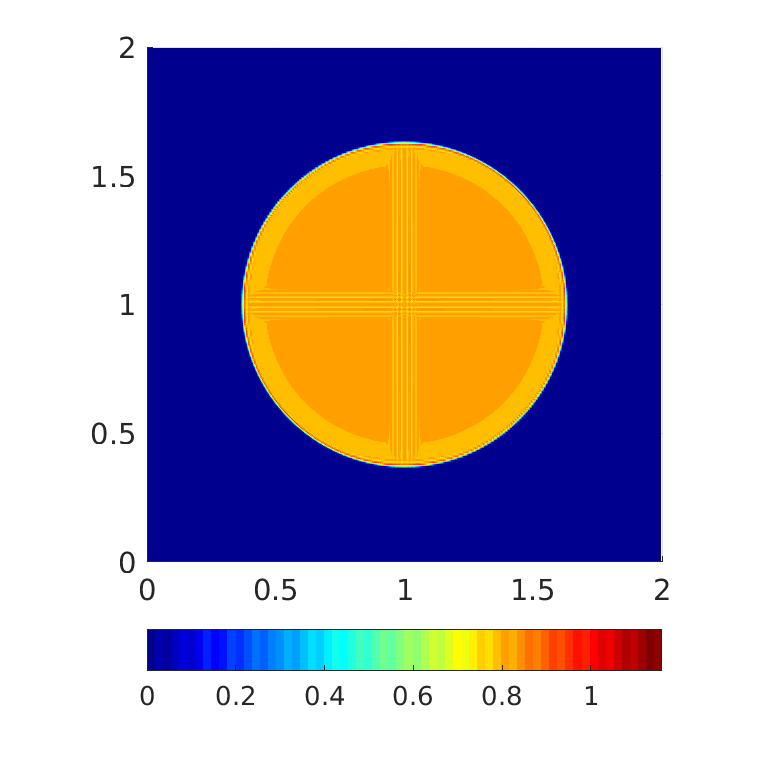}\hspace*{-2em}
\includegraphics[trim={3em 0 0 0}, clip, width=0.26\textwidth]{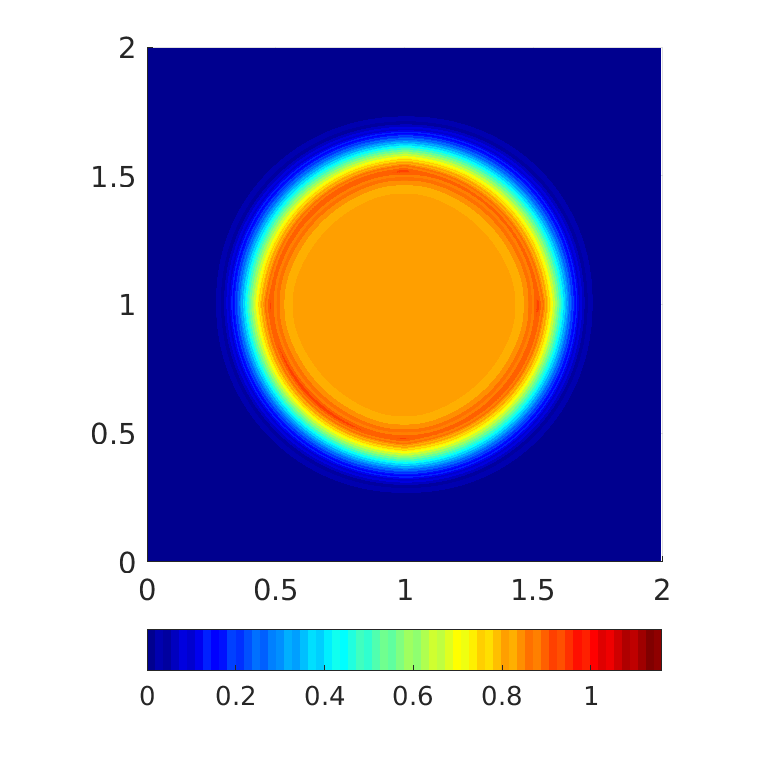}\hspace*{-2em}
\includegraphics[trim={3em 0 0 0}, clip, width=0.26\textwidth]{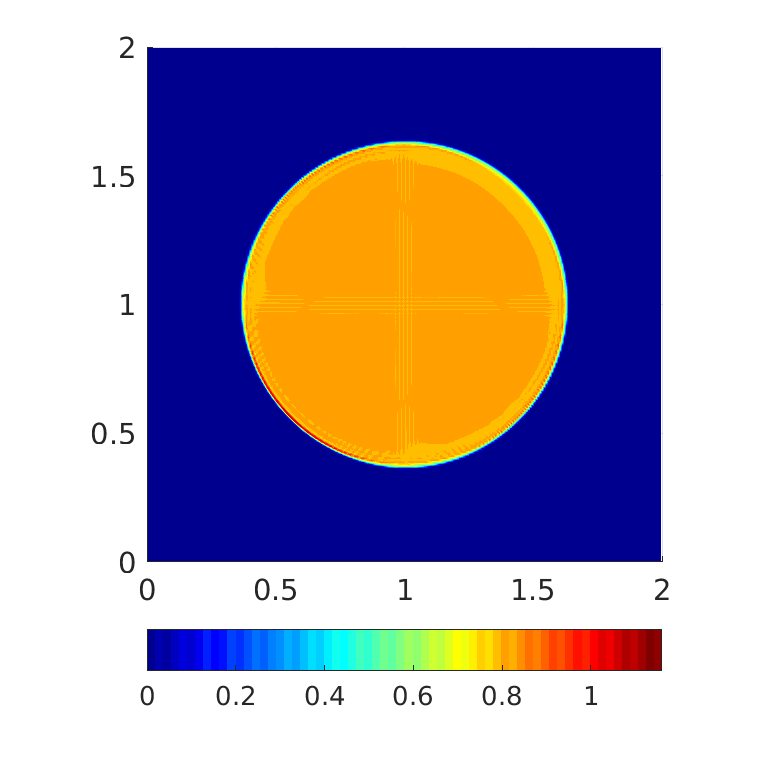}\hspace*{-2em}}
\caption{Comparison of all four schemes with $W(\bx)=\hf|\bx|^2-\log(|\bx|)/\sqrt{2\pi}$ at two different times. One blob centered at $(1,1)$ as initial data. From left to right: 1st LxF, 1st upw, 2nd LxF, 2nd upw. The normalization factor is $M=0.3137$}\label{fig:attrep2D2}
\end{figure}

\section{Conclusions and open questions}
We have developed a (formally) second-order accurate scheme for the aggregation equation \eqref{eq:agg_eq} that is shown to have a uniformly convergent subsequence in the Monge--Kantorovich distance $d_1$ to a distributional solution of \eqref{eq:agg_eq} under the assumptions \ref{cond:lip}--\ref{cond:c1}. Under the additional assumption \ref{cond:convex} the limit is shown to be the unique gradient flow solution of \eqref{eq:agg_eq}. (See Theorems \ref{thrm:main} and \ref{thrm:conv2d} for the exact statements.)

Numerical examples have been provided to demonstrate that the scheme can indeed obtain a second-order convergence rate when the solution is smooth enough and to show that it resolves the solution more sharply than the corresponding first-order schemes. Examples showing that the scheme also handles attractive-repulsive potentials, for which the convergence of the scheme is unknown and not covered by the theory, are provided. An overall good qualitative behavior is observed albeit with minor overshoots and oscillations which are typically present in other finite volume schemes \cite{CCH15} or variational schemes \cite{CRW16} due to the singularity of the asymptotic behavior of the solutions for these specific cases.

Finding a rate of convergence for our scheme is currently out of reach. Due to the reconstruction procedure utilized in the scheme presented here, it is not covered by the convergence rate results in \cite{DLV2016, DLV2017}. A proof of a rate of convergence would be highly desirable, but given the immense difficulty in proving high-order (higher than $\Dx^\hf$) convergence rates for numerical methods for hyperbolic conservation laws \eqref{eq:burgers}, this is expected to be very challenging.

\section*{Acknowledgments}
\small{JAC was partially supported by the EPSRC grant number EP/P031587/1 and the Advanced Grant Nonlocal-CPD (Nonlocal PDEs for Complex Particle Dynamics: Phase Transitions, Patterns and Synchronization) of the European Research Council Executive Agency (ERC) under the European Union’s Horizon 2020 research and innovation programme (grant agreement No. 883363). Parts of this research was conducted at the Institut Mittag--Leffler during the fall of 2016, and at SAM, ETH Z\"urich during the spring of 2017, and the authors would like to thank both institutions for their warm hospitality.}

\bibliography{references}
\bibliographystyle{abbrv}

\end{document}